\documentclass[11pt, oneside]{amsart}   	
\usepackage{geometry}                		
\usepackage{graphicx}				
\usepackage{amssymb}
\usepackage{xcolor}
\usepackage{pgfplots,tikz}
\usetikzlibrary{patterns, shapes.geometric}
\definecolor{imayou}{RGB}{154, 154, 235}
\definecolor{usuai}{RGB}{9, 150, 126}
\definecolor{persred}{RGB}{154,63,63}
\definecolor{sand}{RGB}{201, 177, 60}
\pgfplotsset{compat=1.18}
\usepackage{hyperref}
\usepackage{amsmath}
\usepackage{amsthm}
\usepackage{mathrsfs}
\usepackage{mathtools}
\usepackage[title]{appendix}
\mathtoolsset{showonlyrefs=true}
\usepackage{xcolor}
\usepackage{bbm}
\usepackage{enumerate, comment}
\title[Decay via non-polynomial derivative bound condition]{Sharp energy decay rates for the damped wave equation on the torus via non-polynomial derivative bound conditions}
\author{Perry Kleinhenz}

\date{}							

\theoremstyle{definition}
\newtheorem{definition}{Definition}
\newtheorem{example}{Example}[subsection]

\newtheorem{assumption}{Assumption}
\newtheorem{remark}{Remark}
\theoremstyle{theorem}
\newtheorem{theorem}{Theorem}[section]
\newtheorem{lemma}[theorem]{Lemma}

\newtheorem{proposition}[theorem]{Proposition}
\numberwithin{equation}{section}
\newcommand{\Qb}{\mathbb{Q}}
\newcommand{\Rb}{\mathbb{R}}
\newcommand{\Zb}{\mathbb{Z}}

\newcommand{\Dc}{\mathcal{D}}

\newcommand{\F}{\mathcal{F}}

\newcommand{\Nb}{\mathbb{N}}
\newcommand{\Tb}{\mathbb{T}}

\newcommand{\essinf}{\mathop {\rm ess\,inf}}

\newcommand{\ra}{\rightarrow}

\newcommand{\<}{\left\langle}
\renewcommand{\>}{\right\rangle}
\newcommand{\e}{\varepsilon}
\renewcommand{\d}{\delta}

\newcommand{\limn}{\lim_{n \ra \infty}}

\newcommand{\nm}[1]{\left\| #1 \right\|}

\newcommand{\lp}[2]{ \nm{#1}_{L^{#2}}}

\newcommand{\hp}[2]{\nm{#1}_{H^{#2}}}

\newcommand{\ltwo}[1]{\lp{#1}{2}}

\newcommand{\ltwom}[1]{\nm{#1}_{L^2(M)}}

\newcommand{\Cs}{C^{\infty}_0}

\newcommand{\p}{\partial}

\newcommand{\Ci}{C^{\infty}}

\newcommand{\supp}{\text{supp }}
\newcommand{\T}{\mathbb{T}}
\newcommand{\Ac}{\mathcal{A}}
\newcommand{\Lc}{\mathcal{L}}

\newcommand{\Hc}{\mathcal{H}}

\newcommand{\ti}{\widetilde}

\newcommand{\Rc}{\mathcal{R}}

\newcommand{\Op}{\text{Op}}

\newcommand{\Sb}{\mathbb{S}}

\renewcommand{\Re}{\operatorname{Re }}

\newcommand{\abs}[1]{|#1|}

\newcommand{\limh}{\lim_{h \ra 0}}
\newcommand{\tiroh}{\ti{R}_{1}^{-1}(h)}
\newcommand{\ch}{\chi_h}

\newcommand{\Vidz}{\digamma(\d z)}
\newcommand{\Vidg}{U(g)}
\newcommand{\Vidgh}{U(g(h))}

\newcommand{\Udg}{U(g(h))}
\newcommand{\Qc}{\mathcal{Q}}
\newcommand{\vht}{v_h^{(2)}}
\newcommand{\vhe}{v_{h,E}}
\newcommand{\bvhe}{\overline{v}_{h,E}}
\newcommand{\hvhe}{\hat{v}_{h,E}(0)}
\newcommand{\rhe}{r_{h,E}}
\renewcommand{\dh}{\rho(h)}
\newcommand{\tix}{X}
\newcommand{\tiy}{Y}
\newcommand{\dist}{\operatorname{dist}}
\newcommand{\Pc}{\mathcal{P}}
\newcommand{\g}{\mathfrak{g}}
\newcommand{\ltwomo}[1]{\nm{#1}_{L^2(M_1 \times B(0,1))}}

\newcommand{\ed}[1]{{#1}}
\newcommand{\orange}[1]{{#1}}

\def\XXint#1#2#3{{\setbox0=\hbox{$#1{#2#3}{\int}$ }
\vcenter{\hbox{$#2#3$ }}\kern-.6\wd0}}

\begin{document}
\maketitle
\begin{abstract}
For the damped wave equation on the torus, when some geodesics never meet the positive set of the damping, energy decay rates are known to depend on derivative bounds and growth properties of the damping near the boundary of its support, as well as the geometry of the support of the damping. In this paper we obtain, sometimes sharp, energy decay rates for damping which satisfy more general non-polynomial derivative bounds and growth properties. We also show how these rates can depend on the geometry of the support of the damping. We prove general results and apply them to examples of damping growing exponentially slowly, polynomial-logarithmically, or logarithmically. The decay rates found for these examples interpolate between known sharp rates for purely polynomial damping. The proof relies on estimating the solution at very fine, non-polynomial, semiclassical scales to obtain resolvent estimates, which are then converted to energy decay rates via semigroup theory. 
\end{abstract}

\tableofcontents
\section{Introduction}

\ed{For a Riemannian manifold $(M,g)$}, let $W \in L^{\infty}(\ed{M})$ \ed{be a nonnegative function. Then} consider the damped wave equation 
\begin{equation}\label{DWE}
\begin{cases}
\ed{(}\p_t^2 - \Delta_g + W \p_t) u =0, \quad (z,t) \in \ed{M}\times\Rb \\
(u, \p_t u)|_{t=0} = (u_0, u_1) \in H^1(\ed{M}) \times L^2(\ed{M}).
\end{cases}
\end{equation}
The quantity of particular interest is the energy
\begin{equation}
E(u,t) = \frac{1}{2} \int_{M} |\nabla u(z,t) |^2 + |\p_t u(z,t) |^2 dz.
\end{equation}
It is straightforward to compute $\p_t E(u,t) = -\int_{M} W |\p_t u|^2 dz \leq 0$, so the energy is non-increasing. It is natural then to ask for a rate $r(t) \ra 0$, as $t \ra \infty$, such that
\begin{equation}
E(u,t) \leq r(t) E(u,0), \quad \text{ for all initial data } (u_0, u_1) \in H^1 \times L^2.
\end{equation}
When $W \in C^0(\ed{M})$, this holds with an exponential rate $r(t)=Ce^{-ct}$, if and only if the damping satisfies the geometric control condition (GCC)  \cite{RauchTaylor1975, Ralston1969}, that is every geodesic eventually intersects $\{W>0\}$. 

\ed{In this paper we will mostly focus on the case $M=\Tb^2$. Using $\T^2 \simeq (\Rb/2\pi \Zb)^2$, we will write a point on $\T^2$ as either $(x,y)$ or $z$. On the torus}, when the GCC is not satisfied, there is still energy decay, but it is polynomial. In particular, if $\{W>0\}$ is open and nonempty then 
\begin{equation}\label{eq:stable}
E(u,t)^{\frac{1}{2}} \leq C r(t) \left( \hp{u_0}{2} + \hp{u_1}{1} \right), \quad \text{ for all initial data } (u_0, u_1) \in H^2 \times H^1 \ed{,}
\end{equation}
holds with $r(t)=t^{-\frac{1}{2}}$ \cite[Theorem 2.3]{AL14}. On the other hand, when $\overline{\{W>0\}}$ does not satisfy the GCC there are solutions decaying no faster than $r(t)=t^{-1}$ \cite[Theorem 2.5]{AL14}.
Making additional assumptions on $W$ refines this rate. If for some $C, \e>0,$ $W$ satisfies the derivative bound condition
\begin{equation}\label{DBC}
|\nabla W|\leq C W^{1-\e},
\end{equation}
then \eqref{eq:stable} holds with $r(t)= t^{-\frac{1}{1+4\e}}$ \cite[Theorem 2.6]{AL14}. 


The derivative bound condition \eqref{DBC} of \cite{AL14} is unable \ed{to} distinguish different values of $\alpha>0$ in $W(x,y)=\exp(-(|x|-\sigma)_+^{-\alpha})$. In particular, regardless of $\alpha$, \eqref{DBC} holds for any $\e>0$, and \cite{AL14} provides energy decay at $t^{-1+\d}$, for any $\d>0$. We improve this rate to only a $\ln(\ed{2}+t)$ loss from $t^{-1}$, where the size of the loss depends on $\alpha$. 
\begin{theorem}
If $W(x,y) = \exp(-(|x|-\sigma)_+^{-\alpha})$ for any $\alpha>0$, then \eqref{eq:stable} holds with rate 
\begin{equation}
r(t)=t^{-1} \ln(\ed{2}+t)^{\frac{2\alpha+1}{\alpha}}.
\end{equation}
\end{theorem}
\noindent For a more detailed and general statement, see Theorem \ref{wideresolvethm}. For the case $\sigma=0$, see Theorems \ref{thinresolvethm} and \ref{thinqmthm}.

Despite our focus on non-polynomial derivative bound conditions, we also improve the energy decay rate of $r(t)=t^{-\frac{1}{1+4\e}}$ from \cite[Theorem 2.6]{AL14} for $W$ satisfying the derivative bound condition \eqref{DBC}, at the cost of an additional derivative bound condition assumption. 
\begin{theorem}\label{prelimgendbc}
If $W \in W^{9,\infty}(\T^2)$ and for some $\e \in (0,1/4], C>0$, we have $|\nabla W| \leq CW^{1-\e}, |\nabla^2 W| \leq CW^{\frac{1}{2}}$, then \eqref{eq:stable} holds with rate 
\begin{equation}
r(t)=t^{-\frac{1+\e}{1+2\e}}.
\end{equation}
\end{theorem}
\noindent For a more general statement see Theorem \ref{widegenresolvethm}.

Sharp energy decay rates have also been obtained for $y$-invariant polynomially growing damping. If $W(x,y)=(|x|-\sigma)_+^{\beta}$, $\beta>-1$, then \eqref{eq:stable} holds with $r(t)=t^{-\frac{\beta+2}{\beta+3}}$, and there are solutions decaying no faster than this rate \cite{Kleinhenz2019, DatchevKleinhenz2020, KleinhenzWang2026}. When $W$ has the same growth, but is supported on a strictly convex set with positive curvature, the sharp energy decay rate is faster, $\beta$ is replaced by $\beta+\frac{1}{2}$ \cite{Sun23}, that is \eqref{eq:stable} holds with $r(t)= t^{-\frac{\beta+\frac{1}{2}+2}{\beta+\frac{1}{2}+3}}$.
When the damping is assumed to be supported on a strictly convex set, we generalize the improvement \cite{Sun23} made for polynomial damping, to apply to poly-log damping.
\begin{theorem}
Fix $\beta \geq 9$, if for $z$ near $\p B(0,1)$, $W(z) = (1-|z|)_+^{\beta} \ln((1-|z|)_+^{-1})^{-\gamma}$ then \eqref{eq:stable} holds with rate
	\begin{equation}
	r(t) = t^{-\frac{\beta+\frac{1}{2}+2}{\beta+\frac{1}{2}+3}} \ln(\ed{2}+t)^{\frac{\gamma}{\beta+\frac{1}{2}+3}}.
	\end{equation}
\end{theorem}
\noindent
For a more detailed and general statement of this result, see Theorem \ref{polylogconveximprovethm}. For a further adjustment that provides improvements of any size between $0$ and $\frac{1}{2}$, see Theorem \ref{ellipsethm}.

Now we will provide some additional comments on the approach used to prove our results. To begin, the common strategy in \cite{AL14, Kleinhenz2019, DatchevKleinhenz2020, KleinhenzWang2026, Sun23}  is to obtain resolvent estimates for the stationary equation, then convert these to energy decay rates using semigroup theory \cite{BorichevTomilov2010}. These resolvent estimates are obtained by using a cutoff $\ch = \chi(h^{-\d} W)$ to separately estimate $u$ on $\{W \leq h^{\d}\}$ and $\{W \geq h^{\d}\}$. The constant $\d>0$ controls the semiclassical scale and is chosen to balance the size of $u$ on the two regions. Using a pairing argument, it is straightforward to estimate $u$ on $\{W \geq h^{\d}\}$, so the main challenge is to estimate $u$ on $\{W \leq h^{\d}\}$. This is where the derivative bound condition \eqref{DBC}, growth properties of $W$, and the geometry of the support of $W$ are used. 

In this paper, we take $\ch = \chi(g(h)W)$ for $g(h)$ non-polynomial and consider non-polynomial  derivative bound conditions $|\nabla W| \leq W q(W)^{-1}$. We choose our semiclassical scale $g(h)$ to balance contributions from $u$ on $\{W \leq g(h)\}$ and $\{W \geq g(h)\}$. Two of the major technical novelties of this work are: generalizing the approach on $\{W \leq g(h)\}$ for non-polynomial $W$ and $g$, and balancing the non-polynomial sizes of $u$ on the two regions. Ultimately, this produces non-polynomial resolvent estimates, which we convert to non-polynomial energy decay rates via a refinement to the semigroup theory \cite{rss19}.


\subsection{$y$-invariant damping with general derivative bound conditions}
To state the results of this paper for $y$-invariant damping in their full generality, we begin by noting that the damped wave equation can be restated as a first order equation
\begin{equation}
\p_t \begin{pmatrix} u \\ \p_t u \end{pmatrix} = \Ac \begin{pmatrix} u \\ \p_t u \end{pmatrix}, \qquad \Ac = \begin{pmatrix} 0 & Id \\ \Delta & -W \end{pmatrix}.
\end{equation}
Then $\Ac$ generates a strongly continuous semigroup on $\Hc = H^1 \times L^2$ and energy decay rates can be obtained by proving estimates on the resolvent $(i \lambda -\Ac)^{-1}$. 
We now state some definitions and assumptions necessary for the statement of the theorem. 
\begin{assumption}\label{wideassumption}
Assume $W(x,y) = W(x) \geq 0$ and there exists $V(x) \in C^1(\Sb^1)$ with $V(x)\geq 0$ and $C\geq 1,$ such that 
\begin{enumerate}
	\item $\frac{1}{C}  V(x) \leq W(x) \leq C V(x),$
	\item There exists an increasing and continuous function $q$, such that $q(0)=0$, and 
	\begin{equation}
	|\p_x V| \leq \frac{V}{q(V)}, \quad \text{ near } V=0.
	\end{equation}
\end{enumerate}
\end{assumption}
This assumption is a generalization of the derivative bound condition $|\p_x W| \leq C W^{1-\e}$, as can be seen by taking $q(z)=z^{\e}$.

\begin{assumption}\label{widebadassumption}
Assume that $V$ from Assumption \ref{wideassumption} also satisfies $V \in C^2(\Sb^1)$ and there exists an increasing and continuous function $p,$ such that $p(0)=0$, and 
\begin{equation}
|\p_x^2 V| \leq \frac{V}{p(V)},  \quad \text{ near } V=0.
\end{equation}
\end{assumption}
This assumption is a generalization of the derivative bound condition $|\p_x^2 W| \leq W^{1-2\e}$, as can be seen by taking $p(z)=z^{2\e}$. Note that we frequently have $q(z)^2=p(z)$, but there are $V$ for which this is not true, such as $V(x)=\ln((|x|-\sigma)_+^{-1})^{-\gamma}$.

A standard definition we need to apply the semigroup theory of \cite{rss19} is that of a function of positive increase. 
\begin{definition}\label{posincdef}
	A function $M: [a, \infty) \ra (0,\infty)$ has positive increase if there exists $\alpha>0, c \in (0,1), s_0 \geq a$ such that 
	\begin{equation}
		\frac{M(\lambda s)}{M(s)} \geq c \lambda^{\alpha}, \quad \lambda \geq 1, s \geq s_0.
	\end{equation}
\end{definition}

It will be necessary for us to invert quantities involving \ed{the non-polynomial} $q$ and $p$. When $q$ and $p$ are polynomial it is straightforward to take inverses, even after multiplying by constants or absorbing lower order error terms. To allow ourselves the same flexibility for non-polynomial $q$ and $p$ we make the following definition. 
\begin{definition}\label{envinvdef}
For a continuous and strictly increasing function $f: (0, a) \ra (0,b)$ we define an envelope inverse $\ti{f}^{-1}:(0,b) \ra (0,a)$ to be a continuous function, such that there exists $C,h_0 > 0,$ and for all $h \in (0,h_0)$
\begin{equation}
\frac{h}{C} \leq f \circ \ti{f}^{-1}(h) \leq C h.
\end{equation}
\ed{We make a similar definition for $g:(a,\infty) \ra (b,\infty)$ continuous and strictly increasing. We define an envelope inverse $\ti{g}^{-1}:(b,\infty) \ra (a,\infty)$ to be a continuous function such that there exist $C,t_0>0$, and for all $t>t_0$
	\begin{equation}
		\frac{t}{C} \leq g \circ \ti{g}^{-1}(t) \leq Ct.
	\end{equation}}
\end{definition}

We now make a definition and technical assumption related to the size of sets where $V$ is small. This provides information about the growth behavior of $V$ and can be used to improve the energy decay rate. 
\begin{definition}\label{vidzdef}
\ed{Let $\nu$ be the Lebesgue measure on $\Sb^1$.} Define 
\begin{equation}
\digamma(\zeta) = \ed{\nu}\left(\{x \in \Sb^1; 0<V(x) \leq \zeta \}\right).
\end{equation}
Fix $\d>1$ and let $U(z)$ be a continuous increasing function, such that 
\begin{equation}
	\Vidz \leq U(z)< \ed{\nu((0,\max V))}.
\end{equation}
When applying our result we only take $U(z)=\ed{C}$ or $U(z)=\Vidz$, but this more flexible definition is permitted. \ed{The size of $\d>1$ is usually not important, although we must always take $\d>1$. In particular, $\d$ can be neglected when $\digamma(C h) \simeq \digamma(h)$ as $h \ra 0$.}
\end{definition}
\begin{assumption}\label{wideinversebassumption}
There exists $C, z_0>0$ such that $U(z) \geq C q(z)$ for $z \leq z_0$. 
\end{assumption}

We now define quantities in terms of $q, p$ and $U$ which specify the energy decay rate. 

\begin{definition}\label{RMdefs}

\begin{equation}
R_{1}(z) = z q(z) U(z), \quad R_2(z) = z \min(q^2(z), p(z)).
\end{equation}
\begin{equation}
M_j(h) = \frac{\ti{R}_j^{-1}(h) U(\ti{R}_j^{-1}(h))}{h^2}.
\end{equation}
We also require for technical reasons, that there exist $\e>0$, such that $M_1(h) \geq h^{-1-\e}$, and that $M_2(h) \geq h^{-1}$.
\end{definition}
We can now state our result, which gives energy decay rates in terms of the $q$ and $p$ from the derivative bound conditions, as well as improvements due to the growth behavior of $V$.
\begin{theorem}\label{wideresolvethm}
\begin{enumerate}Suppose $W \in L^{\infty}(\T^2)$. 
\item If $W$ satisfies Assumption \ref{wideassumption}, then the following holds with $j=1$.
\item If $W$ satisfies Assumptions \ref{wideassumption}, \ref{widebadassumption}, and \ref{wideinversebassumption}, then the following holds with $j=2$.
\end{enumerate}
There exists $C, h_0>0$, such that for all $h \in (0,h_0)$ 
\begin{equation}
\nm{\left( \frac{i}{h} - \Ac \right)^{-1} }_{\Lc(\Hc)} \leq C M_{j}(h).
\end{equation}
Furthermore, if $m_{j}(\lambda):=M_j\left(\frac{1}{\lambda}\right)$ has positive increase, then there exists $C>0$ such that for all $u$ solving \eqref{DWE}
$$
E(u,t)^{\frac{1}{2}} \leq \frac{C}{\ed{\widetilde{m}_{j}^{-1}(t)}} \left( \hp{u_0}{2} + \hp{u_1}{1}\right).
$$
\end{theorem}
\begin{remark}
\begin{enumerate}
	\item Note that replacing the damping $W$ with the envelope function $V$ in the derivative bound condition of Assumption \ref{wideassumption} is a meaningful generalization. If $W(x) = \left(\sin\left( \frac{1}{|x|-\sigma} \right) + \ed{2} \right) (|x|-\sigma)_+^{\beta}$ then $V=(|x|-\sigma)_+^{\beta}$, satisfies the derivative bound condition with $q(z)=z^{1/\beta}$ and produces energy decay at rate $t^{-\frac{\beta+2}{\beta+3}}$. However,  $W$ only satisfies the derivative bound condition with $q(z)=z^{\frac{1}{2}}$, when $\beta \geq 2$, which would produce a slower decay rate, and when $\beta<2$ it satisfies no derivative bound condition and working with $W$ directly would not give an energy decay rate. 
	\item It is always possible to choose $U(z)=1$, but choosing a smaller $U(z)$ will improve the decay rate. When $\digamma(\d z)=q(z)$, cases 1 and 2 provide the same energy decay rate. However we do not always have this, see Example \ref{polyexp1dex}.
	\item There are fewer assumptions to check in case 1, but the strength of the resolvent estimate is limited, so energy decay cannot be faster than $t^{-1+\d}$ for some $\d>0$. In both cases, due to our lower bounds on $M_j$, we cannot have decay faster than $t^{-1}$, but this is already forbidden for $V$ vanishing on a strip by \cite[Theorem 2.5]{AL14}.
\end{enumerate}
\end{remark}

Note that we are only interested in the top order asymptotic behavior of $M_j(h)$ for $h$ small, and so we typically will not track the size of constants multiplying it or its constituent pieces, as well as lower order terms that can be absorbed into the top order term by adjusting the size of a constant. \ed{As a piece of notation,} we will write $f(h) \simeq g(h)$ when there exists $C,h_0>0$ such that for $h \in (0,h_0)$ 
\begin{equation}
	C^{-1} f(h) \leq g(h) \leq C f(h).
\end{equation}
\ed{Similarly, we will write $f(t) \simeq g(t)$ when there exists $C,t_0>0$ such that for $t \in (t_0,\infty)$ 
\begin{equation}
	C^{-1} f(t) \leq g(t) \leq C f(t).
\end{equation}}

\subsection{Damping depending on $x$ and $y$ with general derivative bound conditions}\label{sec:xydamp}

Before stating the next result, we must introduce \ed{a technical definition and another version of} our generalized derivative bound conditions. 

\begin{definition}\label{genregassumption}
	We write $W \in \mathcal{D}^{k,\e}(\T^2)$, if $W \in W^{k,\infty}(\T^2)$, and there exists $C>0$, such that 
	\begin{equation}\label{ddbc}
		|\nabla W| \leq C W^{1-\e}, \qquad |\nabla^2 W| \leq CW^{1-2\e}.
	\end{equation}
\end{definition}
Note that for $\e_1 <\e_2$ we have $\mathcal{D}^{k,\e_1} \subset \mathcal{D}^{k, \e_2}$. Note also that \eqref{ddbc} is only non-trivial near where $W=0$. In general we will ask that $W \in \mathcal{D}^{9,\frac{1}{4}}(\T^2)$.

\begin{assumption}\label{noninvarassumption}
	Assume there exist $\e_1>0$ and an increasing, continuous function $q$, such that $q(0)=0$, and 
	\begin{equation}
		|\nabla W| \leq \frac{W}{q(W)}, \text{ when } W \leq \e_1,
	\end{equation}
	and the function $r_1(z) = \frac{z}{q(z)}$ is concave on $[0,\e_1]$.
\end{assumption}
This assumption is a generalization of the derivative bound condition, $|\nabla W| \leq C W^{1-\e}$, as can be seen by taking $q(z)=z^{\e}$.


\begin{assumption}\label{widegenbadassumption}
	For the same $\e_1$ as in Assumption \ref{noninvarassumption}, assume there exists an increasing, continuous function $p,$ such that $p(0)=0$, and 
	\begin{equation}
		|\nabla^2 W| \leq \frac{W}{p(W)} , \text{ when } W \leq \e_1,
	\end{equation}
	and the function $r_2(z) = \frac{z}{p(z)}$ is concave on $[0,\e_1]$.
\end{assumption}
This assumption is a generalization of the derivative bound condition $|\nabla^2 W| \leq C W^{1-2\e}$, as can be seen by taking $p(z)=z^{2\e}$. \ed{As before we frequently take $p(z)=q(z)^2$, but this is not always true.}

We can now state a general resolvent estimate result for damping depending on $x$ and $y$.




%

\begin{theorem}\label{widegenresolvethm}
Assume $W \in \mathcal{D}^{9,\frac{1}{4}}(\T^2)$. Recall the definitions of $M_j$ in Definition \ref{RMdefs}.
\begin{enumerate}
	\item If $W$ satisfies Assumption \ref{noninvarassumption}, then the following holds with $j=1$ and $U(z)=1$.
	\item If $W$ satisfies Assumptions \ref{noninvarassumption} and \ref{widegenbadassumption}, then the following holds with $j=2$ and $U(z)=1$.
\end{enumerate}
\ed{There exists $C, h_0>0$, such that for all $h \in (0,h_0)$ }
$$
\nm{\left( \frac{i}{h} - \Ac\right)^{-1}}_{\Lc(\Hc)} \leq \ed{C} M_j(h).
$$
Furthermore if $m_j(\lambda)=M_j(1/\lambda)$ has positive increase, then there exists $C>0$ such that for all $u$ solving \eqref{DWE}
$$
E(u,t)^{\frac{1}{2}} \leq \frac{C}{\ed{\widetilde{m}_j^{-1}(t)}} \left( \hp{u_0}{2} + \hp{u_1}{1} \right).
$$
\end{theorem}
\begin{remark}
\begin{enumerate}
	\item This result generalizes \cite[Theorem 2.6]{AL14} to allow for non-polynomial derivative bound conditions. Furthermore, this frequently produces improved energy decay rates for damping which satisfy both sets of hypotheses. For example consider damping smooth and vanishing like $\exp(-x^{-\alpha})$ in local coordinates, with no additional assumptions on the geometry of the support. Then \eqref{DBC} is satisfied for any $\e>0$, so \cite[Theorem 2.6]{AL14} provides energy decay at rate $t^{-1+\d}$, for any $\d>0$. On the other hand, such a damping satisfies Assumptions \ref{noninvarassumption} and \ref{widegenbadassumption} with $q(z) =\ln(z^{-1})^{-\frac{\alpha+1}{\alpha}}, p(z)=q(z)^2$ and Theorem \ref{widegenresolvethm} provides energy decay at rate $t^{-1} \ln(\ed{2}+t)^{2\left(\frac{\alpha+1}{\alpha}\right)}$. 
	\item In this result we do not obtain improvements provided by $\Vidz$, as in Theorem \ref{wideresolvethm}. This is because Theorem \ref{widegenresolvethm} relies on averaging the original damping along periodic trajectories $v$ on the torus. In general, it is not straightforward to relate $\Vidz$ to the corresponding quantity after averaging,
	\begin{equation}
		\digamma_v(\d z) = \ed{\nu}( \{x \in \Sb^1: 0<A_v(W) < \d z\})
	\end{equation}
	 where $A_v(W)$ is the average of $W$ along $v$, which is defined more carefully below. However, when the damping has some structure to how it turns on and the geometry of its support, it becomes possible to insert these improvements using Theorem \ref{averagedresolventthm}. This is exactly how Theorems \ref{rectthm}, \ref{polylogconveximprovethm} and \ref{ellipsethm} are proved.
\end{enumerate}
\end{remark}

The resolvent estimates for non-$y$-invariant damping will follow from sufficiently good resolvent estimates for $y$-invariant damping via a normal form argument. To state the underlying result, we first introduce averaging along periodic directions on the torus.

Considering $\T^2= \Rb^2/(2\pi \Zb)^2$, we think of $\Sb^1$ as directions for geodesic trajectories on $\T^2$ and decompose $\Sb^1$ as: rational directions
$$
\Qc:= \left\{\zeta \in \Sb^1; \zeta = \frac{(p,q)}{\sqrt{p^2+q^2}}, (p,q) \in \Zb^2, \gcd(p,q)=1\right\},
$$ 
and irrational directions $\Rc:= \Sb^1 \backslash \Qc$. Put another way, $v \in \Qc$ implies $v=(\xi, \eta)$, where $(\xi, \eta)$ are $\Qb$ linearly dependent. 

For $v \in \Qc$ we define the averaging operator along $v$:
\begin{equation}
f \mapsto A(f)_v(z) = \lim_{T \ra \infty} \frac{1}{T} \int_0^T f(z+tv) dt.
\end{equation}
Since $v \in \Qc$, the orbit $z \mapsto z +tv$ is periodic with some period $T_v$. Because of this, we call such a $v$ a periodic direction. Furthermore, when $f \in C^0(\T^2)$ 
\begin{equation}
A(f)_v(z) = \frac{1}{T_v} \int_0^{T_v} f(z+tv) dt.
\end{equation}
Finally, let 
\begin{equation}
\Ac_v= \begin{pmatrix} 0 & Id \\ \Delta & -A_v(W) \end{pmatrix}.
\end{equation}
Now we can state the averaging result that relates resolvent estimates for $y$-invariant damping and general damping. Roughly, it says that for sufficiently regular damping $W$, if the average of $W$ along every periodic direction produces a resolvent estimate, then $W$ does as well. Note that after averaging along a periodic direction $v$, the resultant function $A_v(W)$ is invariant in the direction $v$. 
\begin{theorem}\label{averagedresolventthm}
Suppose $W \in \Dc^{9,\frac{1}{4}}(\T^2)$. If there exists a function $\dh:(0,1) \ra (0,1]$ with $\dh^{-1} = o(h^{-\frac{1}{3}})$, such that  for all $v \in \Sb^1$ periodic, there exists $h_v, C_v>0,$ such that for $h \in (0,h_v)$ 
\begin{equation}
\nm{\left(\frac{i}{h} -\Ac_v\right)^{-1}}_{\Lc(\Hc)} \leq \frac{C_v}{h \dh},
\end{equation}
then there exists $C, h_0>0$ such that for $h \in (0,h_0)$
\begin{equation}
\nm{\left(\frac{i}{h}-\Ac\right)^{-1}}_{\Lc(\Hc)} \leq \frac{C}{h \dh}.
\end{equation}
The hypotheses can be slightly weakened by re-writing the statement using the stationary equation. If $\ed{\rho(h)^{-1}=o(h^{-\frac{1}{3}})}$, and for all $v \in \Sb^1$ periodic, there exists $h_v, C_v>0$, such that for $h \in (0,h_v)$ and all $\mathcal{E} \in (0, \frac{\dh}{h})$
\begin{equation}\label{assumed1dresolventeq}
\nm{\left(-\p_x^2 + \frac{i}{h}A_v(W)(x) -\mathcal{E} \right)^{-1} }_{\Lc(\ed{L^2(\Sb^1)})} \leq C_v \dh^{-1}, 
\end{equation}
then there exists $C, h_0>0$ such that for $h \in (0,h_0)$ 
\begin{equation}
\nm{\left(-\Delta+ \frac{i}{h} W - \frac{1}{h^2}\right)^{-1}}_{\Lc(\ed{L^2(\T^2)})} \leq C \dh^{-1}. 
\end{equation}
\end{theorem}
\begin{remark}
	\begin{enumerate}
		\item This result is what allows us to convert the energy decay results for $y$-invariant  damping of Theorem \ref{wideresolvethm}, into the non $y$-invariant result of Theorem \ref{widegenresolvethm}.
		\item The hypotheses involving the semigroup generator version are stronger than the hypotheses for the stationary equation version. That is, if $\nm{(\frac{i}{h} - \Ac)^{-1}} \leq \frac{C}{h \dh}$  then $\nm{(-\p_x^2+\frac{i}{h} A_v(W)(x)-\mathcal{E})^{-1}} \leq \frac{C}{\dh}$ for all $\mathcal{E} \in \Rb$, but as can be seen from inspecting the proof of this Theorem, and the proof of Proposition \ref{1dresolveprop}, we only use the stationary equation resolvent estimate for $\mathcal{E} \in (0, \frac{\dh}{h})$.
		\item This result generalizes the averaging approach of \cite{Sun23} to non-polynomial $\rho(h)$ and provides a more portable and general statement of Theorem 1.2 of that paper. 
		\item We also point out that the stated regularity required of $W$ can be relaxed from $W\in W^{10,\infty}$ to $W^{9,\infty}$, but the details of this are already in \cite{Sun23}. See Remark \ref{Gardingrmk} for more details on the required regularity of $W$.
	\end{enumerate}
\end{remark}

\subsection{Damping is $0$ only along geodesics}
We again introduce some necessary definitions, following the conventions of \cite{BurqZuily2015}.

Let $(M,g)$ be a Riemannian manifold of dimension $d$ and consider \eqref{DWE} on it. Define the domain of effective damping 
\begin{equation}
\Omega = \bigcup_{\{U \text{open}: \essinf_U W >0\}} U.
\end{equation}
Write $S^* M$ for the cosphere bundle on $M$ and $S^* \Omega$ for the cosphere bundle on $\Omega$. Then write $\phi(s): S^* M \ra S^*M$ for the geodesic flow and define the set of geometrically controlled points
$$
\mathcal{GC} = \{ (z,\zeta) \in S^*M: \exists s \in \Rb \text{ such that } \phi(s)(z,\zeta) \in S^* \Omega \},
$$
to be the set of all points in the cosphere bundle that reach the effective damping in finite time. When $\mathcal{GC}=S^*M$ energy decay is guaranteed to be exponential, so we will assume this is not the case. Define the trapped set and its projection to $M$
$$
\mathcal{T} = S^*M \backslash \mathcal{GC}, \quad T = \Pi_x \mathcal{T},
$$
to be those points in the cosphere bundle that generate geodesic flows which never encounter the damping. 

\begin{theorem}\label{thinresolvethm}
Assume that 
\begin{enumerate}
	\item There exists a neighborhood $\mathcal{U}$ of $T$ in $M$, a compact Riemannian manifold $(M_1, g_1), g_1 \in L^{\infty}$ of dimension $k$, a Lipschitz metric $g_2$ on the unit ball $B(0,1) \subset \Rb^{d-k}$ and a $W^{2, \infty}$ 
	\ed{isometry} 
	\begin{equation}
	\Psi: (\mathcal{U}, g) \ra (M_1 \times B(0,1), \tilde{g} = g_1 \otimes g_2).
	\end{equation}
	\item There exists an increasing and continuous function $V:[0,1] \ra [0,\infty)$, with $V(0)=0$ and $C,c>0$ such that 
	\begin{equation}
	cV(|x|) \leq W(\Psi^{-1}(p,x)) \leq C V(|x|), \quad \text{ for all } (p,x) \in M_1 \times B(0,1).
	\end{equation}
\end{enumerate}
For any $\e>1$, let $R_{\e}(z) = z^2 \sqrt{V(z)V(\e z)}$. Then, there exists $C>0$ such that 
\begin{equation}
\nm{\left(\frac{i}{h} - \Ac\right)^{-1}}_{\Lc(\Hc)} \leq C M_{\e}(h) := \max\left( \frac{1}{V(\ti{R}_{\e}^{-1}(h))}, \frac{\ti{R}_{\e}^{-1}(h)^2}{h} \right).
\end{equation}
Furthermore, if $m_{\e}(\lambda) = M_{\e}\left(\frac{1}{\lambda}\right)$ is of positive increase, then there exists $C>0$, such that for all solutions of \eqref{DWE}
\begin{equation}
E(u,t) \leq \frac{C}{\ed{\ti{m}_{\e}^{-1}(t)}} \left( \hp{u_0}{2} + \hp{u_1}{1} \right).
\end{equation}
If $m_{\e}(\lambda)$ does not have positive increase, define $m_{\e,\log}(\lambda) = m_{\e}(\lambda)( \ln(1+\lambda) + \ln(1+m_{\e}(\lambda)))$, then there exists $C>0$, such that for all $u$ solving \eqref{DWE} 
\begin{equation}
E(u,t)^{\frac{1}{2}} \leq \frac{C}{\ed{\ti{m}_{\e,\log}^{-1}(t/C)}} \left( \hp{u_0}{2} + \hp{u_1}{1} \right).
\end{equation}
\end{theorem}
\begin{remark}\label{xydampremark}
\begin{enumerate}
	\item As in \cite{BurqZuily2015}, a more straightforward, but weaker, statement would be to assume that $M$ is a product. That is 
\begin{enumerate}
	\item $(M,g)= (M_1 \times M_2, g_1 \otimes g_2), q_0 \in M_2, T= \Sigma = M_1 \times \{q_0\}$,
	\item $cV(d_g(m,\Sigma)) \leq W(m) \leq C V(d_g(m,\Sigma))$ for all $m \in M_1 \times B(q_0, 1)$.
\end{enumerate}
	\item The size of $\e>1$ is usually not important, although we must always take $\e>1$. In particular, $\e$ can be neglected when $V(\e z) \simeq V(z)$ for $z$ near $0$. 
	\item This Theorem can be combined with Theorem \ref{averagedresolventthm}, to obtain energy decay rates for damping on $\T^2$ vanishing along multiple intersecting geodesics. For example if $W \simeq C|x|^{\beta_1} |y|^{\beta_2}$, the energy decays at rate $t^{-\frac{\beta+2}{\beta}}$ where $\beta=\ed{\max}(\beta_1,\beta_2)$. See Example \ref{ex:xybeta} for details of the proof and a further discussion.
\end{enumerate}
\end{remark}
%
We also have a result providing a lower bound on the energy decay rate, which only requires the damping to be bounded from above by $V$. 
\begin{definition}
Let $R(z) =z^2 V(z)$ and define 
\begin{equation}
M(h) = \frac{1}{V(\ti{R}^{-1}(\ed{h}))}, \quad m(\lambda) = M\left( \frac{1}{\lambda}\right).
\end{equation}
\end{definition}
\begin{theorem}\label{thinqmthm} 
Under the same geometric assumptions as in Theorem \ref{thinresolvethm}, if there exists an increasing and continuous function $V:[0,1] \ra [0,\infty)$ with $V(0)=0$ and $C>0,$ such that 
\begin{equation}
W(\Psi^{-1}(p,x)) \leq C V(|x|), \quad \forall (p,x) \in M_1 \times B(0,1),
\end{equation}
then there exist a sequence $h \ra 0,$ and quasimodes $\ed{(u_{0,h}, u_{1,h}) \in H^2 \times H^1}$, such that 
\begin{equation}
\nm{ \left( \frac{i}{h} -\Ac \right) \ed{ \begin{pmatrix} u_{0,h} \\ u_{1,h} \end{pmatrix}}}_{\Lc(\Hc)} \leq \frac{1}{M(h)}, \quad \nm{ \ed{\begin{pmatrix} u_{0,h} \\ u_{1,h} \end{pmatrix} }}_{\Hc} = 1. 
\end{equation}
Therefore 
\begin{equation}
\ed{\limsup_{h \ra 0} \frac{\nm{ \left(\frac{i}{h} - \Ac\right)^{-1}}_{\Lc(\Hc)}}{M(h)} >0.}
\end{equation}
Furthermore, for all $C>0$, there exists a solution of \eqref{DWE} decaying no faster than
\begin{equation}
E(u,t) \leq \frac{C}{\ti{m}^{-1}(t)} \left( \hp{u_0}{2} + \hp{u_1}{1} \right).
\end{equation}
\end{theorem}

\begin{remark}
\begin{enumerate}
	\item Both Theorem \ref{thinresolvethm} and \ref{thinqmthm} allow $W$ to depend on $x$ and $p$, although because $W$ is bounded by multiples of $V$ its dependence on $p$ is not completely free. 
	\item If $V(\e z) \simeq V(z)$ for $z$ near $0$, then $\ed{\ti{R}_{\e}^{-1}(h) \simeq \ti{R}^{-1}(h)}$. In this case, as is seen in the proof of Theorem \ref{thinqmthm}, $\frac{1}{V(\ti{R}^{-1}(h))} \simeq \frac{\ti{R}^{-1}(h)^2}{h}$ and the resolvent bounds from the two theorems coincide. That is, the resolvent bound in Theorem \ref{thinresolvethm} is sharp. So long as $m(\lambda)$ is of positive increase, the energy decay rate is sharp as well. 
	\item Theorem \ref{thinqmthm} does not require $m(\lambda)$ to be of positive increase in order to get the lower bound. This is due \ed{to} a general feature of lower bounds on energy decay rates from the underlying semigroup theory. See Proposition \ref{bdlemma}. 
\end{enumerate}
\end{remark}

\subsection{Literature Review}\label{litreview}
The equivalence of exponential uniform stabilization and the geometric control condition is due to \cite{RauchTaylor1975} and \cite{Ralston1969}. When $M$ has a boundary, this is due to \cite{BardosLebeauRauch1992} and \cite{BurqGerard1997}. See also \cite{Kleinhenz2026Time} for a generalization when the damping is allowed to depend on time and \cite{KeelerKleinhenz2023} when the damping is a $0$th order pseudodifferential operator. 

Decay rates without the geometric control condition go back to \cite{Lebeau1996}. When $W \in \ed{C^0}(M)$ is nonnegative and $\{W>0\}$ is open and nonempty, then \eqref{eq:stable} holds with $r(t)=\log(2+t)^{-1}$ \cite{Lebeau1996, Burq1998}. Furthermore, this rate cannot be improved in general, as it is sharp on spheres and surfaces of revolution \cite{Lebeau1996}. 

When $M$ is a torus and $W \in L^{\infty}$ this a priori rate can be improved to $r(t)= t^{-\frac{1}{2}}$ \cite{AL14}, as a consequence of Schr\"odinger observability. \ed{For earlier results, in the setting of the square and partially rectangular domains, see respectively \cite{LiuRao2005} and \cite{BurqHitrik2007}, and for a related result in the setting of a degenerately hyperbolic undamped set, see \cite{csvw}.} See \cite{KleinhenzWang2023} for the case when $W$ is unbounded, but $\{W>0\}$ still satisfies Schr\"odinger observability. There is a complimentary lower bound when $\overline{\{W>0\}}$ does not satisfy the geometric control condition, in particular the energy cannot decay faster than $r(t)=t^{-1}$ \cite{AL14}. 

The gap between this upper and lower bound is genuine. For example, when $W(x,y) = (|x|-\sigma)_+^{\beta}, \beta \geq 0$, the energy decays at the sharp rate $r(t)=t^{-\frac{\beta+2}{\beta+3}}$ \cite{Kleinhenz2019, DatchevKleinhenz2020}, the case $\beta=0$ is due to \cite{AL14, Stahn2017}. Note that $\frac{\beta+2}{\beta+3}$ ranges from $\frac{2}{3}$ to $1$. This sharp rate was extended to $\beta \in (-1,0)$ in \cite{KleinhenzWang2026}, although $W$ becomes unbounded for negative $\beta$. See also \cite{ArnaizSun23}.

Damping on the torus depending on both $x$ and $y$ have been previously studied as well. In \cite{AL14} it is shown that a damping $W \in W^{8, \infty}(\T^2)$ satisfying $|\nabla W| \leq C W^{1-\e}$ for some $C>0$ and $\e \in (0, 1/29)$ produces energy decay at rate $r(t) = t^{-\frac{1}{1+4\e}}$. When the support of the damping is assumed to be strictly convex with positive curvature there can be an improvement to the energy decay rate. In \cite{Sun23} it is shown that damping of the form $W(x,y) = (\sigma-|x^2+y^2|^{1/2})^{\beta}_+$ with $\beta \geq 9$ produces energy decay at the sharp rate $r(t) = t^{-\frac{\beta+\frac{1}{2}+2}{\beta+\frac{1}{2}+3}}$. That is, the energy decays as if $\beta$ was increased by $\frac{1}{2}$ in \cite{Kleinhenz2019, DatchevKleinhenz2020}. \ed{See \cite{DKP25} for a further generalization of this result.}

When the derivative bound condition is only satisfied with $|\nabla W| \leq C W^{1/2}$, for example $W$ vanishes like $\exp(-x^{-1}) \sin(x^{-1})^2$, \cite{AL14} does not provide a decay rate. For such a damping, if it is $y$-invariant, \cite{Kleinhenz2022} provides a decay rate of $t^{-\frac{4}{5}+\e}$ which improves on the a priori rate of $t^{-\frac{1}{2}}$. 

When $W=0$, only along a positive codimension, invariant torus and $W$ turns on like $x^{\beta}$ \cite{LeautaudLerner2017} provide the sharp rate of $r(t)=t^{-\frac{\beta+2}{\beta}}$ This was extended to a more general geometric setup, matching that of our Theorem \ref{thinresolvethm} in \cite{BurqZuily2015}. See also \cite{BurqZuily2016}.

All of the above polynomial energy decay rates rely on the semigroup theory result of \cite{BorichevTomilov2010}, which provides an equivalence between polynomial energy decay rates and polynomial resolvent estimates for the generator of the semigroup. This eliminated a log loss in decay rates, from the more general \cite{BattyDuyckaerts2008}. The equivalence of \cite{BorichevTomilov2010} was then generalized to allow for sharp non-polynomial energy decay rates in \cite{rss19}. Although the resolvent estimates proved in this paper are independent of these semigroup results, \cite{rss19} is essential for the energy decay rates we obtain. 

\subsection{Proof Outline}
\ed{To demonstrate the power of the main results, in Section \ref{s:applications} we prove  energy decay rates for a variety of non-polynomial dampings.}
Section \ref{wideexsec} contains the proof of decay rates for explicit examples which are $y$-invariant and $0$ on a strip, Theorem \ref{wideexthm}, using the more general Theorem \ref{wideresolvethm}. Section \ref{thinexsec} contains the proof of decay rates for explicit examples which are $y$-invariant and $0$ on a single trajectory, Theorem \ref{thinexthm} using the more general Theorems \ref{thinresolvethm} and \ref{thinqmthm}.  
Section \ref{widegenexsec} contains the proof of decay rates for damping depending on $x$ and $y$ and satisfying a polynomial derivative bound condition, Theorem \ref{dbctheorem}, using the more general Theorem \ref{widegenresolvethm}, as well as the proof of decay rates for damping vanishing along multiple geodesics, Remark \ref{xydampremark}, using Theorems \ref{averagedresolventthm} and \ref{thinresolvethm}.

Section \ref{semigroupsec} contains statements of the abstract results that relate energy decay rates to resolvent estimates for the semigroup generator and the stationary equation. 

Section \ref{wideresolventsec} contains the proof of the 1-d resolvent estimates of  Theorem \ref{wideresolvethm}. The proof uses a Morawetz multiplier method and a cutoff localized at a non-polynomial scale $g(h)$ to separately estimate $u$ on $\{W \leq g(h)\}$ and $\{W \geq g(h)\}$. 

Section \ref{widegensec} contains the proof of energy decay rates for non $y$-invariant damping satisfying non-polynomial derivative bound conditions, Theorem \ref{widegenresolvethm}, and the relation between resolvent estimates and averaged damping resolvent estimates, Theorem \ref{averagedresolventthm}. The proof of Theorem \ref{widegenresolvethm} follows from the 1-d resolvent estimates of Theorem \ref{wideresolvethm} and Theorem \ref{averagedresolventthm}. The proof of Theorem \ref{averagedresolventthm} uses a normal form argument to replace the damping by an averaged version of it.

\ed{Section \ref{s:geometryApplications} contains another set of applications of our main results, to dampings which have structured geometry to their supports. Specifically we prove improved energy decay rates when the damping is supported on a rectangle, strictly convex set, or superellipse, Theorems \ref{rectthm}, \ref{polylogconveximprovethm} and \ref{ellipsethm} respectively.} Section  \ref{avggrowthsec} contains proofs of improvements to the growth order of $W$ under averaging based on the geometry of the support of $W$. \ed{Section \ref{proofconvexrate} then uses these improvements to complete the proofs of the Theorems}. Those proofs rely on the 1-d resolvent estimates of Theorem \ref{wideresolvethm} and use Theorem \ref{averagedresolventthm} to convert them into the necessary resolvent estimates.

Section \ref{eigsec} contains the proof of a result describing how eigenfunctions of the Laplacian can concentrate on neighborhoods of geodesics at non-polynomial scales.

Section \ref{thinsec} contains the proofs of energy decay rates for damping $0$ along geodesics, Theorems \ref{thinresolvethm} and  \ref{thinqmthm}, via a resolvent estimate and quasimode construction respectively. The resolvent estimate relies on the eigenfunction non-concentration result from Section \ref{eigsec}

Appendix \ref{appendix} contains some basic inequalities and estimates used throughout the paper. 

\textbf{Acknowledgments} I am thankful to David Seifert for pointing out the problem of exponential damping to me and for some initial conversations on the problem. I am also thankful to Chenmin Sun, who originally pointed out this problem to David Seifert. I am also thankful to Jared Wunsch for recommending to me to consider damping with strictly convex support and for helpful comments on the draft. I am also thankful to Matthieu L\'eautaud for pointing out that decay for $x$ and $y$ dependent damping vanishing on geodesics, should be a consequence of the $y$-invariant and averaging results. I am also thankful to Nicolas Burq, and Ruoyu P.T. Wang for helpful correspondence and conversations. \ed{I am thankful to the referee for their thoughtful comments which improved the paper.}

\ed{\textbf{Funding} This work was supported by the National Science Foundation DMS-2530465. The initial conversations with David Seifert which started this paper took place at the 2023 International Workshop on Operator Theory and Applications in Helsinki, Finland. The author's travel to that conference was partially supported by the National Science Foundation grant DMS-1953940, administered by Ra\'ul Curto. }

\section{Applications}\label{s:applications}

\subsection{Examples for $y$-invariant $W$ with $W=0$ on an interval}\label{wideexsec}
We begin by stating new energy decay results for non-polynomial damping depending only on $x$.
\begin{theorem}\label{wideexthm}
	Suppose $W(x,y)=W(x),$ and there exists $C \geq 1, V(x),$ such that  when $W$ is near $0,$ $C^{-1} V(x) \leq W(x) \leq C V(x)$.
	\begin{enumerate}
		\item If $V(x)= \ed{(|x|-\sigma)_+^{\beta}} \exp(-c(|x|-\sigma)_+^{-\alpha})$ for $\alpha,c>0, \beta \in \Rb$, then \eqref{eq:stable} holds with rate 
		$$
		r(t) = t^{-1} \ln(\ed{2+}t)^{\frac{2\alpha+1}{\alpha}}.
		$$
		\item If $V(x)= (|x|-\sigma)_+^{\beta} \ln((|x|-\sigma)_+^{-1})^{-\gamma}$ for $\beta>0, \gamma \in \Rb$, then \eqref{eq:stable} holds with rate
		$$
		r(t) = t^{-\frac{\beta+2}{\beta+3}} \ln(\ed{2+}t)^{\frac{-\gamma}{\beta+3}}.
		$$
		\item If $V(x)= \ln((|x|-\sigma)_+^{-1})^{-\gamma}$ for $\gamma>0$, then for all $\e>0,$ \eqref{eq:stable} holds with rate
		$$
		t^{-\frac{2}{3}+\e} \ln(\ed{2+}t)^{\frac{-\gamma-1}{3}+\e(2-\gamma)}.
		$$
	\end{enumerate}
	
\end{theorem}

\begin{remark}
	\begin{enumerate}
		\item When $\overline{\{W>0\}}$ does not satisfy the GCC, based on past results we anticipate that $W$ turning on more gradually will produce faster energy decay. Theorem \ref{wideexthm} bears this out. In all cases, a \emph{larger} $\alpha, \beta,$ or $\gamma$ means $W$ turns on more \emph{slowly} and a \emph{faster} decay rate is guaranteed. 
		\item As mentioned above, case 1 improves on the $t^{-1+\d}$ decay guaranteed by \cite[Remark 2.7]{AL14}. Now the decay is only slower than $t^{-1}$ by a $\ln(\ed{2}+t)$ term. Because $\overline{\{W>0\}}$ does not satisfy the GCC the energy cannot decay faster than $t^{-1}$ \cite[Theorem 2.5]{AL14}, and so the polynomial rate is sharp. 
		\item Case 2 is a generalization of the rate from \cite{Kleinhenz2019, DatchevKleinhenz2020}. It recovers the  sharp rate for $\gamma=0$, and when $\gamma$ is non-zero there is a $\gamma$-dependent $\ln(\ed{2}+t)$ change in the energy decay rate.  Case 2 is nearly an interpolation between the polynomial decay rates of \cite{DatchevKleinhenz2020} and case 3 as $\beta \ra 0$. However, the powers on the $\ln(\ed{2}+t)$ terms do not match. On the other hand as $\beta \ra \infty$, the rate saturates the $t^{-1}$ lower bound with a log gain or loss. 
		\item In case 3 as $\gamma \ra 0,$ $W$ approaches an indicator function and  the energy decay rate approaches $t^{-2/3}$, which is the sharp decay rate for indicator functions \cite{Stahn2017} \cite[Corollary 2.9]{AL14}. If we were able to set $\e=0$, the rate would have a $\gamma$-dependent $\ln(\ed{2}+t)$ gain relative to the indicator function rate. 
		\item Based on existing results it is reasonable to expect that the powers of polynomial terms cannot be improved, but it is unclear what the sharp powers on the $\ln(\ed{2}+t)$ terms should be.  
	\end{enumerate}
\end{remark}

In this subsection, we apply Theorem \ref{wideresolvethm} to prove Theorem \ref{wideexthm}.

\begin{example} \label{polyexp1dex}
For $\alpha, c >0, \beta \in \Rb$, suppose $V(x) = (|x|-\sigma)_+^{\beta} \exp\left( -c (|x|-\sigma)_+^{-\alpha} \right)$.
\begin{figure}[h]
\centering
\begin{tikzpicture}
\node at (0.3, 2) {$W$};
\node at (3, -0.2) {$x$};
\node at (0, -0.2) {$0$};
\draw [->] (0,0) -- (0,1.9);
\draw [->] (-3,0) -- (3,0);
\draw [usuai, domain=1:1.5, smooth, variable=\x, thick] plot [domain=1.001:2, smooth, variable=\x] ({2.5*(\x-.75)},{5*exp(-1/(\x-1)^.5)});
\draw [usuai, domain=1:1.5, smooth, variable=\x, thick] plot [domain=-2:-1, smooth, variable=\x] ({2.5*(\x+.75)},{5*exp(-1/(-\x-1)^.5)});
\draw [imayou, domain=1:1.5, smooth, variable=\x, thick] plot [domain=1.001:2, smooth, variable=\x] ({2.5*(\x-.75)},{5*exp(-1/(\x-1))});
\draw [imayou, domain=1:1.5, smooth, variable=\x, thick] plot [domain=-2:-1, smooth, variable=\x] ({2.5*(\x+.75)},{5*exp(-1/(-\x-1))});
\draw [persred, domain=1:1.5, smooth, variable=\x, thick] plot [domain=1.01:2, smooth, variable=\x] ({2.5*(\x-.75)}, {5*exp(-1/(\x-1)^2)});
\draw [persred, domain=1:1.5, smooth, variable=\x, thick] plot [domain=-2:-1.01, smooth, variable=\x] ({2.5*(\x+.75)}, {5*exp(-1/(-\x-1)^2)});

\draw [sand, domain=1:1.5, smooth, variable=\x, thick] plot [domain=1.1:2, smooth, variable=\x] ({2.5*(\x-.75)}, {5*exp(-1/(\x-1)^3)});
\draw [sand, domain=1:1.5, smooth, variable=\x, thick] plot [domain=-2:-1.1, smooth, variable=\x] ({2.5*(\x+.75)}, {5*exp(-1/(-\x-1)^3)});
\node[text=persred] at (4, .75) {$\alpha=2$};
\node[text=imayou] at (4,1.25) {$\alpha=1$};
\node[text=usuai] at (4,1.75) {$\alpha=.5$};
\node[text=sand] at (4,.25) {$\alpha=3$};
\end{tikzpicture}
\caption{$W=\exp(-(|x|-\sigma)_+^{-\alpha})$. As $\alpha$ increases, $W$ turns on more gradually.}
\end{figure}

 Then $q(z) \simeq \ln(z^{-1})^{-\frac{\alpha+1}{\alpha}}$, $p(z)  \simeq q^2(z)$ and $\Vidz \simeq \ln(z^{-1})^{-\frac{1}{\alpha}}$. So taking $U(z)=\Vidz$, Assumption \ref{wideinversebassumption} is satisfied. 
Then, using Lemma \ref{envinvlemma} 
\begin{equation}
R_2(z) \simeq z \ln (z^{-1})^{\frac{-2(\alpha+1)}{\alpha}}, \qquad \ti{R}_2^{-1}(h) \simeq h \ln ( h^{-1} )^{\frac{2(\alpha+1)}{\alpha}}. 
\end{equation}
Then 
\begin{equation}
M_2(h) \simeq \frac{1}{h} \ln(h^{-1})^{\frac{2\alpha+1}{\alpha}}, \quad m_2(\lambda) \simeq \lambda \ln(\lambda)^{\frac{2\alpha+1}{\alpha}}. 
\end{equation}
This $m_2(\lambda)$ is of positive increase, so applying Theorem \ref{wideresolvethm} part 2), gives energy decay at rate
\begin{equation}
r(t) = t^{-1} \ln(\ed{2}+t)^{\frac{2\alpha+1}{\alpha}}.
\end{equation}
Note that the values of $\beta$ and $c$ do not \ed{affect} the decay rate. 
\end{example}


\begin{example}\label{polylog1dex}
Suppose $V(x) =(|x|-\sigma)_+^{\beta} \ln((|x|-\sigma)_+^{-1})^{-\gamma}$, for $\beta>0, \gamma \in \Rb$.
\begin{figure}[h]
\centering
\begin{tikzpicture}
\node at (0.3, 2) {$W$};
\node at (3, -0.2) {$x$};
\node at (0, -0.2) {$0$};
\draw [->] (0,0) -- (0,1.9);
\draw [->] (-3,0) -- (3,0);
\draw [persred, domain=1:1.5, smooth, variable=\x, thick] plot [domain=1:1.3, smooth, variable=\x] ({7*(\x-.9)},{8*(\x-1)^2});
\draw [persred, domain=1:1.5, smooth, variable=\x, thick] plot [domain=-1.3:-1, smooth, variable=\x] ({7*(\x+.9)},{8*(-\x-1)^2});
\draw [imayou, domain=1:1.5, smooth, variable=\x, thick] plot [domain=1.001:1.3, smooth, variable=\x] ({7*(\x-.9)},{8*ln(1/(\x-1))^2*(\x-1)^2});
\draw [imayou, domain=1:2, smooth, variable=\x, thick] plot [domain=-1.001:-1.3, smooth, variable=\x] ({7*(\x+.9)},{8*ln(1/(-\x-1))^2*(-\x-1)^2});
\draw [usuai, domain=1:2, smooth, variable=\x, thick] plot [domain=1:1.3, smooth, variable=\x] ({7*(\x-.9)},{6*(\x-1)});
\draw [usuai, domain=1:2, smooth, variable=\x, thick] plot [domain=-1:-1.3, smooth, variable=\x] ({7*(\x+.9)},{6*(-\x-1)});
\draw [sand, domain=1:1.5, smooth, variable=\x, thick] plot [domain=1.001:1.3, smooth, variable=\x] ({7*(\x-.9)},{8/ln(1/(\x-1))*(\x-1)^2});
\draw [sand, domain=1:2, smooth, variable=\x, thick] plot [domain=-1.001:-1.3, smooth, variable=\x] ({7*(\x+.9)},{8/ln(1/(-\x-1))*(-\x-1)^2});
\node[text=persred] at (4, .75) {$\beta=2, \, \gamma=0$};
\node[text=imayou] at (4,1.25) {$\beta=2, \, \gamma=-2$};
\node[text=usuai] at (4,1.75) {$\beta=1, \, \gamma=0$};
\node[text=sand] at (4,.25) {$\beta=2, \,\gamma=2$};
\end{tikzpicture}
\caption{$W=(|x|-\sigma)_+^{\beta} \ln( (|x|-\sigma)^{-1})^{-\gamma}$. As $\beta$ and $\gamma$ increase, $W$ turns on more gradually.}
\end{figure}

 Then  $q(z) \simeq z^{\frac{1}{\beta}} \ln(z^{-1})^{\frac{\gamma}{\beta}}$, and we can take $U(z)=\Vidz \simeq C z^{\frac{1}{\beta}} \ln(z^{-1})^{\frac{\gamma}{\beta}}$.
Then, using Lemma \ref{envinvlemma} 
\begin{equation}
R_{1}(z) \simeq z^{\frac{\beta+2}{\beta}} \ln(z^{-1})^{\frac{2\gamma}{\beta}} , \quad \tiroh \simeq h^{\frac{\beta}{\beta+2}} \ln\left(h^{-1} \right)^{-\frac{2\gamma}{\beta+2}}. 
\end{equation}
Then 
\begin{equation}
M_{1}(h) \simeq h^{-\frac{\beta+3}{\beta+2}} \ln(h^{-1})^{\frac{-\gamma}{\beta+2}} , \qquad m_{1}(\lambda)\simeq \lambda^{\frac{\beta+3}{\beta+2}} \ln(\lambda)^{\frac{-\gamma}{\beta+2}}. 
\end{equation}
So $M_1(h) \geq h^{-1-\e}$ for some $\e>0$, and $m_{1}(\lambda)$ is of positive increase, so applying Theorem \ref{wideresolvethm} part 1), gives energy decay at rate
\begin{equation}
r(t) = t^{-\frac{\beta+2}{\beta+3}} \ln(\ed{2}+t)^{\frac{-\gamma}{\beta+3}}.
\end{equation}
\end{example}
\begin{example}\label{lnthinex}
If $V(x) =\ln\left( (|x|-\sigma)_+^{-1} \right)^{-\gamma}$, for $\gamma>0$,  then $q(z) \simeq z^{-\frac{1}{\gamma}} \exp(-z^{-\frac{1}{\gamma}})$.
We can take $U(z)=\Vidz \simeq \exp(-(\d z)^{-\frac{1}{\gamma}})$. Therefore, letting $C_1 = \frac{1+\d^{1/\gamma}}{\d^{1/\gamma}}\in (1,2)$ we have by Lemma \ref{envinvlemma}
\begin{align}
R_{1}(z) \simeq z^{\frac{\gamma-1}{\gamma}} \exp\left(-C_1 z^{\ed{-\frac{1}{\gamma}}} \right), & \quad  \tiroh \simeq \ln\left( h^{-\frac{1}{C_1}} \ln(h^{-\frac{1}{C_1}})^{\frac{(1-\gamma)}{C_1}} \right)^{-\gamma}.
\end{align}
Then 
\begin{equation}
M_{1}(h) \simeq h^{-\frac{C_1+1}{C_1}} \ln(h^{-1})^{\frac{1-\gamma-C_1}{C_1}} , \qquad m_{1}(\lambda) \simeq \lambda^{\frac{C_1+1}{C_1}} \ln(\lambda)^{\frac{1-\gamma-C_1}{C_1}} . 
\end{equation}
This $M_1(h) \geq h^{-1-\e}$ for some $\e>0$, and $m_{1}(\lambda)$ is of positive increase, so applying Theorem \ref{wideresolvethm} part 1)  gives energy decay at rate
\begin{equation}
r(t)=t^{-\frac{C_1}{C_1+1}} \ln(\ed{2}+t)^{\frac{1-\gamma-C_1}{C_1+1}}.
\end{equation}
Since $C_1= \frac{1+\d^{1/\gamma}}{\d^{1/\gamma}}$ and we can take $\d$ arbitrarily close to 1 from above, $C_1$ can be taken arbitrarily close to $2$ from below. Then for any $\e>0$, setting $\ed{C_1=\frac{2-3\e}{1+3\e}<2}$ \ed{we have} $\frac{C_1}{1+C_1}=\frac{2}{3}-\e$ and $\frac{1-\gamma-C_1}{C_1+1}=\frac{-\gamma-1}{3}+\e(2-\gamma)$. This gives the decay rate of the desired form in Theorem \ref{wideexthm} part 4).
\end{example}


\subsection{Examples for $y$-invariant $W$ with $W=0$ only along $x=0$}
\label{thinexsec}
When $\overline{\{W>0\}}$ satisfies the GCC, but $\{W>0\}$ does not, the behavior of $W$ as it turns on is still important. The most complete results in this area \cite{LeautaudLerner2017,BurqZuily2015} give sharp energy decay for $W(x,y)=|x|^{\beta}$ at rate $t^{-\frac{\beta+2}{\beta}}$. Another contribution of this paper is to provide sometimes sharp energy decay rates for a variety of non-polynomial $W$ in this setup.

\begin{theorem}\label{thinexthm}
	\begin{enumerate}
		\item If $W(x,y) = x^{\beta} \exp\left(-c|x|^{-\alpha}\right)$, for $\alpha,c >0, \beta \in \Rb,$  then for any $\e>0$, \eqref{eq:stable} holds at rate 
		$$
		r(t) = \ed{t^{-1} \ln(2+t)^{\frac{2\alpha+1}{\alpha}}},
		$$
		and cannot hold at rate faster than 
		$$
		r(t) = t^{-1} \ln(\ed{2}+t)^{-\frac{2}{\alpha}}.
		$$
		\item If $W(x,y) = |x|^{\beta}\ln(|x|^{-1})^{-\gamma}$ for $\beta>0, \gamma \in \Rb$, then \eqref{eq:stable} holds at rate
		$$
		r(t) = t^{- \frac{\beta+2}{\beta}}\ln(\ed{2}+t)^{\frac{2\gamma}{\beta}},
		$$
		and cannot hold at a faster rate. 
		\item If $W(x,y) =\ln( |x|^{-1})^{-\gamma}$ for $\gamma > 0$, then for some $C>0$, \eqref{eq:stable} holds at rate
		$$
		r(t) = \exp(-C t^{\frac{1}{\gamma+1}}),
		$$
		and  cannot hold at rate faster than
		$$
		r(t) = \exp(-C t^{\frac{1}{\gamma}}).
		$$
	\end{enumerate}
\end{theorem}
\begin{remark}
	\begin{enumerate}
		\item When $\overline{\{W>0\}}$ satisfies the GCC, we expect and see the exact \emph{opposite} trend from Theorem \ref{wideexthm}. That is we expect that the slower $W$ turns on, the slower the energy decays. Theorem \ref{thinexthm} bears this out. In all cases a \emph{larger} $\alpha, \beta$, or $\gamma$ means $W$ turns on more \emph{slowly} which produces \emph{slower} energy decay. 
		\item Case 2 with $\gamma=0$ exactly reproduces the sharp rate from \cite{LeautaudLerner2017, BurqZuily2015}. Case 2 can be thought of as a modification of \cite{LeautaudLerner2017, BurqZuily2015} by multiplying by $\ln(x)^{\gamma}$, which in turn produces a $\gamma$ dependent $\ln(\ed{2}+t)$ change in the energy decay rate.
		\item In case 3, as $\gamma \ra 0$, $W$ approaches the identity function, which satisfies the GCC, and the energy decay rate approaches, \ed{but is never faster than,} $\exp(-t)$, which is the expected energy decay rate when the GCC is satisfied.
		
	\end{enumerate}
\end{remark}
\begin{example}\label{ex:expthin}
If $V(x)= x^{\beta} \exp(-c x^{-\alpha})$, for $\alpha,c >0, \beta \in \Rb$. \ed{To get the upper bound we actually apply Theorem \ref{wideexthm}(1) with $\sigma=0$. If instead we used Theorem \ref{thinresolvethm} the rate we obtain is worse. To expand on this,} for $\e>1$ letting $C_1= \frac{\e^{-\alpha}+1}{2}$ we have
\begin{align}
R_{\e}(z) &\ed{\simeq} z^{2+\beta} \exp(-C_1 c \ed{z}^{-\alpha}), \quad \ed{\ti{R}_{\e}^{-1}}(h) \simeq \ln\left(h^{-\frac{1}{C_1c}} \ln(h^{-\frac{1}{C_1 c}})^{-\frac{\beta+2}{\alpha C_1 c}}\right)^{-\frac{1}{\alpha}},\\
V(\ti{R}_{\e}^{-1}(h))&\simeq h^{\frac{1}{C_1}} \ln(h^{-1})^{\frac{\ed{2}+\beta(1-C_1)}{\alpha C_1}}, \quad \frac{\ti{R}^{-1}_{\e}(h)^2}{h}\simeq h^{-1} \ln(h^{-1})^{-\frac{2}{\alpha}}.
\end{align}
Therefore 
\begin{align}
M_{\e}(h) \simeq h^{-\frac{1}{C_1}} \ln(h^{-1})^{\frac{-2\ed{-\beta}(1-C_1)}{\alpha C_1}}, \qquad m_{\e}(\lambda) \simeq \lambda^{\frac{1}{C_1}} \ln(\lambda)^{\frac{-2\ed{-\beta}(1-C_1)}{\alpha C_1}}.
\end{align}
This $m_{\e}(\lambda)$ is of positive increase, so by Theorem \ref{thinresolvethm}, the energy decays at rate 
\begin{equation}
r(t) = t^{-C_1} \ln(\ed{2}+t)^{\frac{-2\ed{-\beta}(1-C_1)}{\alpha}}.
\end{equation}
Since $\e$ can be taken arbitrarily close to $1$ from above, $C_1= \frac{\e^{-\alpha}+1}{2}$ can be taken arbitrarily close to $1$ from below, and this gives the rate 
	\begin{equation}
	r(t)=t^{-1+\e} \ln(t)^{\frac{-2\ed{-}\beta \e}{\alpha}},
\end{equation}
\ed{which we note is worse than the rate provided by Theorem \ref{wideexthm}.}

On the other hand
\begin{align}
&R(z) = z^{2+\beta} \exp(-c\ed{z}^{-\alpha}), \quad \ti{R}^{-1}(h)\simeq \ln(h^{-1/c} \ln(h^{-1/c})^{-\frac{\beta+2}{\alpha c}} )^{-\frac{1}{\alpha}}\\
&V(\ti{R}^{-1}(h)) \simeq h \ln(h^{-1})^{\ed{\frac{2}{\alpha}}}.
\end{align}
Then
\begin{equation}
M(h) \simeq h^{-1} \ln(h^{-1})^{-\frac{2}{\alpha}}, \quad m(\lambda) \simeq \lambda \ln(\lambda)^{-\frac{2}{\alpha}}.
\end{equation}
So by Theorem \ref{thinqmthm}, there are solutions decaying no faster than 
\begin{equation}
r(t) = t^{-1} \ln(\ed{2}+t)^{-\frac{2}{\alpha}}.
\end{equation}
\end{example}
%

\begin{example}\label{ex:xbetathin}
If $V(x) = x^{\beta} \ln(x^{-1})^{-\gamma}$ for $\beta>0, \gamma \in \Rb$, then 
\begin{align}
&R(z)\simeq R_{\e}(z) \simeq z^{\beta+2} \ln(z^{-1})^{-\gamma}, \qquad \ti{R}^{-1}(h)\simeq h^{\frac{1}{\beta+2}} \ln(h^{-1})^{\frac{\gamma}{\beta+2}}\\
&V(\ti{R}^{-1}(h)) \simeq h^{\frac{\beta}{\beta+2}} \ln(h^{-1})^{\frac{-2\gamma}{\beta+2}} \simeq \frac{h}{\ti{R}^{-1}(h)^2}.
\end{align}
Then
\begin{equation}
M(h) \simeq  h^{-\frac{\beta}{\beta+2}} \ln(h^{-1})^{\frac{2\gamma}{\beta+2}}, \qquad m(\lambda)\simeq \lambda^{\frac{\beta}{\beta+2}} \ln(\lambda)^{\frac{2\gamma}{\beta+2}}.
\end{equation}
This $m(\lambda)$ is of positive increase and so by Theorems \ref{thinresolvethm} and \ref{thinqmthm}  the sharp energy decay rate is 
\begin{equation}
r(t) = t^{-\frac{\beta+2}{\beta}} \ln(\ed{2}+t)^{\frac{2\gamma}{\beta}}.
\end{equation} 
\end{example}

\begin{example}\label{ex:logthin}
Suppose $V(x) = \ln(x^{-1})^{-\gamma}$ for $\gamma>0$. 
\begin{figure}[h]
\centering
\begin{tikzpicture}
\node at (0.3, 2) {$W$};
\node at (1, -0.2) {$x$};
\node at (0, -0.2) {$0$};
\draw [->] (0,0) -- (0,1.9);
\draw [->] (-2,0) -- (2,0);
\draw [usuai, domain=0:1, smooth, variable=\x, thick] plot [domain=0.0001:.37, smooth, variable=\x] ({5*\x},{2/ln(1/\x)});
\draw [usuai, domain=0:1, smooth, variable=\x, thick] plot [domain=-0.0001:-.37, smooth, variable=\x] ({5*\x},{2/ln(-1/\x)});
\draw [imayou, domain=0:1, smooth, variable=\x, thick] plot [domain=0.0001:.37, smooth, variable=\x] ({5*\x},{2/ln(1/\x)^2});
\draw [imayou, domain=0:1, smooth, variable=\x, thick] plot [domain=-0.0001:-.37, smooth, variable=\x] ({5*\x},{2/ln(-1/\x)^2});
\draw [persred, domain=0:1, smooth, variable=\x, thick] plot [domain=0.0001:.37, smooth, variable=\x] ({5*\x},{2/ln(1/\x)^3});
\draw [persred, domain=0:1, smooth, variable=\x, thick] plot [domain=-0.0001:-.37, smooth, variable=\x] ({5*\x},{2/ln(-1/\x)^3});
\draw [sand, domain=0:1, smooth, variable=\x, thick] plot [domain=0.0001:.37, smooth, variable=\x] ({5*\x},{2/ln(1/\x)^4});
\draw [sand, domain=0:1, smooth, variable=\x, thick] plot [domain=-0.0001:-.37, smooth, variable=\x] ({5*\x},{2/ln(-1/\x)^4});
\node[text=sand] at (3, .25) {$\gamma=4$};
\node[text=persred] at (3,.75) {$\gamma=3$};
\node[text=imayou] at (3,1.25) {$\gamma=2$};
\node[text=usuai] at (3,1.75) {$ \gamma=1$};
\end{tikzpicture}
\caption{$W=\ln(|x|^{-1})^{-\gamma}$. As $\gamma$ increases, $W$ turns on more gradually.}
\end{figure}
Then we have
\begin{align}
&R(z) \simeq R_{\e}(z) \simeq z^2 \ln(z^{-1})^{-\gamma} \qquad  \ti{R}^{-1}(h) \simeq h^{\frac{1}{2}} \ln(h^{-1})^{\frac{\gamma}{2}}. \\
&V(\ti{R}^{-1}(h)) = \ln(h^{-1})^{-\gamma} \simeq \frac{h}{\ti{R}^{-1}(h)^2}.
\end{align}
So $M(h) \simeq \ln(h^{-1})^{\gamma}$ and $m(\lambda)\simeq\ln(\lambda)^{\gamma}$. That is by Theorems \ref{thinresolvethm} and \ref{thinqmthm}, the sharp resolvent estimate is 
\begin{equation}
\nm{(i\lambda - \Ac)^{-1}}_{\Lc(\Hc)} \lesssim \ln(\lambda)^{\gamma}.
\end{equation}
This $m(\lambda)$ is not of positive increase, so by Theorem \ref{thinresolvethm}, there exists $C>0$, such that the energy decays at rate 
\begin{equation}
r(t) = \exp(-Ct^{\frac{1}{\gamma+1}}).
\end{equation}
On the other hand, Theorem \ref{thinqmthm} ensures that for all $C>0$, there are solutions decaying no faster than
\begin{equation}
r(t) = \exp(-Ct^{\frac{1}{\gamma}}).
\end{equation}
\end{example}

\subsection{Example for damping depending on $x$ and $y$.}
\label{widegenexsec}
In this subsection we prove \ed{Theorem \ref{prelimgendbc}} using Theorem \ref{widegenresolvethm}. We also show how to combine Theorems \ref{averagedresolventthm} and Theorem \ref{thinresolvethm} \ed{to obtain energy decay rates for damping on the torus which vanish on finitely many intersecting geodesics}.


As mentioned above, if $W \in W^{8, \infty}(\T^2)$ and $W$ satisfies $|\nabla W| \leq CW^{1-\e}$ for $\e \in (0, \frac{1}{29})$, then \cite[Theorem 2.6]{AL14} proves energy decay at  rate $r(t)=t^{-\frac{1}{1+4\e}}$. \ed{We now recall our Theorem \ref{prelimgendbc} and provide some discussion of it.}
\begin{theorem}\label{dbctheorem}
	If $W \in \mathcal{D}^{9,\frac{1}{4}}(\T^2)$ and for some $C>0, \e \in (0,\frac{1}{4}]$, $|\nabla W| \leq C W^{1-\e}$, then \eqref{eq:stable} holds at rate 
	\begin{equation}
		r(t) = t^{-\frac{1+\e}{1+2\e}}.
	\end{equation}
\end{theorem}
\begin{remark}
	\begin{enumerate}
		\item  As a point of comparison at $\e=\frac{1}{4}$, \cite[Theorem 2.6]{AL14} would give decay at $t^{-\frac{1}{1+4\e}}=t^{-\frac{1}{2}}$, which does not improve over the a priori energy decay rate, while \ed{our} result gives energy decay at $t^{-\frac{5}{6}}$. Both results give energy decay rates approaching $t^{-1}$ as $\e\ra0$ although $\frac{1+\e}{1+2\e} > \frac{1}{1+4\e}$ for all $\e \in (0,\frac{1}{4}]$, so the decay from Theorem \ref{dbctheorem} is always faster than that of \cite[Theorem 2.6]{AL14}.
		\item This improvement ultimately comes from an improvement in the 1-d resolvent estimate, which originally appeared in \cite{KleinhenzThesis} and was due to joint work with Kiril Datchev.  It is perhaps possible to adapt the approach used in this 1-d resolvent estimate to the creation of the key cutoff function in \cite[Proposition 7.3]{AL14} to remove the additional regularity and derivative bound condition assumptions on $W$ in Theorem \ref{dbctheorem}.
	\end{enumerate}
\end{remark}
For damping that fit the hypotheses of both Theorem \ref{wideexthm} and \ref{dbctheorem}, the energy decay rate provided by Theorem \ref{dbctheorem} is never as good as the decay rate provided by Theorem \ref{wideexthm}. This improvement comes from estimating the size of $\{W \leq g(h)\}$ using the growth behavior of $W$, which cannot be done directly from the derivative bound condition. The same strategy can be used for $W$ depending on $y$, so long as there is some structure to the growth behavior of $W$.
\begin{example}[Proof of Theorem \ref{dbctheorem}]
Suppose $W \in \Dc^{9,\frac{1}{4}}(\T^2)$, and there exists $C>0, \e \in (0,\frac{1}{4}]$ such that $|\nabla W| \leq C W^{1-\e}$. Thus $W$ satisfies Assumption \ref{noninvarassumption} with $q(z)=z^{\e}$, where $r_1(z)=z^{1-\e}$ is concave. Then
\begin{equation}
R_{1}(z) = z^{1+\e}, \quad \tiroh= h^{\frac{1}{1+\e}}. 
\end{equation}
Therefore $M_{1}(h) = h^{-\frac{1+2\e}{1+\e}}, m_{1}(\lambda) = \lambda^{\frac{1+2\e}{1+\e}}$. This $M_1(h) \geq h^{-1-\d}$ for some $\d>0$, and $m_{1}(\lambda)$ is of positive increase, so applying Theorem \ref{widegenresolvethm} part 1) gives energy decay at rate
\begin{equation}
r(t)=t^{-\frac{1+\e}{1+2\e}}.
\end{equation} 
Note that this is faster decay then obtained by applying Theorem \ref{widegenresolvethm} part 2), which would require a stronger assumption on $|\nabla^2 W|$ and would only give energy decay at rate $t^{-\frac{1+2\e}{1+4\e}}$. \qed
\end{example}

Now we show how Theorems \ref{averagedresolventthm} and Theorem \ref{thinresolvethm} can be combined to provide energy decay rates for damping $W$ on the torus which vanish along finitely many intersecting geodesics. First, we consider $W(x,y) \simeq |x|^{\beta_1} |y|^{\beta_2}$, which is the most straightforward example.
\begin{example}\label{ex:xybeta}
	Consider $W(x,y)=|x|^{\beta_1}|y|^{\beta_2}$ for $\beta_1, \beta_2>0$. So $W=0$ along $x=0$ and $y=0$ and it does not satisfy the hypotheses of Theorem \ref{thinresolvethm} directly. 
	
	Instead, for $v \in \Sb^1$ periodic consider $A_v(W)$. If $v \neq (1,0), (0,1)$ then $A_v(W) \geq c >0$ for some $c>0$, and so by classical geometric control condition results (c.f. \cite{RauchTaylor1975} or \cite[Section 2.3]{KleinhenzThesis}) \ed{there exists $h_0>0$ such that for $h \in (0,h_0)$ we have }
	\begin{equation}
		\nm{\left(\frac{i}{h}-\Ac_v\right)^{-1}}_{\Lc(\Hc)} \leq C. 
	\end{equation}
	On the other hand, when $v=(1,0)$, then $\Ac_v(W) \simeq y^{\beta_2}$ and so by Theorem \ref{thinresolvethm} and Example \ref{ex:xbetathin} \ed{there exists $h_y>0$ such that for $h \in (0,h_y)$ we have }
	\begin{equation}
		\nm{\left(\frac{i}{h}-\Ac_v\right)^{-1}}_{\Lc(\Hc)} \leq C h^{-\frac{\beta_2}{\beta_2+2}}.
	\end{equation}
	Similarly, when $v=(0,1)$, then $\Ac_v(W) \simeq x^{\beta_1}$ and so by Theorem \ref{thinresolvethm} and Example \ref{ex:xbetathin} \ed{there exists $h_x>0$ such that for $h \in (0,h_x)$ we have }
	\begin{equation}
		\nm{\left(\frac{i}{h}-\Ac_v\right)^{-1}}_{\Lc(\Hc)} \leq C h^{-\frac{\beta_1}{\beta_1+2}}.
	\end{equation}
	So now if $\beta=\ed{\max}(\beta_1, \beta_2)$ \ed{then for $h \in (0,h_*)$ with $h_*=\min(h_0,h_y,h_x)$ we have}
	\begin{equation}
		\nm{\left(\frac{i}{h}-\Ac_v\right)^{-1}}_{\Lc(\Hc)} \leq C h^{-\frac{\beta}{\beta+2}},
	\end{equation}
	for all $v \in \Sb^1$ periodic. Then by Theorem \ref{averagedresolventthm} and writing $h^{-1}=\lambda$, \ed{there exists $\lambda_0>0$ such that for $\lambda \in \Rb$ with $|\lambda| \geq \lambda_0$ we have}
	\begin{equation}
		\nm{\left(i\lambda - \Ac\right)^{-1}}_{\Lc(\Hc)} \leq C \lambda^{\frac{\beta}{\beta+2}}.
	\end{equation}
	Since this right hand side is of positive increase, energy decays at rate 
	\begin{equation}
		t^{-\frac{\beta+2}{\beta}}.
	\end{equation}
	The same argument can be used, possibly invoking Examples \ref{ex:expthin} and \ref{ex:logthin}, to give explicit energy decay rates for $W(x,y)=V_1(x) V_2(y)$ with $V_j(z) \in \{z^{\beta}\ln(z^{-1})^{-\gamma}, z^{\beta}\exp(-cz^{-\alpha}), \ln(z^{-1})^{-\gamma}\}.$ The decay rate will be the slower of the rates coming from $W(x,y)= V_1(x)$ and $W(x,y)=V_2(y)$. 
	
	Now consider $W$ vanishing along finitely many periodic geodesics $\gamma_j, j=1, \ldots, n$ and growing like $V_j(d_j)$ near $\gamma_j$, where $d_j$ is the distance to $\gamma_j$ and $V_j$ satisfies Hypothesis 2 of Theorem \ref{thinresolvethm}. For example $V_j(d_j)=d_j^{\beta_j}$ and $W(x,y)=\prod_{j=1}^n d_j^{\beta_j}$. In this general setup, we can follow the same argument, replacing $(1,0)$ and $(0,1)$ by the $v_j \in \Sb^1$ which generate $\gamma_j$, and applying the full abstract version of Theorem \ref{thinresolvethm}, to obtain energy decay at rate given by the slowest energy decay rate of $W(x,y)=V_j(d_j)$. In the explicit example here, that rate would be $t^{-\frac{\beta+2}{\beta}}$ for $\beta = \ed{\max}_{1 \leq j \leq n} \beta_j$. 
\end{example}

\section{Semigroup Theory}
\label{semigroupsec}
\ed{In this section we discuss results from abstract semigroup theory, which apply for general Riemannian manifolds $(M,g)$.} There is a standard equivalence between resolvent estimates for the semigroup generator of the damped wave equation and its stationary equation. The following is Lemma 3.7 of \cite{KleinhenzWang2023}.
\begin{lemma}\label{resolvequivlem}
Let $\lambda_0 >0$ and $K(\lambda) \geq C$ uniformly for $|\lambda| \geq \lambda_0$. The following are equivalent: 
\begin{enumerate}
\item There exists $C,\lambda_0>0$ such that
\begin{equation}\label{resolveequiv}
\|(-\Delta+i\lambda W -\lambda^2)^{-1} \|_{\mathcal{L}(L^2)}\le C  \frac{K(\abs{\lambda})}{\abs{\lambda}},
\end{equation}
for all $\lambda\in\mathbb{R}$ with $\abs{\lambda}\ge \lambda_0$.
\item There exists $C,\lambda_0>0$ such that
\begin{equation}\label{semigroupequiv}
\|(i\lambda-\Ac)^{-1}\|_{\mathcal{L}(\mathcal{H})}\le CK(\abs{\lambda}),
\end{equation}
for all $\lambda\in\mathbb{R}$ with $\abs{\lambda}\ge \lambda_0$.
\end{enumerate}
\end{lemma} 
When $K(\lambda)$ is a function of positive increase, there is an equivalence between these resolvent estimates and energy decay rates, due to \cite[Theorem 3.2]{rss19} (see also Sections 3.2 and 3.3 of \cite{KleinhenzWang2023} for technical details of the equivalence). Note that, by elliptic unique continuation, if $\{W>0\}$ is open and nonempty, then $(-\Delta-\lambda^2)u =0$ implies that $Wu \neq 0.$
\begin{proposition}\label{rsslemma}
Let $K(\lambda):[0,\infty) \ra (0,\infty)$ be a \ed{continuous} function of positive increase. The following are equivalent:
\begin{enumerate}
\item There exists $C,\lambda_0>0$ such that
\begin{equation}\label{resolveequiveq}
\|(-\Delta+i\lambda W -\lambda^2)^{-1} \|_{\mathcal{L}(L^2)}\le C  \frac{K(\abs{\lambda})}{\abs{\lambda}},
\end{equation}
for all $\lambda\in\mathbb{R}$ with $\abs{\lambda}\ge \lambda_0$.
\item There exists $C,\lambda_0>0$ such that
\begin{equation}\label{semigroupequiveq}
\|(i\lambda-\Ac)^{-1}\|_{\mathcal{L}(\mathcal{H})}\le CK(\abs{\lambda}),
\end{equation}
for all $\lambda\in\mathbb{R}$ with $\abs{\lambda}\ge \lambda_0$.
\item There exists $C>0$ such that for all $t \geq 0$ we have 
\begin{equation}\label{energyequiveq}
E(u,t)^{\frac{1}{2}} \le \frac{C}{\ed{\ti{K}^{-1}(t)}} \left( \nm{u_0}_{H^2} + \nm{u_1}_{H^1} \right).
\end{equation}
\end{enumerate}
\end{proposition}
\ed{ Strictly speaking, \cite{rss19} has $K^{-1}(t)$ in the denominator in the energy decay rate in (3). However, by \cite[Proposition 2.2]{rss19} and the definition of the envelope inverse, if $K$ is continuous and of positive increase then $K^{-1}(t) \simeq \ti{K}^{-1}(t)$}

If $K(\lambda)$ is not of positive increase, as in Example \ref{lnthinex}, then \cite[Proposition 3, Theorem 5]{BattyDuyckaerts2008} are used to provide a relation between resolvent estimates and energy decay rates with a $\log$ loss in one direction.
\begin{proposition}\label{bdlemma}
Let $K(\lambda)$ be a continuous non-decreasing function \ed{and $\lambda_0>0$}, such that \ed{for all $\lambda \in \Rb$ with $|\lambda| \geq \lambda_0$ we have}
$$
\|(-\Delta+i\lambda W -\lambda^2)^{-1} \|_{\mathcal{L}(L^2)}\le C \frac{K(\abs{\lambda})}{\abs{\lambda}},
$$
or equivalently 
$$
\|(i\lambda-\Ac)^{-1}\|_{\mathcal{L}(\mathcal{H})}\le CK(\abs{\lambda}).
$$
Define $K_{\log}(\lambda) = K(\lambda)\left(\ln(1+\lambda)+\ln(1+K(\lambda)) \right)$. Then, there exists $C,c>0,$ such that for all $t\ge 0$
\begin{equation}\label{logenergyuppereq}
E(u,t)^{\frac{1}{2}} \le \frac{C}{\ed{\ti{K}}_{\log}^{-1}(ct)} \left( \nm{u_0}_{H^2} + \nm{u_1}_{H^1} \right).
\end{equation}
On the other hand, if the above resolvent estimates are sharp, that is, there exists sequences $\lambda_n \in \Rb, u_n \in L^2(M)$ with $\ltwo{u_n}=1$ and 
\begin{equation}
\ltwo{(-\Delta+i\lambda_n W-\lambda_n^2) u_n} \leq \frac{C |\lambda_n|}{K(|\lambda_n|)},
\end{equation}
then, \ed{there exists} $C,c>0,$ \ed{and} $(u_0, u_1) \in H^2(M) \times H^1(M)$, such that for all $t \ge 0$ 
\begin{equation}\label{logenergylowereq}
E(u,t)^{\frac{1}{2}} \geq \frac{C}{\ed{\ti{K}}^{-1}(ct)} \left(\nm{u_0}_{H^2} + \nm{u_1}_{H^1} \right).
\end{equation}
\end{proposition}

\section{\ed{Proof for $y$-invariant $W$, Theorem \ref{wideresolvethm}}}
\label{wideresolventsec}

Starting from \eqref{resolveequiveq} in Proposition \ref{rsslemma} with $K(\lambda)=m_1(\lambda)$, then expanding in a Fourier series in the $y$ variable, and setting $\lambda=1/h$, we see that to prove Theorem \ref{wideresolvethm} it is sufficient to prove the following proposition
\begin{proposition}\label{wideresolveprop}
\begin{enumerate}
\,
	\item If $W$ satisfies Assumption \ref{wideassumption}, then the following holds with $j=1$.
	\item If $W$ satisfies Assumptions \ref{wideassumption}, \ref{widebadassumption}, and \ref{wideinversebassumption}, then the following holds with $j=2$.
\end{enumerate}
	There exists $C, h_0 >0$ such that for all $h \in (0,h_0), \ed{E \leq h_0^{-2}}, u \in H^2(\Sb^1)$ and 
\begin{equation}\label{fueq}
f:= \left(-\p_x^2 + \frac{i}{h} W - E\right)u,
\end{equation}
then 
\begin{equation}
\nm{u}_{L^2(\Sb^1)} \leq  h M_j(h) \nm{f}_{L^2(\Sb^1)}.
\end{equation}
\end{proposition}
The core idea of the proof is to separately estimate $u$ where $V$ is small, and where $V$ is large. The key is to select the location of this split correctly, so that the contributions from these regions are balanced, and thereby minimized. The optimal location of the split is related to the growth of $V$ and a novelty of this work is doing this for non-polynomial $V$.

Let $g(h)$ be a function such that $\limh g(h)=0$. We will eventually take $g(h) \simeq \ed{\ti{R}_j^{-1}(h)}$ to obtain our desired estimate. The definition of $\tiroh$ is needed to balance terms in \eqref{whyrd} and the definition of $\ti{R}_2^{-1}(h)$ is needed to balance terms in \eqref{whyr2}. Note that indeed $\limh \tiroh=0=\limh \ti{R}_2^{-1}(h)$. 

As a matter of exposition, note that by their construction and definition the key quantities $h, g(h), q(g(h)),$ and  $p(g(h))$ all go to zero as $h \ra 0$. $U(g(h))$ is bounded as $h \ra 0$ and sometimes goes to $0$ with $h$. 


Now, let $\chi \in \Ci(\Rb; [0,1])$ have $\chi(x) \equiv 1$ on $|x|\leq1$, and $\chi(x)\equiv 0$ on $|x|\geq \d$. Note that this $\d$ is the same $\d$ as in the definition of $U(z)$ in Definition \ref{vidzdef}. Note also, $\d>1$ is needed to ensure that $\chi$ is continuous. Then define $\chi_h(x)= \chi\left(\frac{V(x)}{g(h)}\right)$, so $\chi_h(x)\equiv 1$ on $V(x) \leq g(h)$, and $\chi_h(x) \equiv 0$ on $V(x) \geq \d g(h)$. 

We will split $u$ as $u(x) = \chi_h(x) u(x) + (1-\chi_h(x)) u(x)$. An important feature of this cutoff is that there is control of the size of its first derivative in terms of $q$, and control of its second derivative in terms of $q$ and $p$.

\begin{lemma}\label{cutoffsizelemma}
Let $\mathbbm{1}_{\ch'}(x)=1$ on $\supp \ch'$ and be $0$ elsewhere. 
\begin{enumerate}
\item If $W$ satisfies Assumption \ref{wideassumption}, then there exists $C>0$ such that 
$$
|\ch'(x)| \leq C \mathbbm{1}_{\ch'}(x) \frac{1}{q(g(h))} \leq \ed{\frac{C}{q(g(h))} \min  \left\{\frac{W}{g(h)}, \frac{W^{\frac{1}{2}}}{g(h)^{\frac{1}{2}}} \right\} },
$$
	\item If $W$ satisfies Assumptions \ref{wideassumption} and \ref{widebadassumption}, then \ed{there exists $C>0$ such that}
\begin{align*}
|\ch''(x)| &\leq C \mathbbm{1}_{\ch'}(x) \left( \frac{1}{q^2(g(h))} + \frac{1}{p(g(h))} \right) \\
&\leq \ed{C\left( \frac{1}{q^2(g(h))} + \frac{1}{p(g(h))} \right) \min \left\{\frac{W}{g(h)}, \frac{W^{\frac{1}{2}}}{g(h)^{\frac{1}{2}}} \right\}}.
\end{align*}
\end{enumerate}
\end{lemma}
\begin{proof}
Well
$$
|\ch'(x)| = \left| \chi'\left(\frac{V(x)}{g(h)}\right) \frac{V'(x)}{g(h)} \right| \leq \mathbbm{1}_{\ch'}(x) \frac{CV(x)}{g(h) q(V(x))},
$$
by Assumption \ref{wideassumption}.

1) Since $q$ is increasing, and $g(h) \leq V(x)$ on $\supp(\ch')$, then $\frac{1}{q(V(x))} \leq \frac{1}{q(g(h))}$. Plugging this in and using that $\frac{V(x)}{g(h)} \leq \d$ on $\supp \ch'$ gives the desired conclusion. The second inequality follows since $\ed{ \frac{W}{g(h)} \simeq \frac{W^{\frac{1}{2}}}{g(h)^{\frac{1}{2}}} \simeq C}$ on $\supp \ch'$.


\ed{2)} For 
$$
|\ch''(x)| \leq \left| \ed{\chi}''\left(\frac{V(x)}{g(h)}\right) \left( \frac{V'(x)}{g(h)}\right)^2 \right| + \left| \ed{\chi}'\left(\frac{V(x)}{g(h)}\right) \frac{V''(x)}{g(h)}\right|,
$$
an analogous argument gives the desired conclusion.
\end{proof}

We now record a useful lemma, which is exactly Lemma 1 from \cite{DatchevKleinhenz2020}. The integrals below and the remaining integrals in this section are all over $\Sb^1$ unless otherwise specified.
\begin{lemma}
Let $u \in H^2(\Sb^1), E \in \Rb$ and $f=(-\p_x^2 +\frac{i}{h}W-E) u$, then for all $h>0$ 
\begin{enumerate}
\item
\begin{equation}\label{wuestimate}
\int W|u|^2 dx \leq h \int |fu| dx.
\end{equation}
\item For any $\psi \in \Ci(\Sb^1)$ vanishing on an open neighborhood of $\{W>0\}$ there exists $C>0$ such that 
\begin{equation}\label{wupestimate}
\int \psi |u'|^2 dx \leq C(1+\max(0,E) h) \int |fu| dx.
\end{equation}
\item There exists $\ed{E_0}>0, C>0$ such that 
\begin{equation}\label{umusmall}
\int |u|^2 + |u'|^2 dx \leq C \int |f|^2 dx, \text{ when } E \leq \ed{E_0}.
\end{equation}
\end{enumerate}
\end{lemma}
The first two parts of this Lemma indicate that it is ``easy" to get a ``good" estimate of $u$ localized to where $W$ (and thus $V$) is large, in terms of $f$. Note that based on Definition \ref{RMdefs} we have $h M_j(h)  \geq C$ for $h$ small enough. Because of this, the third part of this Lemma proves our desired resolvent estimate for $E$ small enough, and it remains to show our desired estimate when $E \geq E_0$. 

Another useful tool is the following estimate, proved via the Morawetz multiplier method, which is arranged via the energy functional
\begin{equation}\label{multdef}
G(x) = |u'(x)|^2 + E |u(x)|^2.
\end{equation}
Before stating the estimate, recall $\digamma(\zeta)=\ed{\nu}\left(\{x: 0<V(x) \leq \zeta\}\right)$. 
	\ed{Since $V$ is continuous, $\digamma(\zeta)$ is strictly increasing for $\zeta \in (0, \max V)$. Thus $\digamma$ has an inverse, and since $U(z)<\nu((0,\max V))$, $\digamma^{-1}(U(z))$ is always well defined.}
Since $\digamma$ is strictly increasing and $\digamma(\d z) \leq U(z) < \ed{\nu}\ed{((0,\max V))}$, then $\d z \leq \digamma^{-1}(U(z)) <\max V$, and so $\{x: 0<V(x) \leq \d z\} \subseteq \{x: 0 < V(x) \leq \digamma^{-1}(U(z))\}$. Furthermore $\digamma(\digamma^{-1}(U(z))) \frac{1}{U(z)} \ed{=} 1$. 

As in \cite{DatchevKleinhenz2020} we adapt the construction of our multiplier to the growth properties of $V$.
\begin{lemma}\label{multiplierlem}
For $U(z)$ an increasing function, such that $\Vidz \leq U(z) < \ed{\nu}(\ed((0,\max V)))$,  define
\begin{equation}
\Theta(x) = \begin{cases} \frac{1}{U(g(h))} & \{x: 0<V(x)<\digamma^{-1}(U(g(h)))\} \\
1 & \text{ elsewhere},
\end{cases}
\end{equation}
then, there exists $C, h_0>0$, such that for all $\ed{E \in (0,h_0^{-2})},  h \in (0,h_0), u \in H^2(\Sb^1)$ and $f=(-\p_x^2 + \frac{i}{h} W - E) u$
\begin{equation}
\int \Theta (|u'|^2 + E |u|^2 ) dx \leq  C \int |f|^2 dx + \frac{C}{h} \int W|uu'| dx.
\end{equation}
\end{lemma}
\begin{proof}
Fix $c\in(\digamma^{-1}(U(z)), \max V)$, so  $\{V(x) > c\}$ is open and nonempty. We now define a continuous and piecewise linear function $b$ such that 
\begin{equation}\label{bdef}
b'(x)= \begin{cases}
1 & x \in \{V(x)=0\} \\
\frac{1}{\Vidgh} & x \in \{0<V(x) < \digamma^{-1}(\Vidgh)\} \\
1 & x \in \{\digamma^{-1}(\Vidgh) \leq V(x) \leq c\}\\
-M & x\in\{V(x) > c\},
\end{cases}
\end{equation}
with $M>0$ chosen, so that $b$ is $2\pi$ periodic. Note such an $M$ can be chosen by our construction of $\digamma^{-1}$.  
\ed{Then computing directly using $G$ and the equation for $f$ and $u$ we have }
$$
0 = \int (bG)'dx  = \int b' |u'|^2 + E b'|u|^2 dx  - 2 \Re \int b f \bar{u}' dx +  2 \Re \int \frac{b}{h} i W u \bar{u}' dx.
$$
Rearranging 
\begin{equation}\label{multipliereq}
\int b' |u'|^2 + E b'|u|^2 dx \leq 2  \left| \Re \int  b f \bar{u}' \right| dx  + 2\int \frac{b}{h} W |u\bar{u}'| dx.
\end{equation}
Now add a multiple of \eqref{wuestimate} and \eqref{wupestimate}, then apply Young's inequality for products to the $\int f u' dx$ term, absorbing the $u'$ back into the left hand side, to obtain the desired inequality. 
\end{proof}

\subsection{Proof of Proposition \ref{wideresolveprop} part 1}
To obtain the desired resolvent estimate the final term on the right hand side of Lemma \ref{multiplierlem} must be further estimated. To do so we separately estimate $\ed{\ch} W|uu'|$ and $\ed{(1-\ch)} W |uu'|$. 
\begin{lemma}\label{wuestimatelem}
For all $\e>0$ there exists $C, \ed{h_0} >0,$ such that for all $h \ed{\in (0,h_0)}, E  \ed{\in (0,h_0^{-2})}, u \in H^2(\Sb^1)$, $f=(-\p_x^2 + \frac{i}{h}\ed{W} - \ed{E})u,$ then
\begin{align*}
\frac{1}{h} \int W|uu'| dx \leq& C \left( \ed{E^{\frac{1}{2}}} + \frac{1}{q(g(h))} + \frac{g(h) \Udg}{h} \right) \int |f u|dx \\
&+ h^{-\frac{1}{2}} \left( \int |fu| dx \right)^{\frac{1}{2}} \left( \int V(1-\chi) f u dx \right)^{\frac{1}{2}} + \e \int \Theta |u'|^2 dx.
\end{align*}
\end{lemma}
\begin{proof}
Well by \eqref{wuestimate}
\begin{equation}\label{wuu1}
\int W|uu'| dx \leq \left( \int W|u|^2dx  \right)^{\frac{1}{2}} \left( \int W |u'|^{\ed{2}} dx \right)^{\frac{1}{2}} \leq C h^{\frac{1}{2}} \left( \int |fu| dx \right)^{\frac{1}{2}} \left( \int V |u'|^{\ed{2}} dx \right)^{\frac{1}{2}}.
\end{equation}
We now separately estimate the final term with $\ch$ and $(1-\ch)$.  First, on the support of $V(x) \ch(x)$, we have  $V(x)\leq \d g(h) \leq \digamma^{-1}(\Udg)$. Thus $\Theta(x) \Udg =1$ there. Therefore
\begin{equation}\label{vchiu}
\int V \ch |u'|^2dx  \leq C g(h) \Udg \int \Theta |u'|^2 dx. 
\end{equation}
Now note that by Assumption \ref{wideassumption} 
\begin{equation}
|\p_x (V(1-\ch))| = \left|V' (1-\ch) - V \ch' \right|\leq \left|\frac{CV}{q(V)} (1-\ch)\right| +\left| V \mathbbm{1}_{\ch'} \frac{1}{q(g)} \right|\leq \frac{CV}{q(g(h))},
\end{equation}
using that $V(x) \geq g(h)$ on $\supp(1-\ch)$, $q$ is increasing, and Lemma \ref{cutoffsizelemma}. Therefore, integrating by parts and applying \eqref{fueq}
\begin{align}
\int V(1-\ch) |u'|^2 dx &= -\Re \int \p_x (V(1-\ch)) u' \bar{u}dx  - \Re \int V u'' \bar{u} (1-\ch) dx  \\
&\leq \frac{C}{q(g(h))} \int V|uu'| dx+ E \int V(1-\ch) |u|^2dx + \int V (1-\ch) |fu| dx \\
& \leq \frac{C}{q(g(h))} \int V|uu'| dx + E h \int |fu|dx  + \int V(1-\ch) |fu|, \label{vonechiu}
\end{align}
where the last inequality followed from \eqref{wuestimate}. 

So now combining \eqref{vchiu} and \eqref{vonechiu} 
\begin{align}
\int V |u'|^2dx  \leq & \ed{C} g(h) \Udg \int \Theta |u'|^2 dx + E h \int |fu|dx + \int V(1-\ch) |fu|dx  \\
&+ \frac{C}{q(g(h))} \int V|uu'| dx.
\end{align}
Plugging this back into \eqref{wuu1}
\begin{align}
\int W|uu'| dx \leq \ed{E^{\frac{1}{2}}} h \int |fu| dx &+ h^{\frac{1}{2}} \left(\int |fu|dx \right)^{\frac{1}{2}} \bigg(\ed{ C g(h)} \Udg \int \Theta |u'|^2dx  \\
&\qquad \qquad+ \int V(1-\ch) |fu| dx + \frac{C}{q(g(h))} \int V|uu'|dx \bigg)^{\frac{1}{2}}.
\end{align}
Use Young's inequality for products and that $V \leq CW$ to absorb the final term back to obtain 
\begin{align}
\int W|uu'|dx \leq& \left( \ed{E^{\frac{1}{2}}} h + \frac{\ed{C}h}{q(g(h))} \right) \int |fu|dx \\&
+ h^{\frac{1}{2}} \left( \int |fu|dx \right)^{\frac{1}{2}} \bigg( \ed{C} g(h) \Udg \int \Theta |u'|^2dx  + \int V (1-\ch) |fu| dx \bigg)^{\frac{1}{2}}.
\end{align}
Then divide by $h$ and use Young's inequality for products on the $\ed{g(h)}\Udg \int \Theta |u'|^2 dx$ term to obtain the desired inequality. 
\end{proof}

To complete the proof of Proposition \ref{wideresolveprop} part 1) the second term on the right hand side of Lemma \ref{wuestimatelem} must be further estimated. 
\begin{proof}[Proof of Proposition \ref{wideresolveprop} part 1)]
Combining Lemmas \ref{multiplierlem} and \ref{wuestimatelem}, for all $\e>0$ there exists $C>0$ such that 
\begin{align*}
\int \Theta(|u'|^2+E |u|^2) dx &\leq C \int |f|^2 dx + C \left( \ed{E^{\frac{1}{2}}} + \frac{1}{q(g(h))} + \frac{g(h) \Udg}{h} \right) \int |fu|dx   \\
&+ h^{-\frac{1}{2}} \left( \int |fu|dx \right)^{\frac{1}{2}} \left( \int V (1-\chi) |fu| dx \right)^{\frac{1}{2}} + \e \int \Theta |u'|^2.
\end{align*}
For $\e$ small enough $\int \Theta |u'|^2 dx$ can be absorbed into the left hand side. Similarly using Young's inequality for products on $\int |fu|dx$ and absorbing the resultant $\int \ed{E} |u|^2dx$ terms back into the left hand side leaves 
\begin{align}
\int \Theta(|u'|^2 + E |u|^2 ) dx \leq &C \left( 1+\frac{1}{\ed{q(g(h))^2} E} + \frac{\ed{g(h)}^2\Udg^2}{h^2 E} \right) \int |f|^2 dx \\
& + \ed{C}h^{-\frac{1}{2}} \left( \int |fu|dx \right)^{\frac{1}{2}} \left( \int V (1-\ch) |fu| dx\right)^{\frac{1}{2}}.
\end{align}
Now note that when setting $g(h)= \tiroh$ for $R_{1}(z) = z q(z) U(z)$, as in Definition \ref{RMdefs} then
\begin{equation}\label{whyrd}
g(h) q(g(h)) U(g(h)) \simeq h,
\end{equation}
and so 
\begin{equation}\label{m1equiv}
\frac{1}{q(g(h))} \simeq \frac{g(h) \Udg}{h} \simeq h M_1(h).
\end{equation}
Therefore 
\begin{align}\label{finalwideeq}
\int \Theta (|u'|^2 + E |u|^2)dx  \leq &\left(1+ \frac{\ed{h^2 M_1(h)^2}}{E} \right) \int |f|^2 dx \\
& + h^{-\frac{1}{2}} \left( \int |fu| dx \right)^{\frac{1}{2}} \left( \int V (1-\ch) |fu| dx \right)^{\frac{1}{2}},
\end{align}
and to conclude it is enough to show that
\begin{equation}
h^{-\frac{1}{2}} \left( \int |fu| dx \right)^{\frac{1}{2}} \left( \int V (1-\ch) |fu| dx\right)^{\frac{1}{2}} \leq \left(\frac{C}{\ed{q(g(h))^2}}+1 \right) \int |f|^2 dx.
\end{equation}
Where we can work with the $q(g(h))^{-2}$ expression on the right hand side by \eqref{m1equiv}, and we do so to simplify the following calculations. We also drop the $E$ from the right hand side because it does not improve the estimate, and again to simplify calculations.

By linearity, we can split into different cases based on the location of the support of $f$. That is, if we have $N+2$ cases and $h$ dependent constants $\ed{b_0\leq b_1 \leq \cdots \leq b_N}$ to be specified later
\begin{description}
	\item[$0)$] $f_0 = f \mathbbm{1}_{V \leq b_0}$
	\item[$i)$] $f_i = f \mathbbm{1}_{b_{i-1} \leq V \leq b_i}$ for $1 \leq i \leq N$ 
	\item[$N+1)$] $f_{N+1} = f \mathbbm{1}_{b_N \leq V}$.
\end{description}
Then there exist $u_0, u_i, \ed{u_{N+1}} \in H^2(\Sb^1)$ such that $(-\p_x^2 + \frac{i}{h} W - E) u_k(x) = f_k(x)$ for $k=0, \ldots, \ed{N+1}$ \ed{see} \cite{KleinhenzWang2023}[Lemma 3.4] \ed{and $\sum_k u_k=u, \sum f_k = f$}. Now, if we have the estimate 
\begin{equation}
 \ltwo{u_k'}^2 + E \ltwo{u_k}^2 \leq C \left( \frac{1}{q^2(g(h))} + 1 \right) \ltwo{f_k}^2.
\end{equation}
Then, by the triangle inequality and since the $f_k$ have non-overlapping support
\begin{align}
\ltwo{u'}^2 + E \ltwo{u}^2 &\leq C \sum_k \ltwo{u'_k} + E \ltwo{u_k}^2 \leq C\left( \frac{1}{q^2(g(h))} + 1 \right) \sum_k \ltwo{f_k}^2 \\
&= C \left( \frac{1}{q^2(g(h))} + 1 \right) \ltwo{f}^2,
\end{align}
that is the desired estimate holds, and so it is enough to proceed by these cases. 

Now, define $t_i=3^{i+2}-4$, and set $b_i=h q(g(h))^{-t_i}$. We can now further estimate the quantities involved in the different cases. \\
\textbf{Case 0} Well, when $V(x) \leq b_0$ using \ed{Cauchy-Schwarz} and \eqref{wuestimate} \ed{we have}
\begin{align*}
\int |V(1-\ch) fu|dx  &\leq b_0^{\frac{1}{2}} \int V^{\frac{1}{2}} |fu|dx  \leq b_0^{\frac{1}{2}} \left( \int V|u|^2 dx\right)^{\frac{1}{2}} \left( \int |f|^2 dx \right)^{\frac{1}{2}} \\
&\leq b_0^{\frac{1}{2}} h^{\frac{1}{2}} \left( \int |fu| dx \right)^{\frac{1}{2}} \left( \int |f|^2 dx \right)^{\frac{1}{2}}.
\end{align*}
So then, applying Cauchy-Schwarz and Young's inequality for products, \ed{for any $\e>0$ we have}
\begin{align*}
h^{-\frac{1}{2}} \left( \int |fu| dx \right)^{\frac{1}{2}} \left( \int V (1-\ch) |fu|dx \right)^{\frac{1}{2}} &\leq \frac{b_0^{\frac{1}{4}}}{h^{\frac{1}{4}}} \left(\int |fu| dx \right)^{3/4} \left( \int |f|^2 dx \right)^{\ed{\frac{1}{4}}} \\
&\leq \frac{b_0^{\frac{1}{4}}}{h^{\frac{1}{4}}} \left( \int |f|^2 dx \right)^{5/8} \left( \int |u|^2 dx \right)^{3/8} \\
&\leq C \frac{b_0^{2/5}}{h^{2/5}} \int |f|^2 dx + \e \int |u|^2dx. 
\end{align*}
Plug this back into \eqref{finalwideeq}, absorbing the second term on the right hand side into the left hand, to obtain
$$
\int \Theta \left( |u'|^2 + E |u|^2 \right)dx \leq C \left( 1 + \frac{C}{q^2(g)} + \frac{b_0^{2/5}}{h^{2/5}} \right) \int |f|^2 dx.
$$
Now since $b_0 = h q^{-5}$ then $\frac{b_0^{2/5}}{h^{2/5}} = \frac{1}{q^2}$, and the desired estimate holds in this case. 

Note, we could also start case $0$ with $b_0 \leq g(h)$ and use that $(1-\ch) \equiv 0$ where $V \leq g$ to immediately see this term is 0. This would simplify this step, but would make determining how large $N$ should be taken more complicated. 

\textbf{Case $i$} When $b_{i-1} \leq V(x)$, \ed{applying Cauchy-Schwarz and \eqref{wuestimate}}
\begin{align*}
\int |f_i u | dx&\leq \frac{1}{\sqrt{b_{i-1}}} \int |f_i V^{\frac{1}{2}} u|dx \leq \frac{1}{\sqrt{b_{i-1}}} \left( \int |f_i |^2 dx \right)^{\frac{1}{2}} \left( \int V|u|^2 dx \right)^{\frac{1}{2}}\\
& \leq \frac{h^{\frac{1}{2}}}{\sqrt{b_{i-1}}} \left( \int |f_i|^2 dx\right)^{\frac{1}{2}} \left(\int |f_i u| dx\right)^{\frac{1}{2}}.
\end{align*}
So 
\begin{equation}\label{bim1lessv}
\ed{\left(\int |f_i u|dx \right)^{\frac{1}{2}} \leq \frac{C h^{\frac{1}{2}}}{b_{i-1}^{\frac{1}{2}}} \left( \int |f|^2 dx\right)^{\frac{1}{2}}.}
\end{equation}

When $b_{i-1} \leq V(x) \leq b_i$ then, \ed{applying Cauchy-Schwarz, \eqref{wuestimate}, and \eqref{bim1lessv} we have}
\begin{align}
\int V(1-\ch) |f_iu| dx &\leq b_i^{\frac{1}{2}} \int V^{\frac{1}{2}} |f_iu| dx \leq b^{\frac{1}{2}} h^{\frac{1}{2}} \left( \int |f_i |^2 dx \right)^{\frac{1}{2}} \left( \int |f_i u|dx  \right)^{\frac{1}{2}} \\
&\leq \frac{b_i^{\frac{1}{2}}h}{b_{i-1}^{\frac{1}{2}}} \int |f_i |^2 dx. \label{vlessbi}
\end{align}
So then combining \eqref{bim1lessv} and \eqref{vlessbi} 
\begin{equation}
h^{-\frac{1}{2}} \left( \int |f_i u|dx \right)^{\frac{1}{2}} \left( \int |V (1-\ch) fu|dx \right)^{\frac{1}{2}} \leq \frac{h^{\frac{1}{2}} b_i^{\frac{1}{4}}}{b_{i-1}^{3/4}} \int |f|^2 dx. 
\end{equation}
Now note that $b_i = b_{i-1}^3 h^{-2} q^{-8}$. Therefore $\frac{h^{\frac{1}{2}} b_i^{\frac{1}{4}}}{b_{i-1}^{3/4}} \simeq \frac{1}{q^2(g(h))}$,  so the desired estimate holds.

\textbf{Case $N+1$} Since $b_N \leq V(x)$ by the same argument as in case i) with $b_N$ replacing $b_{i-1}$ and a large constant $C$ replacing $b_i$ we have 
\begin{equation}
h^{-\frac{1}{2}} \left( \int |fu| \right)^{\frac{1}{2}} \left( \int V (1-\ch) fu \right)^{\frac{1}{2}} \leq \frac{h^{\frac{1}{2}}}{b_N^{3/4}} \int |f|^2.
\end{equation}
Now, by Definition \ref{RMdefs}, $M_1(h) \geq h^{1-\e}$ for some $\e>0$, and so by \eqref{m1equiv} we have $q(g(h)) \leq h^{\e}$. This along with $\lim_{N \ra \infty} t_N = \infty$, guarantees that there exists $N$ large enough so that $h \geq q^{3t_N+8}$. For such an $N$, because $b_N = h q(g(h))^{-t_N} \geq h^{2/3} q^{8/3}$, we have $\frac{h^{\frac{1}{2}}}{b_N^{3/4}} \leq \frac{1}{q^2(g(h))}$, and the desired estimate holds. 

\end{proof}
\begin{remark}
It is this final step of arguing in cases based on the support of $f_i$ where we require that $M_1(h) \geq h^{-1-\e}$ and that the argument fails for $W(x)=\exp(-(|x|-\sigma)_+^{\alpha})$. This is because $q(\tiroh) = \ln(h^{-1})^{-\frac{\alpha+1}{\alpha}}>h^{\e}$ for all $\e>0$. It is still possible to prove a resolvent estimate for this damping using this approach, but the resultant energy decay rate would be of the form $\ln(\ed{2}+t)^{p(\alpha)} t^{\e-1}$, for any $\e>0$ and some power $p(\alpha)$. This is slower decay than the $\ln(\ed{2}+t)^{p(\alpha)} t^{-1}$ rate achieved using part 2).
\end{remark}

\subsection{Proof of Proposition \ref{wideresolveprop} part 2 with Assumptions \ref{widebadassumption} and \ref{wideinversebassumption}}
We use the same approach of separately estimating $u$ in two regions defined in terms of the size of $W$ and then combining them, but the techniques used to obtain the estimates are slightly different.

For $g(h):(0,\infty) \ra (0,\infty)$, such that $\limh g(h)=0$, recall $\ch(x) = \chi\left( \frac{V(x)}{g(h)} \right)$ and that $u_1(x) = \ch(x) u(x), u_2(x) = (1-\ch(x)) u(x),$ so $u=u_1+u_2$. 

It is straightforward to estimate $u_2$, because it is supported where $V(x) \geq g(h)$ and so we can use the damping estimate \eqref{wuestimate}. 

Below and throughout the rest of this section integrals and $L^2$ norms are taken over $\Sb^1$ unless otherwise specified.
\begin{lemma}\label{utwolemma}
\begin{equation}
\ltwo{u_2} \leq \frac{C h^{\frac{1}{2}}}{g(h)^{\frac{1}{2}}} \left( \int |fu| dx \right)^{\frac{1}{2}}.
\end{equation}
\end{lemma}
\begin{proof}
Since $\sqrt{\frac{V(x)}{g(h)}} \geq 1$ on $\supp(1-\ch)$ and $\frac{V(x)}{C} \leq W(x)$ then $\frac{C^{\frac{1}{2}} W(x)^{\frac{1}{2}}}{g(h)^{\frac{1}{2}}} \geq 1$ on $\supp(1-\ch)$. \ed{Combining these facts with \eqref{wuestimate}, we obtain}
\begin{equation}
\ltwo{u_2} = \ltwo{(1-\ch) u} \leq \frac{C}{g(h)^{\frac{1}{2}}} \ltwo{W^{\frac{1}{2}} u} \leq \frac{C h^{\frac{1}{2}}}{g(h)^{\frac{1}{2}}} \left( \int |fu| dx \right)^{\frac{1}{2}}.
\end{equation}
\end{proof}

So it remains to estimate $u_1$. Note that $u_1\ed{=\ch u}$ solves an equation involving $u$
\begin{equation}\label{voneeq}
-\p_x^2 u_1+ \frac{i}{h} W u_1 - E u_1  = \ch f - 2 \ch' u' - \ch'' u = \ch f -2\p_x(\ch' u) + \ch''u.
\end{equation}
We will estimate $u_1$ in two separate cases. Consider $\sqrt{E} \geq \frac{1}{q(g(h))}$ and $ \sqrt{E} \leq \frac{1}{q(g(h))}$. When $E$ is large, the estimate follows from a standard ``black box" control argument. This is one of the pieces that requires us to have control over the second derivative of $V$, as well as the first. 

For ease of notation, in the remainder of this section we will frequently write $g(h)=g$, $q(g(h))=q$ and $p(g(h))=p$.
\begin{lemma}\label{uonemulargelemma}
If $\sqrt{E} \geq \frac{1}{q(g(h))}$ and $W$ satisfies Assumption \ref{widebadassumption},  then
\begin{align}
\ltwo{u_1} \leq& C \ed{q} \ltwo{f} \\
&+ C \left( \frac{\ed{q} g^{\frac{1}{2}}}{h} + g^{\frac{1}{2}} + \frac{\ed{q}}{g^{\frac{1}{2}}} \left(\frac{1}{q^2} + \frac{1}{p}\right) + \frac{1}{g^{\frac{1}{2}} q} \right) h^{\frac{1}{2}} \left( \int |fu| dx \right)^{\frac{1}{2}}.
\end{align}
\end{lemma}
\begin{proof}
By Lemma \ref{1dgcclemma}, which is proved in \cite[Prop 4.2]{Burq2020}, since $u_1$ satisfies \eqref{voneeq}, there exists $C>0$ such that for $\sqrt{E} \geq \ed{q(g(h))^{-1}}$,
\begin{align*}
\ltwo{u_1} &\leq C \left( \frac{1}{\sqrt{E}} \ltwo{\ch f - \frac{i}{h} W \ed{\ch u} + \ch'' u} + \nm{\p_x(\ch' u)}_{H^{-1}} + \ltwo{W \ed{\ch u}}\right)  \\
&\leq C \left( \ed{q} \ltwo{f} + \left( \frac{\ed{q}}{h} +1\right) \ltwo{W \ed{\ch u}} + \ed{q} \ltwo{\ch'' u} + \ltwo{\ch' u}\right). 
\end{align*}
Now using Lemma \ref{cutoffsizelemma} to control the size of $\ch'$ and $\ch''$, and since \orange{$W(x)^{\frac{1}{2}}\leq CV(x)^{\frac{1}{2}} \leq C g(h)^{\frac{1}{2}}$ } \ed{on $ \supp(\ch)$ we have}
$$
\ltwo{u_1} \leq C  \ed{q} \ltwo{f} + C \left( \frac{\ed{q} g^{\frac{1}{2}}}{h} + g^{\frac{1}{2}} + \frac{\ed{q}}{g^{\frac{1}{2}}} \left( \frac{1}{q^2} + \frac{1}{p} \right) + \frac{1}{g^{\frac{1}{2}}q} \right) \ltwo{W^{\frac{1}{2}} u}.
$$
Then apply \eqref{wuestimate} to the final term to draw the desired conclusion.
\end{proof}

When $\sqrt{E}$ is small, the same multiplier method we used to prove case 1) of Proposition \ref{wideresolveprop}  gives an estimate. This is the other piece of the argument that requires us to have control over the second derivative of $V$, as well as the first. This is also where we make use of $U(z) \geq \Vidz$. 
\begin{lemma}\label{uonemusmall}
Recall $U(z)$ is an increasing function with $\Vidz \leq U(z) <\ed{\nu}\ed{((0,\max V))}$ and from Lemma \ref{multiplierlem}
\begin{equation}
\Theta(x) = \begin{cases} \frac{1}{\Udg} & \{x: 0<V(x)<\digamma^{-1}(U(g(h)))\} \\
1 & \text{ elsewhere}.
\end{cases}
\end{equation}
If $\sqrt{E} \leq \frac{1}{q(g(h))}$ and $W$ satisfies Assumptions \ref{widebadassumption} and \ref{wideinversebassumption} is true, then 
\begin{align}
\int \Theta\left( |u_1'|^2 + E |u_1|^2  \right) dx &\leq  C \int |f|^2 dx \\
+&C \left( \frac{g \Vidg}{h} + \frac{h \Vidg}{g} (q^{-2}+p^{-1})^2  \right)\ed{ \int |fu| dx}.
\end{align}
\end{lemma}
\begin{proof}
From the proof of the multiplier estimate apply \eqref{multipliereq} to $u_1$ using \eqref{voneeq} to see 
$$
\int \Theta\left( |u_1'|^2 + E |u_1|^2 \right)  dx \leq 2\left| \Re \int b (\ch f- 2\p_x(\ch' u) + \ch'' u) \bar{u}_1'dx \right|   + \frac{C}{h} \int W|u_1 \bar{u}_1'| dx,
$$
where $b$ is defined in \eqref{bdef} and we used that $b'(x)=\Theta(x)$ on $\{V(x) \leq \digamma^{-1}(U(g(h)))\} \supseteq \{V(x) \leq \d g(h)\}$, which contains  $\supp(u_1) =\supp(u \ch)$. Then, rearranging absolute values
\begin{equation} \label{uonemulteq}
\int \Theta\left( |u_1'|^2 + E |u_1|^2 \right) dx \leq C \int \frac{1}{h} W |u_1 \bar{u}_1'| + |\ch f \bar{u}_1'| +  |\ch'' u \bar{u}_1'| dx + 2 \left| \Re \int b \p_x (\ch' u) \bar{u}_1' dx\right|. 
\end{equation}
To obtain the desired estimate, we will estimate the final term on the right hand side, and then estimate the resulting terms, using Lemma \ref{cutoffsizelemma} to control the size of derivatives of cutoffs. The final term can be estimated using integration by parts and writing $u_1 = \ch u$ 
\begin{align*}
\left| \Re \int b \p_x (\ch' u ) \ed{\bar{u}_1'} dx \right|&= \left|\Re \int b' \ch' u \bar{u}_1' + b \ch' u (\ch'' \bar{u} + 2 \ch' \bar{u}' + \ch \bar{u}'') dx \right| \\
&\leq C \int \Vidg^{-1}  | \ch' u \bar{u}_1'| +  |\ch' \ch''| |u|^2 + (\ch')^2 |u \bar{u}'| \\
&\qquad \qquad + E |\ch'| |u|^2 +  |\ch' u \bar{f}| dx,
\end{align*}
where we used that $|b| \leq C, |b'| \leq \Vidgh^{-1},$ and applied \eqref{fueq} to $u''$, taking the real part. In order to eventually absorb terms back from the right hand side, we actually prefer to have terms of the form $uu_1'$ instead of $uu'$. To perform this exchange, note that since $u_1' = \ch u' + \ch' u$, then $|u'| \leq |u_1'| + |\ch' u|$ on $\supp \ch'$. Therefore 
\begin{align}
\left| \Re \int b \p_x (\ch' u) dx \right| \leq &C \int \bigg(\Vidg^{-1}  |\ch'|+(\ch')^2 \bigg)|u \bar{u}_1'| \\
&\qquad +  \bigg(|\ch' \ch''|  +|\ch'|^3+ E |\ch'| \bigg)|u|^2 +  |\ch' u \bar{f}| dx.
\end{align}

\ed{Let $\mathbbm{1}_{\chi_h}$ be the indicator function for the support of $\ch$}. Now use $E \leq \frac{1}{q(g(h))^2}$ by assumption and plug \ed{the above inequality} back into \eqref{uonemulteq}, using that $|u_1| =|\ch u| \leq C \mathbbm{1}_{\ch} |u|$, $\ch' =\ch' \mathbbm{1}_{\ch},$ and \orange{$W \leq CV$}
\begin{align*}
\int \Theta\left( |u_1'|^2 + E |u_1|^2 \right) dx \leq C \int& \bigg(\frac{1}{h} V +|\ch''|+\Vidg^{-1} | \ch'|+|\ch'|^2  \bigg) \mathbbm{1}_{\ch} |u \bar{u}_1'| + |\ch f \bar{u}_1'| + \\
&+  \bigg(|\ch' \ch''| +|\ch'|^3 +\frac{\ed{|\ch'|}}{q^2}\bigg) |u|^2 +  |\ch' u \bar{f}| \ed{dx}.
\end{align*}
Now to estimate $\ch'$ and $\ch''$ apply Lemma \ref{cutoffsizelemma}, and note that $\supp u_1, \supp u_1' \subseteq \{V(x) \leq \d g(h)\}$ so \orange{$V(x)^{\frac{1}{2}} \leq \ed{C} g(h)^{\frac{1}{2}}$} there, \ed{so we have}
\begin{align}
\int \Theta\left( |u_1'|^2 + E |u_1|^2 \right) \leq C \int& \bigg( \frac{g^{\frac{1}{2}}}{h} +\frac{1}{g^{\frac{1}{2}}} (q^{-2} + p^{-1}) + \frac{1}{g^{\frac{1}{2}} q \Vidg} \bigg) V^{\frac{1}{2}}  \mathbbm{1}_{\ch} |u \bar{u}_1'|  \\
&+ \frac{1}{gq} (q^{-2} + p^{-1}) V |u|^2 \\
&+\frac{1}{g^{\frac{1}{2}} q} V^{\frac{1}{2}} |u \bar{f}|+ |\ch f \bar{u}_1'| dx.\label{uonepostcutoff}
\end{align}
Now use Young's inequality for products and that $\Theta(x) =\Vidgh^{-1}$ on $\supp (\ch W)=\{0<V<\d g(h)\} \subseteq \{0<V < \digamma^{-1}(\Vidgh)\}$ to estimate \ed{for any $\e>0$}
\begin{equation}
\int A V^{\frac{1}{2}} \mathbbm{1}_{\ch} |u u_1'| dx \leq \frac{\Vidg A^2}{2 \e} \int V |u|^2  dx + \frac{\e}{2} \int \Theta |u_1'|^2 dx.
\end{equation}
We can also apply Young's inequality directly to estimate \ed{for any $\e>0$}
\begin{align}
&\int |\ch f \bar{u}_1'| dx \leq \frac{1}{2 \e} \int |f|^2 dx + \frac{\e}{2} \int |u_1'|^2 dx \\
&\int \frac{1}{g^{\frac{1}{2}}q} V^{\frac{1}{2}} |u \bar{f}| dx \leq \int |f|^2 dx + \frac{1}{g q^2} \int V|u|^2 dx.
\end{align}
Apply these to \eqref{uonepostcutoff}, choosing $\e$ small enough to absorb the $|u_1'|^2$ terms back into the left hand side, to obtain 
\begin{align}
\int \Theta\left( |u_1'|^2 + E |u_1|^2 \right) dx \leq &  C \bigg(\frac{g \Vidg}{h^2} + \frac{\Vidg}{g} (q^{-2} + p^{-1})^2 \\
& \qquad + \frac{1}{g q^2 \Vidg} + \frac{1}{gq}(q^{-2} + p^{-1}) \bigg) \int V|u|^2 dx +C \int |f|^2 dx. 
\end{align}
Now, by Assumption \ref{wideinversebassumption} $\Vidg \geq q(g)$ so
\begin{align}
&\frac{1}{gq} (q^{-2} + p^{-1}) \leq \frac{q}{g}\left( \frac{1}{q^4} + \frac{1}{pq^2}\right) \leq \frac{\Vidg}{g} ( q^{-2}+p^{-1})^2\\
&\frac{1}{gq^2 \Vidg} \leq \frac{1}{gq^3} \leq \frac{q}{g} \left(q^{-2}+p^{-1} \right)^2 \leq \frac{\Vidg}{g} ( q^{-2}+p^{-1})^2.
\end{align}
So we have 
\begin{align}
\int \Theta\left( |u_1'|^2 + E |u_1|^2 \right) dx \leq& C \int |f|^2 dx  \\
&+C \left( \frac{g \Vidg}{h^2} + \frac{\Vidg}{g} (q^{-2}+p^{-1})^2  \right) \int V|u|^2 dx.
\end{align}
Now use $V \leq C W$ and apply \eqref{wuestimate} to estimate the $V|u|^2$ term, thereby obtaining the desired inequality. 
\end{proof}
We can now combine these separate estimates together to prove the desired resolvent estimate.
\begin{proof}[Proof of Proposition \ref{wideresolveprop} part 2)]

In part 2), we assume Assumptions \ref{widebadassumption} and \ref{wideinversebassumption}. We apply \eqref{umusmall} and Lemmas \ref{utwolemma}, \ref{uonemulargelemma}, and \ref{uonemusmall}, to obtain, for all $E \in \Rb$ 
\begin{align*}
\ltwo{u} \leq \ltwo{u_1} + \ltwo{u_2} \leq& (C+q) \ltwo{f} + C \bigg(\frac{h^{\frac{1}{2}}}{g^{\frac{1}{2}}} + \frac{g^{\frac{1}{2}}q}{h^{\frac{1}{2}}}  + \frac{h^{\frac{1}{2}} q}{g^{\frac{1}{2}}} (q^{-2} + p^{-1})+\frac{h^{\frac{1}{2}}}{g^{\frac{1}{2}} q} \\
& + \frac{g^{\frac{1}{2}} \Vidg^{\frac{1}{2}}}{h^{\frac{1}{2}}} + \frac{h^{\frac{1}{2}} \Vidg^{\frac{1}{2}}}{g^{\frac{1}{2}}}(q^{-2} + p^{-1})\bigg) \left(\int |fu| dx \right)^{\frac{1}{2}}.
\end{align*}
Now by Assumption \ref{wideinversebassumption}, $q(g) \leq \Vidg$, so $q(g) \leq \Vidg^{\frac{1}{2}}$ and $q(g)^{-1} \leq \Vidg^{\frac{1}{2}} q(g)^{-2}$. Therefore 
\begin{align*}
&\frac{h^{\frac{1}{2}}}{g^{\frac{1}{2}}} \leq  \frac{h^{\frac{1}{2}}\Vidg^{\frac{1}{2}}}{g^{\frac{1}{2}}}(q^{-2} + p^{-1}),  \\
&\frac{g^{\frac{1}{2}}q }{h^{\frac{1}{2}}} \leq  \frac{g^{\frac{1}{2}} \Vidg^{\frac{1}{2}}}{h^{\frac{1}{2}}},\\
&\frac{h^{\frac{1}{2}} q}{g^{\frac{1}{2}}} (q^{-2} + p^{-1}) \leq \frac{h^{\frac{1}{2}} \Vidg^{\frac{1}{2}}}{g^{\frac{1}{2}}} (q^{-2} + p^{-1}), \\
&\frac{h^{\frac{1}{2}}}{g^{\frac{1}{2}} q } \leq \frac{h^{\frac{1}{2}} \Vidg^{\frac{1}{2}}}{g^{\frac{1}{2}}}(q^{-2} + p^{-1}).
\end{align*}
Therefore, absorbing smaller terms 
\begin{equation}
\ltwo{u} \leq C \left(\frac{g^{\frac{1}{2}} \Vidg^{\frac{1}{2}}}{h^{\frac{1}{2}}} + \frac{h^{\frac{1}{2}} \Vidg^{\frac{1}{2}}}{g^{\frac{1}{2}}}(q^{-2} + p^{-1})\right) \left( \int |fu| dx \right)^{\frac{1}{2}} +  C \ltwo{f}.
\end{equation}
Note that $(q(z)^{-2} + p(z)^{-1}) \simeq \min(q^2(z), p(z))^{-1}$. Recall $R_2(z) = z \min(q^2(z), p(z))$, so setting $g(h) =\ti{R}_2^{-1}(h)$, we have $h \simeq g(h)(q(g(h))^{-2} + p(g(h))^{-1})^{-1}$. So  
\begin{equation}\label{whyr2}
\frac{g(h)^{\frac{1}{2}}}{h^{\frac{1}{2}}} \simeq \frac{h^{\frac{1}{2}}}{g(h)^{\frac{1}{2}}}\left(q(g(h))^{-2} + p(g(h))^{-1}\right).
\end{equation}
Plugging this back in gives
\begin{equation}
\ltwo{u} \leq C \frac{g(h)^{\frac{1}{2}} \Vidgh^{\frac{1}{2}}}{h^{\frac{1}{2}}} \left( \int|fu| dx \right)^{\frac{1}{2}}+ C \ltwo{f}.
\end{equation}
Now applying Young's inequality to products to $|fu|$ term, and absorbing the resultant $\e \ltwo{u}$ terms back into the left hand side gives 
\begin{equation}
\ltwo{u} \leq C \left(\frac{\ti{R}^{-1}_2(h) U(\d \ti{R}^{-1}_2(h))}{h} + 1\right) \ltwo{f} \leq C M_2(h) h \ltwo{f},
\end{equation}
which is exactly the desired estimate.
\end{proof}
\section{Proof of reduction to averaged damping}
\label{widegensec}
In this section we prove Theorem \ref{averagedresolventthm}, which says roughly that, averaged damping resolvent estimates imply a resolvent estimate for the original operator. We also prove Theorem \ref{widegenresolvethm}, which gives energy decay rates for general damping satisfying non-polynomial derivative bound conditions, as a corollary of Theorem \ref{averagedresolventthm}.  As mentioned above Theorem \ref{averagedresolventthm} is a generalization of \cite{Sun23} to non-polynomial $\rho(h)$. We follow the same averaging and normal form approach, and so the argument is similar. 

To begin we prove that averaging maintains membership in $\Dc^{9,\frac{1}{4}}(\T^2)$ as well as the $q$ and $p$ from Assumptions \ref{noninvarassumption} and \ref{widegenbadassumption}.
\begin{lemma}\label{averagingdbclemma}
\begin{enumerate}
	\item Suppose $W \in \Dc^{9,\frac{1}{4}}(\T^2)$, then for all $v \in \Sb^1$ periodic, $A_v(W) \in \Dc^{9,\frac{1}{4}}(\T^2)$
	\item Suppose $W$ satisfies Assumption \ref{noninvarassumption}, then for all $v \in \Sb^1$ periodic, there exists $\e_2>0$ such that
	\begin{equation}
	|\nabla A_v(W)| \leq \frac{A_v(W)}{q(A_v(W))}, \quad \text{ when } A_v(W) \leq \e_2.
	\end{equation}
	\item Suppose $W$ satisfies Assumption \ref{noninvarassumption} and \ref{widegenbadassumption}, then there exists $\e_2>0$ such that 
	\begin{equation}
	|\nabla A_v(W)| \leq \frac{A_v(W)}{q(A_v(W))}, \quad |\nabla^2 A_v(W)| \leq \frac{A_v(W)}{p(A_v(W))}, \quad \text{ when } A_v(W) \leq \e_2.
	\end{equation}
\end{enumerate}
\end{lemma}
\begin{proof}
Fix $v \in \Sb^1$ periodic, with period $T_v$, and $\e_1$ from Assumption \ref{noninvarassumption}. To begin, we claim there exists $\e_2 \leq \e_1$ such that for all $z$ with $A_v(W)(z) \leq \e_2$ then $W(z+tv) \leq \e_1$ for all $t \in [0,T_v]$. 

To see this claim, first by uniform continuity of $W$ on $\T^2$, there exists $\d>0$ such that $|x-y|<\d$ implies $|W(x)-W(y)| <\frac{\e_1}{2}$. Now assume the claim is not true and we will derive a contradiction. So there exists $z_j, t_j$ with $A_v(W)(z_j) \leq \frac{1}{j}$ and $W(z_j+t_j v) > \e_1$. Choose $J > \frac{T_v}{\d \e_1}$, then $W(z_J + tv) > \frac{\e_1}{2}$ for all $t$ with $|t_J-t| < \d$, by the uniform continuity and since $|v|=1$. Therefore
$$
A_v(W)(z_j) = \frac{1}{T_v} \int_0^{T_v} W(z_J +tv) dt \geq \frac{1}{T_v} \d \e_1 > \frac{1}{J},
$$
which is a contradiction, so the claim is true. 

Thus we can work with $A_v(W) \leq \e_2$, ensuring $W \leq \e_1$ and thus $W$ satisfies the derivative bound condition along all the trajectories we are working with.
Now \ed{recalling $r_1(z) = \frac{z}{q(z)}$ from Assumption \ref{noninvarassumption}}
\begin{align*}
|\nabla A_v(W)(z)| &\leq \frac{1}{T_v} \int_0^{T_v} |\nabla W(z+tv)| dt \leq \frac{1}{T_v} \int_0^{T_v} \frac{W(z+tv)}{q(W(z+tv))} dt \\
& = \frac{1}{T_v} \int_0^{T_v} \ed{r}_1(W(z+tv)) dt \leq \ed{r}_1 \left( \frac{1}{T_v} \int_0^{T_v} W(z+tv) dt \right)\\
&= \ed{r}_1 (A_v(W)(z)) = \frac{A_v(W)(z)}{q(A_v(W)(z))},
\end{align*}
where we were able to exchange the order of the integral and the application of $\ed{r}_1$ using Jensen's inequality, since $\ed{r}_1$ is concave. 

An analogous argument shows that $|\nabla^2 A_v(W)(z)| \leq \frac{A_v(W)(z)}{p(A_v(W)(z))}$ and that $W \in \Dc^{9,\frac{1}{4}}(\T^2)$ implies $A_v(W) \in \Dc^{9,\frac{1}{4}}$. 
\end{proof}

\subsection{Proof of Theorem \ref{widegenresolvethm}} \label{wideresolveproof}
Now we prove Theorem \ref{widegenresolvethm} using the 1 dimensional resolvent estimates of Theorem \ref{wideresolvethm} and assuming the relation between averaged resolvent estimates and resolvent estimates of Theorem \ref{averagedresolventthm}.
\begin{proof}[Proof of Theorem \ref{widegenresolvethm}]
Consider $v \in \Sb^1$ \ed{periodic}. In case 1), since $W$ satisfies Assumption \ref{noninvarassumption}, then Lemma \ref{averagingdbclemma} part 2 guarantees that $A_v(W)$ satisfies Assumption \ref{wideassumption} with the same $q$. Therefore by Theorem \ref{wideresolvethm} case 1) with $U(z)=C$ \ed{for each periodic $v \in \Sb^1$ there exists $h_v>0$ such that for all $h \in (0,h_v)$ we have}
\begin{equation}
\nm{\left(\frac{i}{h} - \Ac_v\right)^{-1}}_{\Lc(\Hc)} \leq C  M_1(h).
\end{equation}
Then by Theorem \ref{averagedresolventthm} \ed{there exists $h_0>0$ such that for all  $h \in (0,h_0)$ we have} 
\begin{equation}
\nm{\left(\frac{i}{h} - \Ac\right)^{-1}}_{\Lc(\Hc)} \leq C M_1(h),
\end{equation}
which is exactly case 1. 

In case 2), $W$ satisfies Assumption \ref{widegenbadassumption}, and Lemma \ref{averagingdbclemma} part 3 guarantees that $A_v(W)$ satisfies \ed{Assumptions \ref{wideassumption} and \ref{widebadassumption}} with the same $q$ and $p$. So by Theorem \ref{wideresolvethm} case 2) \ed{for each periodic $v \in \Sb^1$ there exists $h_v>0$ such that for all $h \in (0,h_v)$ we have}
\begin{equation}
\nm{\left(\frac{i}{h} - \Ac_v\right)^{-1}}_{\Lc(\Hc)} \leq C  M_2(h),
\end{equation}
Then by Theorem \ref{averagedresolventthm} \ed{there exists $h_0>0$ such that for all  $h \in (0,h_0)$ we have} 
\begin{equation}
\nm{\left(\frac{i}{h} - \Ac\right)^{-1}}_{\Lc(\Hc)} \leq C M_2(h),
\end{equation}
which is exactly case 2. 
%
%
%

\end{proof}

\subsection{Proof of Theorem \ref{averagedresolventthm}: Start of contradiction argument}

Assume Theorem \ref{averagedresolventthm} does not hold and we will obtain a contradiction. By the semigroup equivalence of Lemma \ref{resolvequivlem} we \ed{can} assume the averaged resolvent estimates \eqref{assumed1dresolventeq} hold for all $v \in \Sb^1$ periodic, and that there exist sequences $\{h_n\} \in \Rb_{>0}$, $\{u_{h_n}\} \in \ed{H}^2(\T^2)$ with $h_n \ra 0$ and  
\begin{equation}\label{eqdwecontra}
(-h_n^2 \Delta + i h_n W - 1) u_{h_n} = f_{h_n}= o_{L^2}(h_{n}^2 \rho(\ed{h_n})),
\end{equation}
with 
\begin{equation}\label{quasimodel2eq}
\ltwo{u_{h_n}}=1.
\end{equation}
\ed{Recall that throughout this proof $\rho(h_n)^{-1}=o(h_n^{-\frac{1}{3}})$.}
For ease of notation we will drop the subscript $n$ and write $h$ and $u_h$, we also define an additional semiclassical scale in terms of $\dh$. Specifically, define $\hbar = \dh^{\frac{1}{2}} h^{\frac{1}{2}}$.

We begin with a version of the standard a priori estimates in this contradiction setting. 
\begin{lemma}\label{aprioricontralemma}
\begin{enumerate}
	\item 
	$$
	\ltwo{W^{\frac{1}{2}} u_h} = o(h^{\frac{1}{2}} \dh^{\frac{1}{2}}) = o(\hbar)
	$$
	\item 
	$$
	\ltwo{h\nabla u_h}^2 - \ltwo{u_h}^2 =o(h^2\dh).
	$$
\end{enumerate}
\end{lemma}
\begin{proof}
Multiply \eqref{eqdwecontra} by $\bar{u}_h$ and integrate by parts to obtain
\begin{equation}
\ltwo{h \nabla u_h}^2 - \ltwo{u_h}^2 + ih\ltwo{\sqrt{W} u_h}^2 = \<f_h, u_h\>_{L^2}.
\end{equation}
Take the imaginary part and the real part to obtain 1) and 2) respectively. 
\end{proof}
Since $\{u_h\}$ is bounded in $L^2(\T^2)$, up to replacement by a subsequence, there exists a semiclassical defect measure $\mu$ on $T^* \T^2$, such that for any symbol $a \in \Cs(T^*\T^2)$ 
\begin{equation}
\limh \<\Op_h(a) u_h, u_h\>_{L^2} = \int_{T^* \T^2} a(z,\zeta) d\mu.
\end{equation}
For details on the existence of this measure see \cite[Chapter 5]{Zworski2012}. 

We recall some standard facts about defect measures. Let $\phi_t$ be the geodesic flow on $T^* \T^2$.
\begin{lemma}\label{defectmeasurelem} 
The defect measure \ed{has total mass 1. Furthermore}, the defect measure is supported on the characteristic set of $\ed{-h^2\Delta-1}$, that is 
\begin{equation}
\supp(\mu) \subseteq \{(z,\zeta) \in T^* \T^2; |\zeta|=1\} := S^* \Tb^2.
\end{equation}
Additionally, letting $\omega=\{z \in \Tb^2; W(z)>0\}$ be the set where the damping is positive and $\mathcal{G}=\{(x,\xi) \in S^* \Tb^2: \phi_t(x,\xi) \in \omega \text{ for some } t \in \Rb\}$ be the set of points which are geometrically controlled by the damped set, then 
\begin{equation}
	\mu|_{\mathcal{G}} = 0. 
\end{equation}
\end{lemma}
\begin{proof}
The first statement follows from \ed{the definition of the defect measure, and the $L^2$ normalization of the $u_h$. The second statement is exactly Theorem 5.3 of \cite{Zworski2012} since $u_h$ is an $o_{L^2}(h)=o_{L^2}(1)$ quasimode.  Note that $ih Wu_h$ is treated as a perturbation of size $o(h \hbar)$ using Lemma \ref{aprioricontralemma}. }

\ed{To see the third statement, first note by Lemma \ref{aprioricontralemma} that the defect measure is $0$ where the damping is positive, that is 
\begin{equation}
	\mu|_{\omega \times \Rb^2} = 0.
\end{equation}
Then by \cite[Theorem 5.4]{Zworski2012} the defect measure is invariant under the geodesic flow
\begin{equation}
	\phi_t^* \mu = \mu. 
\end{equation}
The desired conclusion follows immediately from these two facts.}

\end{proof}

\subsection{Reduce to periodic directions}\label{toruscoordinatessubsection}
In this subsection we follow the approach of \cite[Section 6]{AL14} and \cite[Section 2.2]{Sun23}. See also \cite{BurqZworski2012, AM14}.

We will obtain our contradiction by proving that $\mu \equiv 0$, which contradicts
$$
1= \ed{\limh} \ltwo{u_h}^{\ed{2}} = \ed{\limh} \<u_h, u_h\>_{L^2} = \int_{T^* \T^2} 1 d\mu. 
$$
To that end, we consider $\mu$ along each direction $\zeta \in \Sb^1$, interpreted as a frequency variable in $S^* \T^2$. 

Recall our decomposition of $\Sb^1$ into rational directions $\Qc$ and irrational directions $\Rc$, from the statement of Theorem \ref{averagedresolventthm}. Since the geodesic flow starting at $x \in \T^2$ with irrational direction $v_0 \in \Rc$ is dense in $\T^2$, it necessarily passes \ed{through $\omega$, so by Lemma \ref{defectmeasurelem},  $\mu|_{\T^2 \times \Rc}=0$. }

Therefore
$$
\mu =\mu|_{\T^2 \times \Qc} = \sum_{\ed{\zeta}_0 \in \Qc} \mu_{\zeta_0}, \quad \mu_{\zeta_0} = \mu|_{\T^2 \times \zeta_0}.
$$
So to see $\mu=0$, and obtain the desired contradiction, it is enough to show $\mu_{\zeta_0}=0$, for each $\ed{\zeta_0} \in \Qc$. 

Now focusing on a single $\zeta_0 = \frac{(p_0, q_0)}{\sqrt{p_0^2 + q_0^2}} \in \Qc$ at a time, we can perform a change of coordinates to make $\zeta_0 = (0,1)$, which will simplify the later analysis and notation. This does not interfere with us taking an arbitrary $\zeta_0 \in \Qc$. 

Let $\ed{\Lambda}$ be the rank 1 submodule of $\Zb^2$ generated by $k_0 = (p_0, q_0)$, and let $\ed{\Lambda}^{\bot}$ be vectors in $\Rb^2$ \ed{perpendicular} to $\ed{\Lambda}$. That is 
$$
\ed{\Lambda}^{\bot} := \{\zeta \in \Rb^2; \zeta \cdot \ed{k_0} =0\}.
$$
Denote by 
$$
\T^2_{\Lambda} := (\Rb \ed{\Lambda} / (2\pi \ed{\Lambda}) ) \times (\ed{\Lambda}^{\bot} / ( (2\pi \Zb)^2 \cap \ed{\Lambda}^{\bot}) ).
$$

Then there is a natural smooth covering map $\pi_{\Lambda}: \T^2_{\Lambda} \ra \T^2$ of degree $p_0^2 + q_0^2$, which we also extend to the cotangent bundles $\pi_{\Lambda}: T^* \T^2_{\Lambda} \ra T^* \T^2$. 

The pullback of a $2\pi \times 2 \pi$ periodic function $f: \Rb^2 \ra \Rb$ is $\tau \times \tau$ periodic, for $\tau=2\pi\sqrt{p_0^2+q_0^2}$. That is  
\begin{equation}
(\pi^*_{\Lambda} f)( x + k \tau, y + l \tau ) = (\pi^*_{\ed{\Lambda}} f)(x,y) \quad k,l \in \Zb, (x,y) \in \Rb^2.
\end{equation}

\begin{figure}[h]
\begin{minipage}{.4 \textwidth}
\begin{tikzpicture}
[
declare function={
  a=2;
}]
\draw (-a,-a) rectangle (a,a);
\draw [imayou, -, very thick] (-a,-a) --(a,0);
\draw[imayou, -, very thick] (-a,0)--(a,a);
\draw[persred, -, very thick] (a,-a)--(0,a);
\draw[persred, -, very thick] (0,-a)--(-a,a);
\draw[sand, pattern=north east lines, pattern color=sand] (.3*a,-.1*a) ellipse (a/6 and a/6);
\node[text=imayou] at (0,-a*1.2) {$k_0=(2,1)$};
\node[text=persred] at (0,a*1.2) {$k_0^{\bot}=(-1,2)$};
\end{tikzpicture}
\end{minipage}
\begin{minipage}{.4 \textwidth}
\begin{tikzpicture}[rotate=26.57,
transform shape,
declare function={
  x=2;
  y=1;
  a=o*sqrt(5);
  o=1;
}]
\draw (-a,-a) rectangle (a,a);
\draw[imayou, -, very thick] (-a,-a)--(a,-a);
\draw[imayou,-,very thick] (-a,-3/5*a)--(a,-3/5*a);
\draw[imayou,-,very thick] (-a,-1/5*a)--(a,-1/5*a);
\draw[imayou,-,very thick] (-a,1/5*a)--(a,1/5*a);
\draw[imayou,-,very thick] (-a,3/5*a)--(a,3/5*a);
\draw[imayou, -, very thick] (-a,a)--(a,a);
\draw[persred,-, very thick] (-a,-a)--(-a,a);
\draw[persred,-, very thick] (-3/5*a,-a)--(-3/5*a,a);
\draw[persred,-,very thick](-1/5*a,-a)--(-1/5*a,a);
\draw[persred, -, very thick](1/5*a,-a)--(1/5*a,a);
\draw[persred, -, very thick](3/5*a,-a)--(3/5*a,a);
\draw[persred, -, very thick](a,-a)--(a,a);
\draw[sand, pattern=north east lines, pattern color=sand] (-2/5*a+.2*o,-4/5*a-.2*o) ellipse (o/6 and o/6);
\draw[sand, pattern=north east lines, pattern color=sand] (4/5*a+.2*o,-2/5*a-.2*o) ellipse (o/6 and o/6);
\draw[sand, pattern=north east lines, pattern color=sand] (.2*o,-.2*o) ellipse (o/6 and o/6);
\draw[sand, pattern=north east lines, pattern color=sand] (-4/5*a+.2*o,2/5*a-.2*o) ellipse (o/6 and o/6);
\draw[sand, pattern=north east lines, pattern color=sand] (2/5*a+.2*o,4/5*a-.2*o) ellipse (o/6 and o/6);
\node[text=imayou] at (0,-a*1.1) {$\ed{\Lambda}$};
\node[text=persred] at (a*1.2,0) {$\ed{\Lambda^{\bot}}$};
\end{tikzpicture}
\end{minipage}
\caption{The original torus $\T^2$ on the left with $\ed{k_0=(2,1)}$.  The torus $\T^2_{\Lambda}$ (scaled down by $\sqrt{5}$ to fit) on the right. The purple (red) line indicates a single periodic orbit on $\T^2$ parallel to \ed{$k_0$} (resp. \ed{$k_0^{\bot}$}) which has as pre-images from $\pi_{\Lambda}$, multiple periodic orbits with the same period on $\T^2_{\Lambda}$. The yellow region indicates a specific location on $\T^2$, and the pre-image of those locations end on $\T^2_{\Lambda}$.}
\end{figure}

We can identify the sequence $\{u_h\} \subset L^2(\T^2)$ with $\{\pi^*_{\Lambda} u_h\} \subset L^2(\T_{\Lambda}^2)$ by pulling back via $\pi_{\Lambda}$ to $\T^2_{\Lambda}$. Then \ed{on $\ed{\T}^2_{\Lambda}$, $\zeta_0 = \frac{k_0}{|k_0|} $ becomes $ (0,1)$}. Furthermore, the semiclassical defect measure $\mu$ on $T^* \T^2$ associated to $(u_h)$, is exactly the pushforward of the semiclassical defect measure associated to $(\pi_{\Lambda}^* u_h)$. 

Because we fix a periodic direction $\zeta_0$ and consider the semiclassical limit as $h \ra 0$, the period of the torus $\T^2_{\Lambda}$ does not \ed{affect} the following analysis. Therefore, we will perform these described pull back to work on $\T^2_{\Lambda}$, but write $\T^2$ instead of $\T^2_{\Lambda}$. Similarly we will write $z=(x,y), \zeta=(\xi, \eta)$ as the position and frequency variables, respectively, on $T^* \T^2_{\Lambda}$, and assume the period is $2\pi$. 

One thing which does change is that, because the covering map has degree $p_0^2+q_0^2$, the pre-image of the damped set $\pi^{-1}_{\Lambda}(\{W>0\})$, i.e. what we will call our damped set from now on, is a union of $p_0^2 + q_0^2$ copies of $\{W>0\}$ on $\T^2_{\Lambda}$. However, this does not pose any major issues. For instance if the set $\{W>0\}$ was a rectangle (resp. locally strictly convex with positive curvature or a superellipse), then the copies of $\{W>0\}$ on $\T^2_{\Lambda}$ are disjoint and still rectangles (resp. locally strictly convex with positive curvature or superellipses). 

\subsection{Conclusion of proof of Theorem \ref{averagedresolventthm}}
By the previous subsection, we may assume $\zeta_0=(0,1)$ and it is enough to show $\mu_{\zeta_0}=0$.

For the null bicharacteristic flow $\phi_t$ on $T^* \T^2$ and $\zeta \in \Sb^1$, we will also write $\phi_t(\zeta): \T^2 \ra \T^2$ for the projection of $\phi_t$ onto the position variables.

Fix some $c_1>0$. We may assume the $x$ projection of the set where $W$ is greater than or equal to $c_1$ contains an interval. That is, after shifting coordinates, there exists $\sigma_0$ small, such that 
$$
I_0 = (-\sigma_0, \sigma_0) \subset \pi_x (\{z; W(z) \geq c_1 >0\}),
$$
and so the cylinder over $I_0$ is contained in the vertical flow-out of the positive set of the damping.
\begin{equation}
\omega_0 := I_0 \times \T_y \subset \bigcup_{t \in [0,2\pi]} \phi_t (\zeta_0) \{W>0\}.
\end{equation}

\begin{figure}[h]
\centering
\begin{tikzpicture}
[
declare function={
  a=2;
}]
\draw (-a,-a) rectangle (a,a);
\draw[persred, ->, very thick] (-3/4*a,-a)--(-3/4*a,a);
\node[text=persred] at (-3/4*a,11/10*a) {$\zeta_0=(0,1)$};
\draw[persred, ->, very thick] (-3/4*a,-a)--(-3/4*a,9/10*a);

\draw[imayou, pattern=north east lines, pattern color=imayou] (0,-a) rectangle (a/2,a);
\node[text=imayou] at (a/4,-a*1.2) {$\omega_0$};
\draw[sand, pattern=north west lines, pattern color=sand] (.3*a,-.1*a) ellipse (2/3*a and a/2);
\node[text=sand]  at (-7/20*a,9/20*a) {$\{W>0\}$};

\end{tikzpicture}
\caption{Note that for nearly vertical trajectories, $\{W>0\}$ geometrically controls $\omega_0$.}
\end{figure}

Because of the continuity of the flow, $\{W>0\}$  also geometrically controls $\omega_0$ along nearly vertical trajectories. That is,
 there exist \ed{$T_0>0,$ and sufficiently small $ \e_0, c_0>0$}, so that for any $(z_0, \zeta) \in S^*\omega_0$ with $|\zeta-\zeta_0| \leq \e_0$, then
\begin{equation}
\int_0^{T_0} W \circ \phi_t(z_0, \zeta) dt \geq c_0 >0.
\end{equation}

Now, because we are interested in studying $\mu_{\zeta_0}$, we will frequency localize within this $\e_0$ of  $\zeta_0$. This will ensure that in the Proof of Proposition \ref{normalformquasiprop} the flowout of our quasimodes are still geometrically controlled by $\{W>0\}$ as dictated by the above inequality. 

Now, pick $\psi_0 \in \Cs([-2,2], [0,1])$ with $\psi_0(\xi)=1$ on $|\xi| <1$, and let $u_h^1 = \psi_0 \left( \frac{hD_x}{\e_0}\right) u_h$. Then 
\begin{equation}\label{eq:uh1def}
(-h^2 \Delta+ihW-1)u_h^{1} = f^1_h := \psi_0\left(\frac{hD_x}{\e_0}\right) \ed{f_h} - i h\left [\psi_0\left(\frac{hD_x}{\e_0}\right), W\right] u_h.
\end{equation}
It can be seen that this $u_h^1$ is still a quasimode of the same order. By \cite[Lemma 3.1]{Sun23}
\begin{lemma}
Since $W \in C^2(\T^2)$ and \ed{$h^{\frac{1}{3}} \dh^{-1}=o(1)$, so in particular $h^{\frac{1}{2}} \dh^{-1} =o(1)$, we have}
\begin{equation}
\ltwo{f_h^1} = o(h^2 \dh).
\end{equation}
Therefore from the proof of Lemma \ref{aprioricontralemma} 
\begin{equation}
\ltwo{W^{\frac{1}{2}} u^1_h} = o(h^{\frac{1}{2}} \dh^{\frac{1}{2}}).
\end{equation}
Finally, there exists a semiclassical defect measure $\mu_1$ of $u^1_h$, and 
\begin{equation}
\mu_1 = \left|\psi\left(\frac{\xi}{\e_0}\right)\right|^2 \mu.
\end{equation}
\end{lemma}
Because $\mu_1 =  |\psi(\frac{\xi}{\e_0})|^2 \mu$, then $\mu_{\zeta_0} = \mu_1|_{\zeta=\zeta_0}$. Therefore, in order to establish that $\mu_{\zeta_0}=0$, and thus that $\mu=0$, it is enough to show $\ltwo{u_h^1}=o(1)$, \ed{so that $\mu_1=0$}.

To do so, we invoke the normal form construction of \cite{Sun23}, see also \cite{BurqZworski2012}, \cite{Sjostrand2000, Hitrik2002}, to convert the quasimodes $u^1_h$ into quasimodes $v_h$ of a $y$-invariant equation. We then control these quasimodes $v_h$ using the assumed  1-d resolvent estimates \eqref{assumed1dresolventeq}.

\begin{proposition}\label{normalformquasiprop}
If $W \in \Dc^{9,\frac{1}{4}}(\T^2)$ and $h^{\frac{1}{3}} \dh^{-1} =o(1)$, then there exists $\kappa(x) \in W^{1, \infty}(\T_x), b(\eta) = -\frac{\psi_1(\eta)}{2 \eta}$, with $\psi_1(\eta) \in \Cs(\Rb, [0,1])$ supported on $|\eta\pm 1| \leq \frac{1}{2}$, and identically $1$ on $|\eta \pm 1|\leq \frac{1}{4}$, such that there exists a sequence of quasimodes $\vht$ of 
\begin{equation}\label{vhtdef}
(-h^2\Delta + i h A_{\zeta}(W)-1) \vht + i h \kappa(x) A_{\zeta}(W)^{\frac{1}{2}} b(hD_y) h D_x \vht =: r_h = o_{L^2}(h^2 \dh),
\end{equation}
with 
\begin{equation}
\ltwo{u_h^1} \leq C \ltwo{\vht} + O(\dh),
\end{equation}
and 
\begin{equation}
\ltwo{A_{\zeta}(W)^{\frac{1}{2}} \vht} = o(\hbar).
\end{equation}
\end{proposition}
\begin{proof}
This follows from \cite[Proposition 3.5 and 5.6]{Sun23}. To provide some more detail, to begin we again frequency localize $u_h^1$, this time at $\hbar=h^{\frac{1}{2}} \dh^{\frac{1}{2}}$ scale. This is done so that in \cite[Proposition 5.6 and Lemma 5.7]{Sun23}, $hD_x$ on $|\xi|\leq 1$ can be replaced, as an $O_{L^2}(1)$ $h$-semiclassical operator, by $h \hbar^{-1} \hbar D_x$ on $|\xi|\leq 1$ as an $O_{L^2}(h \hbar^{-1})$ $\hbar$-semiclassical operator. 

In particular, let $\psi \in \Cs([-2,2], [0,1])$ satisfy $\psi=1$ on $|x|\leq 1$, then consider 
$$
v_h = \psi(\hbar D_x) u^1_h, \quad w_h = (1-\psi(\hbar D_x)) u^1_h.
$$
These $v_h$ and $w_h$ are still quasimodes of order $o_{L^2}(h^2 \dh)$ \cite[Lemma 3.2]{Sun23} and $\ltwo{w_h} = O(\dh)$ \cite[Proposition 3.5]{Sun23}. Note, it is in the proof of Proposition 3.5 that the size of $\e_0$ as defined in relation to the Geometric Control Condition is used. 

Thus $\ltwo{u_h^1} = \ltwo{v_h} + O(\dh)$. The $\vht$ quasimodes will be constructed from the $v_h$ quasimodes. By \cite[Proposition 5.2 and Proposition 5.6]{Sun23}, there exists $G$ an invertible bounded operator on $L^2$ such that $\vht=Gv_h$ \ed{satisfies the desired quasimode equation} and $\ltwo{A(W)^{\frac{1}{2}} \vht} = o(\hbar)$. 

Finally, since $v_h$ and $\vht$ are related by an invertible map, $\frac{1}{C} \ltwo{\vht} \leq \ltwo{v_h} \leq C \ltwo{\vht}$ and so indeed $\ltwo{u_h^1} \leq C \ltwo{\vht} + O(\dh)$.
\end{proof}
\begin{remark}\label{Gardingrmk}
\begin{enumerate}	
	\item  Note the requirement of $h \dh^{-\frac{1}{3}}=o(1)$ in \cite[Lemma 5.5]{Sun23} is actually a typo, it comes from ensuring, at the end of the proof of Lemma 5.5 that 
\begin{equation}
(h\hbar^{-1})^{\ed{4}} + h^{4/3} = o(\hbar^2) = o(h\dh),
\end{equation}
which is satisfied when $h^{\frac{1}{3}}\dh^{-1}= o(1)$.

	\item The requirement that $|\nabla W| \leq C W^{3/4}, |\nabla^2 W| \leq C W^{\frac{1}{2}}$ is used throughout the proof to replace these derivatives of $W$ in symbol expansions by the original damping, to estimate terms using Lemma \ref{aprioricontralemma}. This is frequently done using the Sharp \ed{G\aa rding} inequality, which is what requires $W \in W^{9, \infty}(\T^2)$. Specifically from \cite[Proposition 5.1]{Sjostrandwiener}, to apply the sharp \ed{G\aa rding} inequality to terms involving $\nabla^2 W$, we need $\nabla^2 W \in W^{7,\infty}(\T^2)$. Note this is also why we require one more derivative of regularity on $W$ than \cite{AL14}. They also rely on this same sharp \ed{G\aa rding} inequality result, but only apply it to $W$ and $\nabla W$. 

	\item The other times the number of $L^{\infty}$ derivatives of $W$, call it $m$, are used in \cite{Sun23} are to ensure 
\begin{equation}
o(\hbar^m) = o(h^3) \qquad o(h^m \hbar^{-m}) = o(h \hbar^{2}),
\end{equation}
in the proofs of Lemmas 5.4, 5.5 and \ed{Proposition 5.6}. Since we are assuming $\dh^{-1} h^{\frac{1}{3}} =o(1)$, these are guaranteed when $m \geq 6, m \geq 7$ respectively. 
\end{enumerate}
\end{remark}

We now must connect the assumed 1-d resolvent estimates to resolvent estimates for these quasimodes $v_{h}^{(2)}$. 

To do so, suppose $v_{h,E} \in H^2(\T_x)$ solves 
\begin{equation}\label{vhedef}
-h^2 \p_x^2 v_{h,E} - E v_{h,E} +ih \ti{W}(x) v_{h,E} + h^2 \kappa_{h,E} \ti{W}(x)^{\frac{1}{2}} \p_x v_{h,E} = r_{h,E},
\end{equation}
where $(\kappa_{h,E})_{h>0,E\in\Rb}$ is a uniform bounded family in $W^{1,\infty}(\T, \Rb)$.  Then our assumed 1-d resolvent estimates, give resolvent estimates for these $v_{h,E}$.
\begin{proposition} \label{1dresolveprop}
Assume $\ti{W} \in \Dc^{9,\frac{1}{4}}(\T^2)$,\ed{ $\ti{W}(x,y)=\ti{W}(x)$ and $\ti{W}$ is} nonnegative. If there exists $h_1, C>0$ and $\dh :(0,h_1) \ra (0,1)$ with $h^{\frac{1}{3}} \dh^{-1}=o(1)$, such that for $h \in (0,h_1)$, and all $\ed{\mathcal{E}} \in (0, h \dh)$
\begin{equation}\label{1dresolveassume}
\nm{ (-h^2 \p_x^2 +i h \ti{W} -\ed{\mathcal{E}})^{-1} }_{\Lc(L^2)} \leq C \dh^{-1} h^{-2},
\end{equation}
then there exists $h_0 \in (0,1), C_0 >0$ such that for all $h \in (0,h_0)$ and all $E \in \Rb$, the solutions $v_{h,E}$ of \eqref{vhedef} satisfy 
$$
\ltwo{v_{h,E}} \leq C_0 h^{-2} \dh^{-1} \ltwo{r_{h,E}} + C_0 h^{-\frac{1}{2}} \dh^{-\frac{1}{2}} \ltwo{\ti{W}^{\frac{1}{2}} v_{h,E}}.
$$
\end{proposition}
\ed{We point out that we only assume \eqref{1dresolveassume} for $\mathcal{E} \in (0,h\dh)$, because we are able to show the desired resolvent estimate directly for $E \geq h\dh$ using a 1d geometric control result in Lemma \ref{highhyplemma}.}

We postpone the proof of this proposition and complete the proof of Theorem \ref{averagedresolventthm}.
\begin{proof}[Proof of Theorem \ref{averagedresolventthm}]
	Applying Proposition \ref{normalformquasiprop} to $u^1_h$ we obtain a sequence of quasimodes $\vht$. Then taking the Fourier transform in $y$ of \eqref{vhtdef}, and writing $v_{h,n}^{(2)}:= \F_y(\vht)(x,n)$ 
\begin{equation}
(-h^2 \p_x^2 + h^2 n^2-1 + i h A_{\zeta}(W) + h^2 \kappa(x) A_{\zeta}(W)^{\frac{1}{2}} b(hn) \p_x) v_{h,n}^{(2)} = \F_y(r_h).
\end{equation}
Let $E=1-h^2 n^2, \ti{W}=A_{\zeta}(W)(x)$ and $\kappa_{h,E}(x) = \kappa(x) b(hn)$, which is uniformly bounded in $W^{1,\infty}(\T)$ with respect to $h$ and $n$. Note that by Lemma \ref{averagingdbclemma}, $\ti{W} \in \Dc^{9, \frac{1}{4}}(\T^2)$.

Note also, the assumed 1-d resolvent estimates \eqref{assumed1dresolventeq} exactly match \eqref{1dresolveassume} and so we may apply Proposition \ref{1dresolveprop}. Doing so for each fixed $n \in \Zb$, then taking the $l^2_n$ norm and applying the Plancherel \ed{equality, for $h<h_0(\zeta_0)$ we have}
\begin{equation}
\ltwo{\vht} \leq C h^{-2} \dh^{-1} \ltwo{r_{h}} + Ch^{-\frac{1}{2}} \dh^{-\frac{1}{2}} \ltwo{A_{\zeta}(W)^{\frac{1}{2}} \vht} = o(1),
\end{equation}
where the final equality follows by recalling that $\ltwo{r_h} = o(h^2 \dh)$ and $\ltwo{A_{\zeta}(W)^{\frac{1}{2}} \vht} = o(\hbar)=o(h^{\frac{1}{2}} \dh^{\frac{1}{2}})$ from Proposition \ref{normalformquasiprop}.

Now by Proposition \ref{normalformquasiprop} $\ltwo{u_h^1} = o(1)$ \ed{ for $h<h_0(\zeta_0)$. Now because 
	\begin{equation}
		0=\limh \ltwo{u_h^1} = \int_{S^* \Tb^2} d \mu_1,
	\end{equation}
	and $\mu_{\zeta_0}=\mu_1|_{\zeta=\zeta_0}$, we have that $\mu_{\zeta_0}=0$. However this $\zeta_0$ was arbitrary within $\Qc$, so we have $\mu_{\zeta}=0$ for all $\zeta \in \Qc$} and so $\mu \equiv 0$. This contradicts \eqref{quasimodel2eq}, i.e. that $\ltwo{u_h}=1$.
\end{proof}

\subsection{Proof of 1-d resolvent connection}
In this section we prove Proposition \ref{1dresolveprop}. The statement and our approach are similar to that of \cite[Proposition 6.1]{Sun23}, but we proceed directly, rather than by contradiction, and our hypotheses and conclusion allow for non-polynomial $\dh$.
 
\begin{proof}[Proof of Proposition \ref{1dresolveprop}]

Let $\vhe, \rhe, E \in \Rb, h>0$ solve 
$$
(-h^2 \p_x^2 - E+ ih\ti{W} - h^2 \kappa_{h,E}(x) \ti{W}^{\frac{1}{2}} \p_x) \vhe = \rhe.
$$
Before proceeding with the proof we record a basic weighted energy \ed{equality}.
\begin{lemma}\label{weightedenergylemma}
Let $w \in C^2(\T; \Rb)$, then
\begin{align}
\int_{\T} w |h \p_x \vhe|^2 dx + \int_{\T} \left(-\frac{1}{2} h^2 \p_x^2 w-Ew\right) |\vhe|^2 dx - \frac{h^2}{2} \int_{\T}\p_x (w \kappa_{h,E} \ti{W}^{\frac{1}{2}}) |\vhe|^2 dx \\
= \Re \int_{\T} w r_{h,E} \bvhe dx.
\end{align}
\end{lemma}
This is exactly \cite[Lemma 6.3]{Sun23}.

Now the proof of Proposition \ref{1dresolveprop} is dived into 3 cases. Before we describe the cases, we must make a preliminary statement. To begin, note that $\{\ed{\ti{W}}>0\}$ contains at least one interval $I=(\alpha,\beta) \subset (-\pi,\pi)$. For a fixed $c_0 \geq 0$, we can construct a weight $w \in C^2(\T; \Rb)$ such that $w \geq c_0>0$ everywhere and $w'''<0$ in a neighborhood of $\T \backslash I$. Thus, for a $c_1>0$ sufficiently small
\begin{equation}
\frac{1}{2} w'' + c_1 w <0 \text{ in a neighborhood of } \T \backslash I.
\end{equation}
Now the three cases are
\begin{enumerate}
	\item Elliptic regime $E \leq c_1 h^2$
	\item High energy hyperbolic regime $E \geq h \dh$
	\item Low energy hyperbolic regime $c_1 h^2 \leq E \leq h \dh$.
\end{enumerate}

The elliptic regime will be addressed using the above weighted energy identity with the weight $w$ we just introduced. The high energy hyperbolic regime follows using a standard 1-d resolvent estimate of $(-h^2\p_x^2-E)$ and \ed{treats} the other terms as perturbations. For the low energy hyperbolic regime, the term $h^2 \kappa_{h,E} \ti{W}^{\frac{1}{2}} \p_x \vhe$ can be absorbed as an error into the right hand side, and we can then apply the assume 1-d resolvent estimate \eqref{1dresolveassume}.

\noindent 1) Elliptic regime $E \leq c_1 h^2$.

\begin{lemma} \label{ellipticlemma}
If $E \leq c_1 h^2$, then 
\begin{equation}
\ltwo{\vhe} \lesssim h^{-2} \ltwo{\rhe} + \ltwo{\ti{W}^{\frac{1}{2}} \vhe}.
\end{equation}
\end{lemma}
\begin{proof}
We will use the weighted energy identity to control $\p_x \vhe$ and then use the Poincar\'e-Wirtinger inequality to convert this to control of $\vhe$. 

Since $E \leq c_1 h^2$, using the weight $w$ we defined above, we have $\frac{1}{2}h^2 w'' + E w <0$ in a neighborhood of $\T \backslash I$.  Therefore, for a compact $K \subset I$, away from $\T \backslash I$,  
\begin{align*}
\int_{\T} \left(\frac{1}{2} h^2 \p_x^2 w + E w\right) |\vhe|^2 dx \leq \int_K \left(\frac{1}{2} h^2 \p_x^2 w + E w\right) |\vhe|^2 dx \leq Ch^2 \int_{\T} \ti{W} |\vhe|^2 dx,
\end{align*}
where $C= \frac{1}{\min_K \ti{W}}$. Then applying Lemma \ref{weightedenergylemma} with weight $w$, we have 
\begin{align}
c_0& \ltwo{h \p_x \vhe}^2 \leq \int_{\T} w |h \p_x \vhe|^2 dx \\
&\leq \Re \int_{\T} w \rhe \bvhe dx + Ch^2 \int_{\T} \ti{W} |\vhe|^2 dx + C \frac{h^2}{2} \int_{\T}\p_x (w \kappa_{h,E} \ti{W}^{\frac{1}{2}}) |\vhe|^2 dx. \label{vheellipticweight}
\end{align}
Now since $|\p_x(w \kappa_{h,E} \ti{W}^{\frac{1}{2}})| \leq C \ti{W}^{\frac{1}{4}}(x)$, then by H\"older's inequality
\begin{equation}
\ltwo{\ti{W}^{1/8} \vhe}^2 \leq \ltwo{\ti{W}^{\frac{1}{2}} \vhe}^{\frac{1}{2}} \ltwo{\vhe}^{\frac{3}{2}}.
\end{equation}
Applying this to \eqref{vheellipticweight} and applying Young's inequality for products
\begin{equation}\label{vhepelliptic}
\ltwo{\p_x \vhe} \leq C_{\e} h^{-2} \ltwo{\rhe} + C_{\e} \ltwo{\ti{W}^{\frac{1}{2}} \vhe} +\e\ltwo{\vhe}, \forall \e>0.
\end{equation}
Now, letting $\hvhe=\frac{1}{2\pi} \int_{\T} \vhe dx$, by the Poincar\'e-Wirtinger inequality
\begin{equation}
\ltwo{\vhe-\hvhe} \leq C \ltwo{\p_x \vhe}.
\end{equation}
So 
\begin{align}
\ltwo{\vhe} &\leq \ltwo{\vhe-\hvhe} + \ltwo{\hvhe} \\
&\leq C \ltwo{\p_x \vhe} + |\hvhe|.
\end{align}
Now because $\int_{\T} \ti{W} dx >0$  
\begin{align}
\ltwo{\vhe} &\leq C \ltwo{\p_x \vhe} + \left( \int \ti{W} dx \right)^{\frac{1}{2}} |\hvhe| \left( \int \ti{W} dx \right)^{-\frac{1}{2}} \\
&\leq  C \ltwo{\p_x \vhe} + C \left( \int \ti{W} |\hvhe|^2 dx \right)^{\frac{1}{2}} \\
&\leq C \ltwo{\p_x \vhe} + C \left( \int \ti{W} |\vhe|^2 dx \right)^{\frac{1}{2}} + C \left( \int \ti{W} |\vhe(x) - \hvhe|^2 dx\right)^{\frac{1}{2}}.
\end{align}
\ed{Bounding $\ti{W} \leq C$ in the third term, then} applying the Poincar\'e-Wirtinger inequality once more and then combining this with \eqref{vhepelliptic} gives 
\begin{equation}
\ltwo{\vhe} + \ltwo{\vhe'} \leq C_{\e} h^{-2} \ltwo{\rhe} + C_{\e} \ltwo{\ti{W}^{\frac{1}{2}} \vhe},
\end{equation}
as desired.
\end{proof}

\noindent 2) High energy hyperbolic regime $E \geq h \dh$

Recall the standard 1-d geometric control result, for a proof see \cite[Proposition 4.2]{Burq2020}. 
\begin{lemma} \label{1dgcclemma}
Let $I \subset \T$ be a nonempty open set. Then, there exists $C=C_I>0$ such that for any $v \in L^2(\T), f_1 \in L^2(\T), f_2 \in H^{-1}(\T)$ and $\lambda \geq 1$, if 
$$
(-\p_x^2-\lambda^2) v = f_1+f_2,
$$
then
$$
\ltwo{v} \leq C \lambda^{-1} \ltwo{f_1} + C \hp{f_2}{-1} + C\nm{v}_{L^2(I)}.
$$

\end{lemma}
To apply this lemma we treat both $\ti{W}$ and $\kappa_{h,E} \ti{W}^{\frac{1}{2}}$ terms as perturbations.

\begin{lemma}\label{highhyplemma}
There exists $C,h_0>0$ such that for $h \in (0,h_0)$ and $E \geq h \dh$, then 
$$
\ltwo{\vhe} \leq C h^{-\frac{3}{2}} \dh^{-\frac{1}{2}} \ltwo{\rhe} + h^{-\frac{1}{2}} \dh^{-\frac{1}{2}} \ltwo{\ti{W} \vhe} + \ltwo{\ti{W}^{\frac{1}{2}} \vhe}.
$$
\end{lemma}
\begin{proof}
Let $\lambda = h^{-1} E^{\frac{1}{2}} \geq h^{-\frac{1}{2}} \dh^{\frac{1}{2}}$. Then for $h$ small enough $\lambda \geq 1$. We can rearrange \eqref{vhedef}
\begin{equation}
(-\p_x^2 - \lambda^2) \vhe = h^{-2} \rhe - i h^{-1} \ti{W} \vhe - \kappa_{h,E} \ti{W}^{\frac{1}{2}} \p_x \vhe.
\end{equation}
Then by Lemma \ref{1dgcclemma} with $v=\vhe, f_1=h^{-2}\rhe-ih^{-1}\ti{W} \vhe, f_2 = -\kappa_{h,E} \ti{W}^{\frac{1}{2}} \p_x\vhe$ and $I=\{\ti{W}>c\}$
\begin{align}
\ltwo{\vhe} \leq& C \left( h E^{-\frac{1}{2}} \ltwo{h^{-2} \rhe - i h^{-1} \ti{W} \vhe} + \hp{\kappa_{h,E} \ti{W}^{\frac{1}{2}} \p_x \vhe}{-1} + \nm{\vhe}_{L^2(I)} \right) \\
\leq& C h^{-\frac{3}{2}} \dh^{-\frac{1}{2}} \ltwo{\rhe} + h^{-\frac{1}{2}}\dh^{-\frac{1}{2}} \ltwo{\ti{W} \vhe} \\
& + \hp{\p_x(\kappa_{h,E} W^{\frac{1}{2}} \vhe) - \p_x(\kappa_{h,E} \ti{W}^{\frac{1}{2}}) \vhe}{-1} + \frac{1}{c} \ltwo{\ti{W} \vhe}. \label{vhehighhypeq}
\end{align}
Where the second inequality follows \ed{by the product rule}. Since $|\p_x (\kappa_{h,E} \ti{W}^{\frac{1}{2}})| \leq C \ti{W}^{\frac{1}{4}}$, the $H^{-1}$ term is bounded by 
$$
\ltwo{\ti{W}^{\frac{1}{2}} \vhe} + \ltwo{\ti{W}^{\frac{1}{4}} \vhe} \leq \ltwo{\ti{W}^{\frac{1}{2}} \vhe} + C \ltwo{\ti{W}^{\frac{1}{2}} \vhe}^{\frac{1}{2}} \ltwo{\vhe}^{\frac{1}{2}}.
$$
Combining this with \eqref{vhehighhypeq} and Young's inequality for products, we obtain the desired result.
\end{proof}

\noindent 3) Low energy hyperbolic regime $c_1 h^2 \leq E \leq h \dh$

Again we let $\lambda = h^{-1} E^{\frac{1}{2}}$, so $\lambda \leq h^{-\frac{1}{2}} \dh^{\frac{1}{2}}$. As mentioned above, in this case, the $h^2 \kappa_{h,E} \ti{W}^{\frac{1}{2}}\p_x \vhe$ term can be treated as a perturbation. 
\begin{lemma}\label{lowhypinterlemma}
If $E \leq h \dh$, then there exists $C>0$, not depending on $h, \dh$ or  $E$ such that 
\begin{equation}
\ltwo{\ti{W}^{\frac{1}{2}} \p_x \vhe} \leq C \bigg( h^{-2} \ltwo{\rhe} + h^{-\frac{1}{2}} \dh^{\frac{1}{2}} \ltwo{\ti{W}^{\frac{1}{2}} \vhe}\\
+\ltwo{\ti{W}^{\frac{1}{2}} \vhe}^{\frac{1}{2}} \ltwo{\vhe}^{\frac{1}{2}}\bigg).
\end{equation}
\end{lemma}
\begin{proof}
Writing $\lambda = h^{-1} E^{\frac{1}{2}}$, $\vhe$ solves
\begin{equation}
-\p_x^2 \vhe - \lambda^2 \vhe + i h^{-1} \ti{W} \vhe= h^{-2} \rhe - \kappa_{h,E} \ti{W}^{\frac{1}{2}} \p_x \vhe.
\end{equation}
So now, consider the quantity to be estimated, integrate by parts, and apply the above equation
\begin{align}
\Re \int_{\T} \ti{W} \p_x \vhe \p_x \bvhe dx =& - \Re \int_{\T} \p_x \ti{W} \p_x \vhe \bvhe dx - \Re \int_{\T} \ti{W} \p_x^2 \vhe \bvhe dx \\
=&-\Re \int_{\T} \p_x \ti{W} \p_x \vhe \bvhe dx \\
& -\Re \int_{\T} \ti{W} {\bvhe} \left(-\lambda^2 \vhe + i h \ti{W} \vhe - h^2 \rhe + \kappa_{h,E} \ti{W}^{\frac{1}{2}} \p_x \vhe \right) dx. \label{lowhypintereq}
\end{align}
We now estimate each of the terms on the right hand side in turn.

Writing $\Re(\p_x \vhe {\bvhe}) = \frac{1}{2} \p_x |\vhe|^2$, integrating by parts, using that $\ti{W} \in \Dc^{9,\frac{1}{4}}$ and applying H\"older's inequality
\begin{equation}
-\Re \int_{\T} \p_x \ti{W} \p_x \vhe \vhe dx = \frac{1}{2} \int_{\T} \p_x^2 \ti{W} |\vhe|^2 \leq C \ltwo{\ti{W}^{\frac{1}{4}} \vhe}^2 \leq C \ltwo{\ti{W}^{1/2} \vhe} \ltwo{\vhe}.
\end{equation}
Using the same identity of $\vhe$ we can also write
\begin{align*}
\Re \int_{\T} \ti{W} \kappa_{h,E} \bvhe \ti{W}^{\frac{1}{2}} \p_x \vhe dx &= \frac{1}{2} \int_{\T} \ti{W}^{\frac{3}{2}} \kappa_{h,E} \p_x |\vhe|^2 dx \\
&= -\frac{1}{2} \int_{\T} \p_x (\kappa_{h,E} \ti{W}^{\frac{3}{2}}) |\vhe|^2 dx. 
\end{align*}
So 
\begin{equation}
\left|\Re \int_{\T} \ti{W} \bvhe \kappa_{h,E} \ti{W}^{\frac{1}{2}} \p_x \vhe dx \right| \leq C \ltwo{\ti{W}^{\frac{1}{2}} \vhe}^2.
\end{equation}
Since $\lambda^2 \leq h^{-1} \dh$ 
\begin{equation}
\left| \int_{\T} \kappa_{h,E} \ti{W}^{\frac{3}{2}} \lambda^2 |\vhe|^2 dx \right| \leq C h^{-1} \dh \ltwo{ \ti{W}^{\frac{1}{2}} \vhe}^2. 
\end{equation}
For the final term
\begin{align*}
\left|\Re \int_{\T} \ti{W} \bar{v}_{h,E} (-ih^{-1} \ti{W} \vhe - h^{-2} \rhe) dx \right| &= \left| h^{-2} \Re \int_{\T} \ti{W} \bar{v}_{h,E}  \rhe dx \right| \\
& \leq h^{-2} \ltwo{\rhe} \ltwo{\ti{W}^{\frac{1}{2}} \vhe}.
\end{align*}
Plugging these back into \eqref{lowhypintereq}, applying Young's inequality for products, and taking the square root of both sides gives the desired result. 
\end{proof}

We can now complete the low energy hyperbolic estimate using the assumed 1-d resolvent estimates 
\begin{lemma}\label{lowhyplemma}
There exist $c_1, h_0>0$ such that if $h \in (0,h_0)$ and $c_1 h^2 \leq E \leq h \dh$, then
\begin{equation}
\ltwo{\vhe} \leq C \left( h^{-2} \dh^{-1} \ltwo{\rhe} + h^{-\frac{1}{2}} \dh^{-\frac{1}{2}} \ltwo{\ti{W}^{\frac{1}{2}} \vhe} \right).
\end{equation}
\end{lemma}
\begin{proof}
Write $\vhe$ as a solution of 
\begin{equation}
(-h^2 \p_x^2 + i h \ti{W} - E) \vhe = \rhe - h^2 \kappa_{h,E} \ti{W}^{\frac{1}{2}} \p_x \vhe.
\end{equation}
Then by the assumed 1-d resolvent estimate \eqref{1dresolveassume} and Lemma \ref{lowhypinterlemma}
\begin{align*}
\ltwo{\vhe} &\leq C h^{-2} \dh^{-1} \ltwo{\rhe-h^2 \kappa_{h,E} \ti{W}^{\frac{1}{2}} \p_x \vhe} \\
& \leq C \bigg( h^{-2} \dh^{-1} \ltwo{\rhe} + \dh^{-1}h^{-\frac{1}{2}} \dh^{\frac{1}{2}} \ltwo{\ti{W}^{\frac{1}{2}}\vhe} \\
&\qquad \qquad + \dh^{-1} \ltwo{\ti{W}^{\frac{1}{2}} \vhe}^{\frac{1}{2}} \ltwo{\vhe}^{\frac{1}{2}} \bigg).
\end{align*}
Applying Young's inequality to the last term and absorbing back $\ltwo{\vhe}$ into the left hand side 
\begin{equation}
\ltwo{\vhe} \leq C \left( h^{-2} \dh^{-1} \ltwo{\rhe} + (h^{-\frac{1}{2}} \dh^{-\frac{1}{2}} + C_{\e} \dh^{-2}) \ltwo{\ti{W} \vhe} \right).
\end{equation}
Now $\dh^{-2} \leq C h^{-\frac{1}{2}} \dh^{-\frac{1}{2}}$ since $h^{\frac{1}{3}} \dh^{-1}=o(1)$ and so the desired estimate holds. 
\end{proof}
Now we conclude the proof of Proposition \ref{1dresolveprop}. Since the 3 cases $E \leq c_1 h^2, c_1 h^2 \leq E \leq h \dh, h\dh \leq E$ cover all possible relations between $E$ and $h$, by Lemma \ref{ellipticlemma}, \ref{highhyplemma}, and \ref{lowhyplemma}, we have the existence of $C, h_0>0$ such that if $h \in (0,h_0)$ then
\begin{equation}
\ltwo{\vhe} \leq C \left( (h^{-2} \dh^{-1} + h^{-\frac{3}{2}} \dh^{-\frac{1}{2}} + h^{-2} ) \ltwo{\rhe} + (h^{-\frac{1}{2}} \dh^{-\frac{1}{2}} +1) \ltwo{\ti{W}^{\frac{1}{2}} \vhe} \right).
\end{equation}
Then, potentially taking $h_0$ smaller, we can drop the $h^{-\frac{3}{2}} \dh^{-\frac{1}{2}}, h^{-2},$ and $1$ terms to obtain the desired inequality. 
\end{proof}

\section{Applications involving geometry of the damped set}\label{s:geometryApplications}
\ed{In this section, we discuss consequences of Theorem \ref{averagedresolventthm} where the damped set $\{W>0\}$ satisfies some additional geometric considerations. We first state the results, then in subsection \ref{avggrowthsec} we state and prove averaging results which we use in the proofs. In section \ref{proofconvexrate} we prove these energy decay rates.} 
	
First we consider damping possessing a product structure and supported on a rectangle, and show that it decays at rate matching the decay rate of its ``rougher" direction. 
\begin{theorem}\label{rectthm}
	Assume that $W \in \mathcal{D}^{9,\frac{1}{4}}(\T^2)$. Suppose $\{W>0\}$ is the rectangle $(-\sigma_1, \sigma_1)_x \times (\sigma_2, \sigma_2)_y$ and let $\tix=(|x|-\sigma_1)_+$ and $\tiy = (|y|-\sigma_2)_+$.  
	\begin{enumerate}
		\item Assume that there exists $R_0 \geq 1,$ such that for  $(x,y) \in \{W>0\}$ 
		\begin{equation}
			R_0^{-1} \tix^{\beta_1} \tiy^{\beta_2} \ln(\tix^{-1})^{-\gamma_1} \ln(\tiy^{-1})^{-\gamma_2} \leq W(x,y) \leq R_0 \tix^{\beta_1} \tiy^{\beta_2} \ln(\tix^{-1})^{-\gamma_1} \ln(\tiy^{-1})^{-\gamma_2},
		\end{equation}
		with $0 \leq \beta_1<\beta_2$ and $\gamma_1, \gamma_2 \in \Rb$, or $0< \beta_1=\beta_2$, and $\gamma_1<\gamma_2 \in \Rb$, then energy decays at rate 
		\begin{equation}
			r(t)=t^{-\frac{\beta_1+2}{\beta_1+3}} \ln(\ed{2}+t)^{\frac{-\gamma_1}{\beta_1+3}}.
		\end{equation}
		\item Assume that there exists $R_0 \geq 1,$ such that for   $(x,y) \in \{W>0\}$  
		\begin{equation}
			R_0^{-1} \tix^{\beta_1} \tiy^{\beta_2} \exp(-c_1 \tix^{-\alpha_1} - c_2 \tiy^{-\alpha_2})  \leq W(x,y) \leq R_0 \tix^{\beta_1} \tiy^{\beta_2} \exp(-c_1 \tix^{-\alpha_1} - c_2 \tiy^{-\alpha_2}), 
		\end{equation}
		with $c_1, c_2>0, \beta_1, \beta_2 \in \Rb,$ and $0<\alpha_1 \leq \alpha_2$, and that $W$ satisfies Assumptions \ref{noninvarassumption} and \ref{widegenbadassumption} with $q(z) = \ln(z^{-1})^{-\left(\frac{\alpha_1+1}{\alpha_1}\right)}, p(z)=q(z)^2$, 
		then energy decays at rate 
		\begin{equation}
			r(t) = t^{-1} \ln(\ed{2}+t)^{\frac{2(\alpha_1+1)}{\alpha_1}-\frac{1}{\alpha_2}}.
		\end{equation}
	\end{enumerate}
\end{theorem}
\begin{remark}
	\begin{enumerate}
		\item Note that $W(x,y)=(|x|-\sigma_1)_+^{\beta_1} (\ed{|y|}-\sigma_2)_+^{\beta_2}$ with $9 \leq \beta_1 \leq \beta_2$ satisfies \eqref{ddbc} with $\e=1/\beta_1$. \ed{Thus Theorem \ref{rectthm} gives energy decay at a faster rate than the $t^{-\frac{\beta_1+1}{\beta_1+2}}$ provided by 
		 Theorem \ref{dbctheorem}}.
		\item In case 1 the damping has less regularity and faster growth in the $x$ direction, than in the $y$ direction. One would anticipate that the additional regularity in the $y$ direction should not make the energy decay worse than the rate for $y$-invariant damping in Theorem \ref{wideexthm}, and our result bears this out. 
		\item When $W(x,y) = \exp(-(|x|-\sigma_1)_+^{-\alpha_1} - (|y|-\sigma_2)_+^{\alpha_2})$ with $0 < \alpha_1 \leq \alpha_2$, it satisfies the hypotheses of case 2 of Theorem \ref{rectthm}. It also satisfies \eqref{ddbc} for any $\e>0$, so Theorem \ref{dbctheorem} would only give decay at $t^{-1+\d}$, for any $\d>0$, which is slower than the decay provided by Theorem \ref{rectthm} case 2. For such damping, it would be reasonable to expect that the $\frac{1}{\alpha_2}$ term in the decay rate in case 2 could be replaced by an $\alpha_1$ to match the best decay rate from Theorem \ref{wideexthm}, as we did in case 1 of this theorem, but we are unable to do so. 
		\item We do not include damping with growth behavior given purely by a power of $\ln(x^{-1})$ as such damping do not satisfy $W \in \Dc^{9,\frac{1}{4}}(\T^2)$. 
		\item In case 2, the concavity portions of Assumptions \ref{noninvarassumption} and \ref{widegenbadassumption} are guaranteed for such a $q$ and $p$ by Lemma \ref{logconcavelemma}.
	\end{enumerate}
\end{remark}
\ed{Before discussing our next result, recall that} in \cite{Sun23}, energy decay rates for the $y$-invariant damping $(|x|-\sigma)_+^{\beta}$ are converted to energy decay rates for damping supported on strictly convex sets with positive curvature and growing like $d^{\beta}$, where $d$ is the distance to the boundary of the support. In particular, the energy decay rate is improved from $t^{-\frac{\beta+2}{\beta+3}}$  for $y$-invariant damping, to $t^{-\frac{\beta+\frac{1}{2}+2}{\beta+\frac{1}{2}+3}}$ for damping with strictly convex support.  We generalize this by extending it to damping turning on non-polynomially and to sets that interpolate between rectangles and strictly convex sets with positive curvature. To make this discussion precise, we define locally strictly convex sets with positive curvature.

\begin{definition}
	Let $\T^2_{A,B} :=\Rb^2 /(2\pi A \times 2 \pi B)$ be a general flat torus defined via the covering map $\pi_{A,B}: \Rb^2 \ra \T^2_{A,B}$. 
	An open set $\omega \subset \T^2$ is said to be locally strictly convex with positive curvature, if the boundary of each component of $\pi^{-1}(\omega) \subset \Rb^2$ is $C^2$ and has strictly positive curvature as a closed curve in $\Rb^2$. Sometimes we will say that the boundary of $\omega$ is locally strictly convex. 
	
	For a function $f$, we define notation for the boundary of its support 
	\begin{equation}
		\ed{\Sigma(f) = \p \{z: f(z)>0\}.}
	\end{equation}
\end{definition}

When $\Sigma(W)$ is locally strictly convex with positive curvature, we are able to generalize the improvement made in \cite{Sun23} to poly-log damping. We also point out that applying this same approach to poly-exponential damping does not improve the energy decay rate.

\begin{theorem}\label{polylogconveximprovethm}
	Assume that $W \in \mathcal{D}^{9,\frac{1}{4}}(\T^2)$.  Assume further that $\{ (x,y) \in \T^2; W(x,y)>0\}$ is locally strictly convex and let $d=dist((x,y), \Sigma(W))$.
	\begin{enumerate}
		\item If there exists $R_0 \geq 1, \beta > 0,$ and $\gamma \in \Rb,$ such that for  $(x,y) \in \{W>0\}$
		\begin{equation}
			R_0^{-1} d^{\beta} \ln(d^{-1})^{-\gamma} \leq W(x,y) \leq R_0 d^{\beta} \ln(d^{-1})^{-\gamma},
		\end{equation}
		then solutions of \eqref{DWE} are stable at rate 
		\begin{equation}
			r(t) = t^{-\frac{\beta+\frac{1}{2}+2}{\beta+\frac{1}{2}+3}} \ln(\ed{2}+t)^{\frac{\gamma}{\beta+\frac{1}{2}+3}}.
		\end{equation}
		\item If there exists $R_0 \geq 1, \alpha,c>0,$ and $\beta \in \Rb,$ such that for $(x,y) \in \{W>0\}$
		\begin{equation}
			R_0^{-1} d^{\beta} \exp(-cd^{-\alpha}) \leq W(x,y) \leq R_0 d^{\beta} \exp(-cd^{-\alpha}),
		\end{equation}
		and $W$ satisfies Assumptions \ref{noninvarassumption} and \ref{widegenbadassumption} with $q(z) = \ln(z^{-1})^{-\left(\frac{\alpha+1}{\alpha}\right)}, p(z)=q(z)^2$, 
		then solutions of \eqref{DWE} are stable at rate
		\begin{equation}
			r(t) = t^{-1} \ln(\ed{2}+t)^{\frac{2\alpha+1}{\alpha}}.
		\end{equation}
	\end{enumerate}
\end{theorem}
\begin{remark}
	\begin{enumerate}
		\item Case 1 recovers the sharp decay of \cite{Sun23} at $\gamma=0$. When $\gamma$ is nonzero, there is a $\gamma$ dependent $\ln(\ed{2}+t)$ change in the energy decay rate. Just as \cite{Sun23} replaces $\beta$ by $\beta+\frac{1}{2}$ in the rate of \cite{Kleinhenz2019, DatchevKleinhenz2020}, this Theorem replaces $\beta$ by $\beta+\frac{1}{2}$ in the rate from Theorem \ref{wideexthm}.
		\item This result obtains the improved energy decay rate for poly-log $W$ in part by showing that averaging $W$ along straight lines increases the polynomial power from $\beta$ to $\beta+\frac{1}{2}$, when $\{W>0\}$ is strictly convex. When $W$ is poly-\ed{exponential} and $\{W>0\}$ is strictly convex, averaging $W$ along straight lines actually increases the polynomial power from $\beta$ to $\beta+\frac{1}{2}+\alpha$. However, the change in the polynomial power is only relevant in the derivative bound condition for poly-log damping, \ed{see} Theorem \ref{wideexthm}, Examples \ref{polyexp1dex} and \ref{polylog1dex}, which is why there is no change to the energy decay rate for poly-exponential damping. 
		\item In case 2, the concavity portions of Assumptions \ref{noninvarassumption} and \ref{widegenbadassumption} are guaranteed for such a $q$ and $p$ by Lemma \ref{logconcavelemma}.
	\end{enumerate}
\end{remark}

To interpolate between strictly convex sets and rectangles we use superellipses, also called Lam\'e curves. 

We define the general superellipse, for some $a,b>0, n \geq m \geq 2$ with $n, m \in \Rb$
\begin{equation}
	E_{a,b}^{n,m} :=\{(x,y) \in \T^2: W(x,y)>0\} = \left\{\left|\frac{x}{a}\right|^{n} + \left|\frac{y}{b}\right|^{m} < 1\right \}.
\end{equation}
\begin{figure}[h]
	\centering
	\begin{tikzpicture}
		[
		declare function={
			sxn(\t)= a*cos(\t r)^(2/n);
			syn(\t)= b*sin(\t r)^(2/n);
			sxm(\t)=a*cos(\t r)^(2/m);
			sym(\t)=b*sin(\t r)^(2/n);
			a=3;
			b=2;
			n=4;
			m=10;
		}]
		\draw[usuai, thick] (-a,-b) rectangle (a, b);
		\draw[imayou, variable=\t, domain=0:pi/2, thick] plot ({sxm(\t)},{sym(\t)}) -- plot({-sxm(pi/2-\t)},{sym(pi/2-\t)}) -- plot({-sxm(\t)},{-sym(\t)}) -- plot({sxm(pi/2-\t)},{-sym(pi/2-\t)}) -- cycle;
		\draw[persred, variable=\t, domain=0:pi/2, thick] plot ({sxn(\t)},{syn(\t)}) -- plot({-sxn(pi/2-\t)},{syn(pi/2-\t)}) -- plot({-sxn(\t)},{-syn(\t)}) -- plot({sxn(pi/2-\t)},{-syn(pi/2-\t)}) -- cycle;
		\draw[sand, thick] (0,0) ellipse (a and b);
		\node[sand] at (4.5,-.5) {$n=m=2$};
		\node[persred] at (4.5,0){$n=m=4$};
		\node[imayou] at (4.5,.5){$n=4, m=10$};
		\node[usuai] at (4.5,1) {$``n=m=\infty"$};
	\end{tikzpicture}
	\caption{The superellipse $E_{a,b}^{n,m}$ for various values of $n, m$. As $n$ increases it becomes more ``rectangular" in the $x$ variable and as $m$ increases it becomes more ``rectangular" in the $y$ variable.}
\end{figure}

When $\ed{m}>2$, $\p E_{a,b}^{n,m}$ does not have positive curvature at $(\pm a,0)$ and when $\ed{n} >2, \p E_{a,b}^{n,m}$ does not have positive curvature at $(0, \pm b)$. In both cases it has strictly positive curvature at all other points. This loss of positive curvature prevents the full improvement from $\beta$ to $\beta+\frac{1}{2}$ in the decay rate for poly-log damping, but it is not too severe to prevent a weaker improvement. In particular the energy decay rate is improved by replacing $\beta$ by $\beta+\frac{1}{n}$ (i.e. adding the smaller of $\frac{1}{n}$ and $\frac{1}{m}$). There is again no improvement to the energy decay rate for poly-exponential damping,  for the same reasons as before. 
\begin{theorem}\label{ellipsethm}
	Assume that $W \in \mathcal{D}^{9,\frac{1}{4}}(\T^2)$.  Assume further that $\{ (x,y) \in \T^2; W(x,y)>0\}=E_{a,b}^{n,m}$ and let $d=dist((x,y), \Sigma(W))$.
	\begin{enumerate}
		\item If there exists $R_0 \geq 1, \beta>0,$ and $\gamma \in \Rb,$ such that for $(x,y) \in \{W>0\}$
		\begin{equation}
			R_0^{-1} d^{\beta} \ln(d^{-1})^{-\gamma} \leq W(\ed{x,y}) \leq R_0 d^{\beta} \ln(d^{-1})^{-\gamma},
		\end{equation}
		then the damped wave equation is stable at rate 
		\begin{equation}
			r(t) = t^{-\frac{\beta+\frac{1}{n}+2}{\beta+\frac{1}{n}+3}} \ln(\ed{2}+t)^{-\frac{\gamma}{\beta+\frac{1}{n}+3}}.
		\end{equation}
		\item If there exists $R_0 \geq 1, \alpha,c>0,$ and  $\beta \in \Rb,$ such that for $(x,y) \in \{W>0\}$
		\begin{equation}
			R_0^{-1} d^{\beta} \exp(-cd^{-\alpha}) \leq W(x,y) \leq R_0 d^{\beta} \exp(-cd^{-\alpha}),
		\end{equation}
		and $W$ satisfies Assumptions \ref{noninvarassumption} and \ref{widegenbadassumption} with $q(z) = \ln(z^{-1})^{-\left(\frac{\alpha+1}{\alpha}\right)}, p(z)=q(z)^2$, then solutions of \eqref{DWE} are stable at rate
		\begin{equation}
			r(t) = t^{-1} \ln(\ed{2}+t)^{\frac{2\alpha+1}{\alpha}}.
		\end{equation}
	\end{enumerate}
\end{theorem}

\begin{remark}
	\begin{enumerate}	
		\item When $n=m=2$, this is just an ellipse, which is convex with strictly positive curvature, and the decay rate exactly matches the improvement for convex sets with strictly positive curvature in Theorem \ref{polylogconveximprovethm}.
		\item As $n$ or $m \ra \infty$ the superellipse begins to resemble two parallel lines joined by two curves. If both $n,m \ra \infty$, the superellipse approaches a rectangle. Either way, the energy decay rate approaches that found in Theorem \ref{rectthm}.
		\item This result generalizes \cite{Sun23} to provide improved energy decay rates for domains that are not strictly convex with positive curvature. \ed{See \cite{DKP25} for a further generalization that allows $\Sigma(W)$ to be a curve or polygon when $W(z) \simeq d^{\beta}$.} 
		\item As when $\{W>0\}$ is strictly convex, there is no improvement to the energy decay rate for poly-exponential damping, as different values of $\beta$ have no effect on the $q$ in the derivative bound condition, and thus no effect on the energy decay rate.
		\item In case 2, the concavity portions of Assumptions \ref{noninvarassumption} and \ref{widegenbadassumption} are guaranteed for such a $q$ and $p$ by Lemma \ref{logconcavelemma}.
	\end{enumerate}
\end{remark}

\subsection{Growth properties of averaged functions using the geometry of the support}\label{avggrowthsec}
In this section we prove estimates for the growth behavior of $A_v(W)$ near the boundary of its support, in terms of the growth behavior of $W$ near the boundary of its support, as well as the geometry of the support of $W$. \ed{These will be used in Section \ref{proofconvexrate} along with Theorems \ref{averagedresolventthm} and \ref{wideexthm} to prove the energy decay rates from the previous section.}

First, if $\{W>0\}$ is a rectangle, when averaging along a direction $v$ not parallel to one of the sides of the rectangle, there is a polynomial improvement in the growth behavior of $A_v(W)$ relative to $W$. 

\begin{lemma}\label{averagingsquarelemma}
Assume that $\{W>0\}$ is the rectangle $(-\sigma_1, \sigma_1) \times (-\sigma_2, \sigma_2) $ and let $\tix = (|x|-\sigma_1)_+$ and $\tiy = (|y|-\sigma_2)_+$. 
Let $v \in \Sb^1$ periodic, with $v \not \in\{(1,0), (0,1)\}$.
 
\begin{enumerate}
	\item Assume that there exists $C \geq 1$ such that for $(x,y)$ near $\Sigma(W)$ 
	\begin{equation}
	C^{-1}  \ed{\tix}^{\beta_1} \ed{\tiy}^{\beta_2} \ln(\tix^{-1})^{-\gamma_1} \ln(\tiy^{-1})^{-\gamma_2} \leq W(x,y) \leq C  \ed{\tix}^{\beta_1} \ed{\tiy}^{\beta_2} \ln(\tix^{-1})^{-\gamma_1} \ln(\tiy^{-1})^{-\gamma_2},
	\end{equation}
	for $\beta_1, \beta_2 >0$ and $\gamma_1, \gamma_2 \in \Rb$, then there exists $C_v \geq 1$ such that for all $z \in \{A_v(W)(z) >0\}$ near $\Sigma(A_v(W))$, and letting $s=dist(z, \Sigma(A_v(W)))$
	\begin{equation}
	C_v^{-1} s^{\beta_1+\beta_2+1} \ln(s^{-1})^{-\gamma_1-\gamma_2} \leq A_v(W)(s) \leq C_vs ^{\beta_1+\beta_2+1} \ln(s^{-1})^{-\gamma_1-\gamma_2}.
	\end{equation}
	\item Assume that there exists $C \geq 1$, such that for $(x,y)$ near $\Sigma(W)$
	\begin{equation}	
	 C^{-1} \tix^{\beta_1} \exp(-c_1 \tix^{-\alpha_1}) \tiy^{\beta_2} \exp(-c_2 \tiy^{-\alpha_2} )\leq W(x,y) \leq C \tix^{\beta_1} \exp(-c_1 \tix^{-\alpha_1}) \tiy^{\beta_2} \exp(-c_2 \tiy^{-\alpha_2}),
	 \end{equation}
	 for $0 < \alpha_1 \leq \alpha_2,$ $0<c_1,c_2$ and $\beta_1, \beta_2 \in \Rb$, then there exists $C_v, c_{v} \geq 1$ such that for all $z \in \{A_v(W)(z) >0\}$ near $\Sigma(A_v(W))$, and letting $s=dist(z, \Sigma(A_v(W)))$
	\begin{equation}
	\ed{C_v^{-1} \exp\left(\ed{-}c_{v} s^{-\alpha_2}\right) \leq A_v(W)(s) \leq C_v  \exp\left( \ed{-}c_{v}^{-1} s^{-\alpha_2}\right) }.
	\end{equation}
\end{enumerate}
\end{lemma}
Note that $W(x,y)=f(x)g(y)$ and so when $v$ is parallel to one of the sides of the rectangle $A_v(W) = C f(x)$ or $Cg(y)$. 
\begin{proof}
1) Fix $v \neq (1,0), (0,1)$. Note that $A_v(W)(z)$ can only be near 0, when $z+vt$ is a line that passes near a single corner of $R$, \ed{see Figure \ref{f:squareCorner}}. Because of this, after changing coordinates so that the corner we are interested in is located at $(0,0)$, we can understand the behavior of $A_v(W)(z)$ near the boundary of its support via 
\begin{equation}
A_v(W)(s) = \frac{1}{T_v} \int_{-T_v/2}^{T_v/2} W(s v^{\bot} + tv) dt, \quad 0<s<<1,
\end{equation} 
where $T_v$ is the period of the trajectory generated by $v$. We will focus on the case where the corner is the bottom left corner of $R$, the other cases will follow from analogous arguments. 

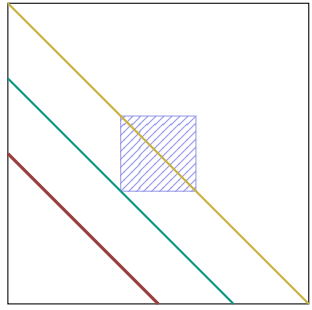
\begin{figure}[h]
\centering
\begin{tikzpicture}
[
declare function={
  a=2;
}]
\draw (-a,-a) rectangle (a,a);
\draw[imayou, pattern=north east lines, pattern color=imayou] (-a/4,-a/4) rectangle (a/4,a/4);
\draw[persred, -, very thick] (0,-a)--(-a,0);
\draw[sand, -, thick] (-a,a)--(a,-a);
\draw[usuai, -, thick] (-a/4,-a/4) --(a/2, -a);
\draw[usuai, -, thick] (-a/4, -a/4)--(-a,a/2);
\end{tikzpicture}
\caption{The torus with damping supported on the purple rectangle. Note that $A_v(W)(z)=0$ for $z$ along the red line, $A_v(W)(z)$ is near 0 for $z$ near to the green line. $A_v(W)(z) \geq c_0 > 0$ for $z$ near the yellow line.}\label{f:squareCorner}
\end{figure}

For both of our growth behavior cases, we have $W(x,y)=0$ for $x \leq 0, y\leq 0$. Letting $v=(-v_1, v_2)$ and $v^{\bot}=(v_2, v_1)$ for $v_1, v_2>0$, then $sv^{\bot} + tv = (sv_2-t v_1, tv_2 + sv_1)$, so $sv_2 - tv_1<0$ when $t> \frac{sv_2}{v_1}$ and $tv_2 + sv_1<0$ when $t<-\frac{sv_1}{v_2}$. Thus
\begin{align}
A_v(W)(s) &= \frac{1}{T_v}\int_{-T_v/2}^{T_v/2} W(s v^{\bot} + tv) dt = \frac{1}{T_v} \int_{-\frac{sv_1}{v_2}}^{\frac{sv_2}{v_1}} W(sv_2 - tv_1, tv_2 + sv_1) dt \\
&= \frac{1}{T_v} \left( \int_0^{ \frac{sv_2}{v_1}} W dt + \int_{-\frac{sv_1}{v_2}}^0 W dt\right). \label{rectspliteq}
\end{align}
Note that 
\begin{align}\label{rectangleavgcoordeq}
&\text{For } t \in \left[0, \frac{sv_2}{v_1}\right], \quad sv_1 \leq tv_2+sv_1 \leq s\left(v_1 + \frac{v_2^2}{v_1} \right) \\
&\text{ For } t \in \left[-\frac{sv_1}{v_2}, 0\right], \quad sv_2 \leq sv_2 - tv_1 \leq s\left(v_2 + \frac{v_1^2}{v_2} \right).
\end{align}
We can now specialize to each of the cases. 

2) In the first case 
\begin{equation}
W(sv_2 - tv_1, tv_2 +sv_1) = (sv_2-tv_1)^{\beta_1} \ln((sv_2-tv_1)^{-1})^{-\gamma_1} (tv_2+sv_1)^{\beta_2}\ln((tv_2+sv_1)^{-1})^{-\gamma_2}.
\end{equation}
Therefore by \eqref{rectspliteq} and \eqref{rectangleavgcoordeq}
\begin{align}
A_v(W)(s) \simeq& s^{\beta_2} \ln(s^{-1})^{-\gamma_2} \int_0^{\frac{sv_2}{v_1}}(sv_2-tv_1)^{\beta_1} \ln((sv_2-tv_1)^{-1})^{-\gamma_1} dt \\
&+ s^{\beta_1} \ln(s^{-1})^{-\gamma_1} \int_{-\frac{sv_1}{v_2}}^0 (tv_2+sv_1)^{\beta_2} \ln((tv_2+sv_1)^{-1})^{-\gamma_2} dt. 
\end{align}
In the first integral, make the substitution $u=sv_2-tv_1$ and in the second $u=tv_2+sv_1$
\begin{align}
A_v(W)(s) \simeq& s^{\beta_2} \ln(s^{-1})^{-\gamma_2} \int_0^{sv_2} u^{\beta_1} \ln(u^{-1})^{-\gamma_1} dt \\
&+ 
s^{\beta_1} \ln(s^{-1})^{-\gamma_1} \int_0^{sv_1} u^{\beta_2} \ln(u^{-1})^{-\gamma_2} dt. 
\end{align}
Then by Lemma \ref{uppergammalem}, since $s$ is near $0$
\begin{equation}
A_v(W)(s) \simeq s^{\beta_1+\beta_2+1} \ln(s^{-1})^{-\gamma_1-\gamma_2},
\end{equation}
as desired.

2) In the second case 
\begin{equation}
W(sv_2-tv_1, tv_2+sv_1) = (sv_2-tv_1)^{\beta_1} (tv_2+sv_1)^{\beta_2} \exp(-c_1 (sv_2-tv_1)^{-\alpha_1}-c_2(tv_2+sv_1)^{-\alpha_2}).
\end{equation}
By \eqref{rectspliteq} and \eqref{rectangleavgcoordeq} there exists $\ti{c_1}, \ti{c_2} >0$ such that 
\begin{align}
A_v(W)(s) \lesssim& s^{\beta_2} \exp(-\ti{c_2} s^{-\alpha_2} )\int_0^{\frac{sv_2}{v_1}} \exp(-c_1(sv_2-tv_1)^{-\alpha_1}) (sv_2-tv_1)^{\beta_1} dt \\
&+s^{\beta_1} \exp(-\ti{c_1} s^{-\alpha_1}) \int_{-\frac{sv_1}{v_2}}^0 \exp(-c_2(tv_2+sv_1)^{-\alpha_2} (tv_2+sv_1)^{\beta_2} dt.
\end{align} 
We will focus on the upper bound as the arguments for the lower bound are analogous. Using the change of variables $u=c_1^{-\frac{1}{\alpha\ed{_1}}}(sv_2-tv_1)$ for the first integral and $u=c_2^{-\frac{1}{\alpha\ed{_2}}}(tv_2+sv_1)$ for the second integral
\begin{align}
A_v(W)(s) \lesssim& s^{\beta_2} \exp(-\ti{c_2} s^{-\alpha_2}) \int_0^{c_1^{-\frac{1}{\ed{\alpha_1}}} sv_2} u^{\beta_1} \exp(-u^{-\alpha_1}) du \\
&+ s^{\beta_1} \exp(-\ti{c_1} s^{-\alpha_1}) \int_0^{c_2^{-\frac{1}{\ed{\alpha_2}}} sv_1} u^{\beta_2} \exp(-u^{-\alpha_2}) du. 
\end{align}
Therefore by Lemma \ref{uppergammalem} since $s$ is near 0
\begin{align}
A_v(W)(s) \lesssim& s^{\beta_1+\beta_2+\alpha_1+1} \exp(-\ti{c_2} s^{-\alpha_2} -c_1 v_2^{-\alpha_1} s^{-\alpha_1}) \\
&+ s^{\beta_1+\beta_2+\alpha_2+1} \exp(-\ti{c_1} s^{-\alpha_1} -c_2 v_2^{-\alpha_2} s^{-\alpha_2}).
\end{align}
Since $0 < \alpha_1 \leq \alpha_2$, for $s$ small we have 
\begin{equation}
A_v(W)(s) \leq C_{v} s^{\beta_1+\beta_2+\alpha_1+1} \exp(-c_{v}^{-1} s^{-\alpha_2}).
\end{equation}
\ed{Taking $c_{v}$ larger we can drop the polynomial factor and have 
\begin{equation}
	A_v(W)(s) \leq C_v \exp(-c_{v}^{-1} s^{-\alpha_2}).
\end{equation}}
%
\end{proof}

When averaging damping with a known growth behavior at the boundary of its support, there is an improvement to this boundary growth behavior when the support of the damping is convex. This generalizes \cite[Proposition 4.4]{Sun23}.
\begin{lemma} \label{avgconveximproveprop}
Assume $\Sigma(W)$ is a disjoint union of strictly convex curves with positive curvature. 
\begin{enumerate}
	\item If there exists $R>0$ such that for all $z \in \{W>0\}$ near $\Sigma(W)$, letting $d=\dist(z,\Sigma(W))$
	\begin{equation}\label{bndrypolylog}
	R^{-1} d^{\beta} \ln(d^{-1})^{-\gamma} \leq W(z) \leq R d^{\beta} \ln(d^{-1})^{-\gamma},
	\end{equation}
	then for any periodic direction $v \in \Sb^1$, there exists $R_v>0$ such that for all $z \in \{A_v(W) >0\}$ near $\Sigma(A_v(W))$, letting $s= \dist(z,\Sigma(A_v(W)))$
	\begin{equation}\label{bndrypolylogimprove}
	R_v^{-1} s^{\beta+\frac{1}{2}} \ln(s)^{-\gamma} \leq A_v(W)(z) \leq R_v s^{\beta+\frac{1}{2}} \ln(s)^{-\gamma}.
	\end{equation}
	\item If there exists $R>0$ such that for all $z \in \{W>0\}$ near $\Sigma(W)$, letting $d=\dist(z,\Sigma(W))$
	\begin{equation}\label{bndrypolyexp}
	R^{-1} d^{\beta} \exp(-cd^{-\alpha}) \leq W(z) \leq R d^{\beta} \exp(-cd^{-\alpha}),
	\end{equation}
	then for any periodic direction $v \in \Sb^1$, there exists $c_v>0, R_v>0$ such that for all $z \in \{A_v(W) >0\}$ near $\Sigma(A_v(W))$, letting $s= \dist(z,\Sigma(A_v(W)))$
	\begin{equation}\label{bndrypolyexpimprove}
	\ed{ R_v^{-1} \exp(-c_v s^{-\alpha}) \leq A_v(W)(z) \leq R_v \exp(-c_v^{-1}s^{-\alpha}).}
	\end{equation}
\end{enumerate}
\end{lemma}
Note that in Case 1) the same improvement from $\beta$ to $\beta+\frac{1}{2}$ as in \cite{Sun23} is found, which produces an improvement in the energy decay rate. 

We proceed via geometric arguments that apply to the general setup of $\Sigma_{\omega}$ consisting of strictly convex curves, and then address the two cases individually. The geometric argument exactly follows that of \cite[Proposition 4.4]{Sun23}, but we must restate it in order to properly articulate the new arguments. 
\begin{proof}
As mentioned above, because $v$ is a periodic direction, there exists $p_0, q_0 \in \Zb, \gcd(p_0,q_0)=1$ such that $v= \frac{(p_0,q_0)}{\sqrt{p_0^2+q_0^2}}$. Now we perform the change of coordinates described in Subsection \ref{toruscoordinatessubsection}. Recall, the covering map $\pi_v: \T^2_v \ra \T^2$ is of degree $p_0^2+q_0^2$ and it is compatible with the periodicity of the two tori. That is a $2\pi$ periodic function $f$ on $\ed{\T}^2$ becomes a $T_v=2\pi \sqrt{p_0^2+q_0^2}$ periodic function $\pi_v^{*} f$ on $\ed{\T}^2_v$. Locally the coordinate map $(x,y) \in \T^2_v \mapsto z \in \ed{\T}^2$ is given by $xv^{\bot} + yv=z$. Furthermore, because $\pi_v$ is locally isometric, each component of the pre-image of the damped region $\pi^{-1}_v(\{W>0\})$ is still strictly convex with positive curvature, and the inequalities \eqref{bndrypolylog} and \eqref{bndrypolyexp} still hold with $d=\dist((x,y), \Sigma(\pi_v^* W))$ for $z$ near $\Sigma(\pi_v^* W)$.
 
 Define the averaging operator on the new torus $\T_v^2$ as 
$$
\ti{A}(F)(x,y)= \frac{1}{T_v} \int_0^{T_v} F(x,y+t) dt.
$$
Note that we average along $(0,1)$, which is the direction that $v$ is pulled back to. Furthermore 
\begin{align*}
\pi_v^*( A_v(W))(x,y) &= A_v(W)(xv^{\bot} + yv) = \frac{1}{T_v} \int_0^{T_v} W\left( x v^{\bot} + (y+t)v\right) dt \\
&= \frac{1}{T_v} \int_0^{T_v} (\pi_v^* W)(x, y+t) dt = \ti{A}(\pi^*_v W)(x,y).
\end{align*}
Therefore it is enough for us to prove \eqref{bndrypolylogimprove} and \eqref{bndrypolyexpimprove} for the averaged function on $\T^2_v$,  $\ti{A}(\pi^*(W))$. 
 
Therefore, without loss of generality, we may assume $v=(0,1)$ and assume that $\Omega:=\{W>0\}$ has $l$ connected components $\Omega_1, \ldots, \Omega_l$. We will call $W|_{\Omega_j}=W_j$ and can further state that such that the boundary $\Sigma(W_j)$ of each $\Omega_j$ has positive curvature. 
 
We first consider the case $l=1$. By translation invariance we may assume $\Omega_1=\{W_1>0\}$ is contained in the fundamental domain $(-m \pi, m\pi)_x \times (-M \pi, M \pi)_y$, so the function $A_v(W)$ can be identified as a periodic function on $\Rb_x$. 
 
Because $\Omega_1=\{W_1>0\}$ is strictly convex, the line $\mathcal{P}_x=\{(x,0)+t(0,1): t \in \Rb\}$ intersects $0,1,$ or $2$ points of $\Sigma(W_1)$. Thus for all $x$, the set $\mathcal{P}_x \cap \Omega_1$ has finite measure in $\Rb$ and the function $x \mapsto P(x):=\ed{\nu}_{\Rb}(\mathcal{P}_x \cap \Omega_1)$ is continuous. Furthermore by the convexity of $\Omega_1$, $P(x)$ is supported on a single interval $\ed{I=[a,b]} \subset (-m\pi, m\pi)$. Since $A(W_1)(x)=0$ if $P(x)=0$, the vanishing behavior of $A(W_1)$ is determined when $x$ is close to \ed{$a$ and $b$}. We will only estimate $A(W_1)$ for $x \in [\ed{a, a}+\e)$ as analogous arguments apply for $x$ near $\ed{b}$.
 
So now note, that because $P(x)=0$ for $x<\ed{a}$ and $P(x)>0$ for $x>\ed{a}$, then $\mathcal{P}_x$ is tangent to $\Sigma(W_1)$ at some point $z_0=(\ed{a},y_0)$. For a small enough $\e>0$ we may \ed{parameterize} $(x,y) \in \Sigma(W_1)$ for $x \in [\ed{a, a}+\e)$ by $x=\ed{a}+g(y)$, where $g$ is a function satisfying $g(y_0)=g'(y_0)=0$ and $g''(y_0)\geq c_0 > 0$, where the last condition holds because the curvature of $\Sigma(W_1)$ is strictly positive. Therefore, there exists a $C^1$ diffeomorphism $Y= \Phi(y)$ of a neighborhood of $y_0$, such that $\Phi(y_0)=0$ and $\Phi(g(y))=Y^2$.  Under this diffeomorphism the parametrization for $\Sigma(W_1)$ becomes $x=\ed{a} + Y^2$. 
 
From this, we have that for each $x \in (\ed{a,a}+\e)$, we can write $\mathcal{P}_x \cap \Sigma(W_1)=\{(x,l^{-}(x)), (x,l^{+}(x))\}$ where 
\begin{equation}
l^{\pm}(x) = \pm \Phi^{-1}(\sqrt{x-\ed{a}}).
\end{equation}

So now for $x \in \ed{(a, a}+\e)$ we have 
\begin{equation}
A(W_1)(x) = \frac{1}{2M \pi} \int_{-M\pi}^{M\pi} W_1(x,y) dy = \frac{1}{2\ed{M}\pi} \int_{l^-(x)}^{l^+(x)} W_1(x,y) dy.
\end{equation}
Case 1) When $W(z) \simeq d^{\beta} \ln(d^{-1})^{-\gamma}$, note that by Lemma \ref{distancelemma}, \ed{and the diffeomorphism $\Phi$} $d \simeq |x-(\ed{a}+g(y))|$ and so $W(z) \simeq |x-(\ed{a}+g(y))|^{\beta} \ln(|x-(\ed{a}+g(y))|^{-1})^{-\gamma}$. 
Therefore \ed{making the change of variables $Y=\Phi(y)$}
\begin{equation}
A(W_1)(x) \simeq \frac{1}{2M \pi} \int_{-\sqrt{x-\ed{a}}}^{\sqrt{x-\ed{a}}} (|x-\ed{a}-Y^2|)^{\beta} \ln(|x-\ed{a}-Y^2|^{-1})^{-\gamma} |(\Phi^{-1})'(Y)| dY.
\end{equation}
Replacing $x-\ed{a}$ by $X$ and using the symmetry of $Y^2$ 
\begin{equation}
A(W_1)(x) \simeq \int_0^{\sqrt{X}} (X-Y^2)^{\beta} \ln( (X-Y^2)^{-1})^{-\gamma} dY.
\end{equation}
For $Y \in [0,\sqrt{X}]$ we have $\sqrt{X} \leq \sqrt{X}+Y \leq 2 \sqrt{X}$, so
\begin{equation}
(X-Y^2) = (\sqrt{X}+Y)(\sqrt{X}-Y) \simeq \sqrt{X} (\sqrt{X}-Y) = (X-\sqrt{X}Y),
\end{equation}
and so
\begin{equation}
(X-Y^2)^{\beta} \ln( (X-Y^2)^{-1})^{-\gamma}  \simeq (X-X^{\frac{1}{2}}Y)^{\beta} \ln(X-X^{\frac{1}{2}} Y)^{-\gamma}.
\end{equation}
Therefore
\begin{equation}
A(W_1)(x) \simeq \int_0^{\sqrt{X}} (X-X^{\frac{1}{2}}Y)^{\beta} \ln( (X-X^{\frac{1}{2}} Y)^{-1})^{-\gamma} dY.
\end{equation}
Make the substitution $u=X-X^{\frac{1}{2}}Y$ 
\begin{equation}
A(W_1)(x) \simeq X^{-\frac{1}{2}} \int_0^X u^{\beta} \ln(u^{-1})^{-\gamma} du.
\end{equation}
Then by Lemma \ref{uppergammalem}, since $X$ is near $0$ 
\begin{equation}
A(W_1)(x) \simeq X^{\beta+\frac{1}{2}} \ln(X^{-1})^{-\gamma}.
\end{equation}
So indeed, replacing $X$ by $x-\ed{a}$
\begin{equation}
A(W_1)(x) \simeq_{\Phi, k,\beta, \gamma} (x-\ed{a})^{\beta+\frac{1}{2}} \ln\left((x-\ed{a})^{-1}\right)^{-\gamma}.
\end{equation}

Case 2) When $W(d) = d^{\beta} \exp(-cd^{-\alpha})$, note again by Lemma \ref{distancelemma} \ed{and the diffeomorphism}, $d \simeq |x-\ed{a}-g(y)|$ and so there exists $C_1>1$ such that 
\begin{equation}
C_1^{-1} |x-\ed{a}+g(y)|^{\beta} \exp(-C_1|x-\ed{a}-g(y)|^{-\alpha}) \leq W(z) \leq C_1 |x-\ed{a}+g(y)|^{\beta} \exp(-C_1^{-1} |x-\ed{a}-g(y)|^{-\alpha}).
\end{equation} 
Therefore
\begin{align}
&C_1^{-1} \int_{-\sqrt{x-\ed{a}}}^{\sqrt{x-\ed{a}}} |x-\ed{a}-Y^2|^{\beta} \exp(-C_1|x-\ed{a}-Y^2|^{-\alpha})  |(\Phi^{-1})'(Y)|dY \\
&\leq A_v(W)(x) \leq C_1 \int_{-\sqrt{x-\ed{a}}}^{\sqrt{x-\ed{a}}} |x-\ed{a}-Y^2|^{\beta} \exp(-C_1^{-1}|x-\ed{a}-Y^2|^{-\alpha})  |(\Phi^{-1})'(Y)|dY
\end{align}
Replacing $x-\ed{a}$ by $X$ and using the symmetry of $Y^2$
\begin{align}
&C^{-1}  \int_0^{\sqrt{X}} (X-Y^2)^{\beta} \exp(-C_1 (X-Y^2)^{-\alpha}) dY\\
& \leq A(W_1)(x) \leq C \int_0^{\sqrt{X}} (X-Y^2)^{\beta} \exp(-C_1^{-1} (X-Y^2)^{-\alpha}) dY.
\end{align}
As before $(X-X^{\frac{1}{2}} Y) \leq (X-Y^2) \leq 2 (X-X^{\frac{1}{2}} Y)$ for $Y \in [0,\sqrt{X}]$, so 
\begin{align}
 &\exp(-C_1^{-1} (X-Y^2)^{-\alpha}) \leq \exp( C_1^{-1} (2X-2X^{\frac{1}{2}} Y)^{-\alpha})\\
 \exp(-C_1(X-X^{1/2}Y)^{-\alpha}) \leq& \exp(-C_1(X-Y^2)^{-\alpha}).
\end{align}
We will first focus on the upper bound. Letting $c=2C_1^{\frac{1}{\alpha}}$
and making the substitution $u=c(X-X^{\frac{1}{2}}Y)$ 
\begin{equation}
A(W_1)(x) \lesssim X^{-\frac{1}{2}} \int_0^{cX} u^{\beta} \exp(-u^{-\alpha}) du.
\end{equation}
Then by Lemma \ref{uppergammalem}, since $X$ is small
\begin{equation}
A(W_1)(x) \lesssim X^{\frac{1}{2}+\beta+\alpha} \exp(-C_1^{-1} (2 X)^{-\alpha}).
\end{equation}
An analogous argument with $c=C_1^{-\frac{1}{\alpha}}$ gives 
\begin{equation}
A(W_1)(x) \gtrsim X^{\frac{1}{2}+\beta+\alpha} \exp(-C_1 X^{-\alpha}).
\end{equation}
\ed{Replacing $C_1$ by a larger $c_v$ allows us to drop the polynomial factors and obtain 
\begin{equation}
	\exp(-c_v x^{-\alpha}) \lesssim A(W_1)(x) \lesssim \exp(-c_v^{-1} x^{-2\alpha}).
\end{equation}}

To complete the proof, we consider the case where $l>1$, and we must show \eqref{bndrypolylogimprove} and \eqref{bndrypolyexpimprove}. We again follow the argument of \cite[Proposition 4.4]{Sun23}.  Note that the supports of $A(W_j)$ may overlap. By linearity
\begin{equation}
A(W_1+ \cdots + W_l) = A(W_1) + \cdots + A(W_l).
\end{equation}
For a fixed $x_0 \in \Sigma(A(W))$, define those indices $j$, where $A(W_j)$ is positive to the right (or left) of $x_0$ by 
\begin{align*}
S_+ :=\left\{j \in \{1, \ldots, l\}: \int_{x_0}^{x_0+\e} A(W_j)(x) dx > 0, \forall \e>0\right\} \\
S_-:=\left\{j \in \{1,\ldots, l\}: \int_{x_0-\e}^{x_0} A(W_j)(x) dx >0, \forall \e>0\right\}.
\end{align*}
Because $x_0 \in \Sigma(A(W))$, then $A(W_j)(x_0)=0$ for all $j$ and $A(W_j)>0$ near $x_0$ for at least one $j$, i.e. $S_+ \cup S_- \neq \emptyset$. 

Therefore for $j \in S_{\pm}$, and $x \mp x_0 >0$ with $x$ near $x_0$ we have 
\begin{equation}
\dist(x,\Sigma(A(W)))=\dist(x,x_0)=\dist(x,\Sigma(A(W_j))).
\end{equation} 
Because of this and applying our argument when $l=1$ to each $A(W_j)$, we have
\begin{equation}
f(\ed{\dist}(x,\Sigma(A(W)))) =f(\ed{\dist}(x,\Sigma(A(W_j)))) \simeq A(W_j)(x), \quad \text{ for all } x \mp x_0 >0 \text{ near } x_0,
\end{equation}
where $f$ is the desired boundary behavior in \eqref{bndrypolylogimprove} and  \eqref{bndrypolyexpimprove}. On the other hand, for all $j \not \in S_+ \cup S_-$, $A(W_j)=0$ in a neighborhood of $x_0$. Summing over $j \in S_+ \cup S_-$ we obtain the desired bound. 
\end{proof}

\ed{Our final averaging result shows that for damping supported on a superellipse the improvement to the growth behavior depends on the averaging direction and the exponents in the definition of the superellipse.}
\begin{lemma}\label{superellipseaveraging}
Suppose $W$ is positive on the interior of a general superellipse. That is for some $a,b>0, n_1 \geq n_2 \geq 2$
\begin{equation}
\{(x,y) \in \T^2: W(x,y)>0\} = \{(x,y) \in \T^2: \left| \frac{x}{a}\right|^{n_1} + \left| \frac{y}{b} \right|^{n_2} \leq 1\}.
\end{equation}
Let $v \in \Sb^1$ periodic
\begin{enumerate}
	\item Assume that there exists $C \geq 1$, such that for all $z \in \{W>0\}$ near $\Sigma(W)$, letting $d=\dist(z,\Sigma(W))$
	\begin{equation}
	C^{-1} d^{\beta} \ln(d^{-1})^{-\gamma} \leq W(z) \leq C d^{\beta} \ln(d^{-1})^{-\gamma} ,
	\end{equation}
	with $\beta>0, \gamma \in \Rb$, then for $v_1=(1,0)$ or $v_2=(0,1)$, there exists $C_v \geq 1$, such that for all $z \in \{A_v(W)>0\}$ near $\Sigma(A_v(W))$, letting $s=\dist(z,\Sigma(A_v(W)))$
	\begin{equation}
C_v^{-1} s^{\beta+\frac{1}{n_i}} \ln(s^{-1})^{-\gamma} \leq A_{v_i}(W)(z) \leq C_v s^{\beta+\frac{1}{n_i}} \ln(s^{-1})^{-\gamma}.
\end{equation}
	For all other $v \in \Sb^1$ periodic, there exists $C_v \geq 1$, such that for all $z \in \{A_v(W)>0\}$ near $\Sigma(A_v(W))$
	\begin{equation}
C_v^{-1} s^{\beta+\frac{1}{2}} \ln(s^{-1})^{-\gamma} \leq A_{v}(W)(z) \leq C_v s^{\beta+\frac{1}{2}} \ln(s^{-1})^{-\gamma}.
\end{equation}
	
	\item Assume that there exists $C \geq 1$, such that for all $z \in \{W>0\}$ near $\Sigma(W)$, letting $d=\dist(z,\Sigma(W))$
	\begin{equation}
	C^{-1} d^{\beta} \exp(-cd^{-\alpha}) \leq W(z) \leq C d^{\beta} \exp(-cd^{-\alpha}),
	\end{equation}
	with $\alpha, c>0, \beta \in \Rb$, then \ed{for all $v \in \Sb^1$ periodic}, there exists $c_v, R_v \geq 1$, such that for all $z \in \{A_v(W)>0\}$ near $\Sigma(A_v(W))$, letting $s=\dist(z,\Sigma(A_v(W)))$
	\begin{equation}
	\ed{R_v^{-1} \exp(-c_v s^{-\alpha}) \leq A_{v_i}(W)(z) \leq R_v \exp(-c_v^{-1} s^{-\alpha}).}
	\end{equation}
\end{enumerate}

\end{lemma}
\begin{proof}
The statement for $v \in \Sb^1$ periodic not including $v_1, v_2$, follows from the same argument as in Proposition \ref{avgconveximproveprop}. To give some additional detail, note that $\Sigma_W$ is strictly convex with positive curvature at all points other than $(\pm a,0)$ and $(0, \pm b)$. Furthermore, if $x_0 \in \Sigma(A_v(W))$ then for some $t_0 \in \Rb$, $x_0+t_0 v$ is a point where $v$ is tangent to $\Sigma(W)$. When $v \not\in\{(1,0),(0,1)\}$, these points of tangency cannot be $(\pm a, 0), (0,\pm b),$ so $\Sigma(W)$ is strictly convex at them, \ed{see Figure \ref{f:diagonalSuper}}. These points, $x_0+t_0 v$, take the role of \ed{the endpoints of the interval $I$ on which $P(x)$ is supported} in the proof of Proposition \ref{avgconveximproveprop} and so we can use the same proof.

\begin{figure}[h]
\centering
\begin{tikzpicture}
[
declare function={
  sxn(\t)= a*cos(\t r)^(2/n);
  syn(\t)= b*sin(\t r)^(2/n);
  sxm(\t)=a*cos(\t r)^(2/m);
  sym(\t)=b*sin(\t r)^(2/n);
  a=3;
  b=2;
  n=3;
  m=3;
}]

\draw[variable=\t, domain=0:pi/2, thick] plot ({sxn(\t)},{syn(\t)}) -- plot({-sxn(pi/2-\t)},{syn(pi/2-\t)}) -- plot({-sxn(\t)},{-syn(\t)}) -- plot({sxn(pi/2-\t)},{-syn(pi/2-\t)}) -- cycle;
\draw[persred, -, thick] (-4.009,0)--(-1,3.009);
\draw[persred, -, thick] (4.009,0)--(1,-3.009);
\end{tikzpicture}
\caption{Lines parallel to $v$ and tangent to $\Sigma_W$ in red. The points $x_0$ must lie along these red lines.}\label{f:diagonalSuper}
\end{figure}
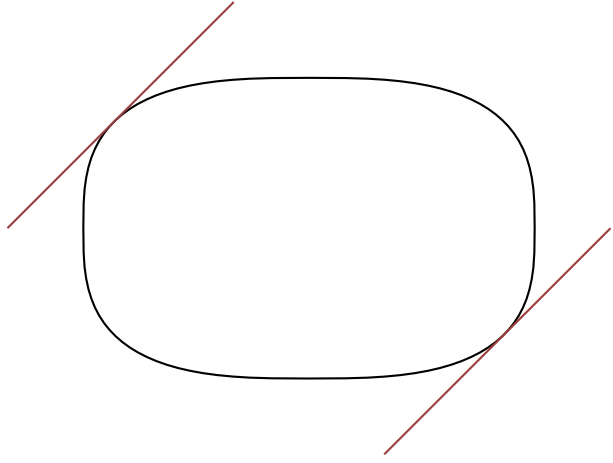

For the $v_i$, we will focus on $v_1$, since the argument for $v_2$ is analogous. The vanishing behavior of $A_{v_1}(W)(x)$ is determined when $x$ is close to $\pm a$. We will focus on $x\in(-a,-a+\e)$ as the analysis for $x$ near $a$ is similar. 
\ed{For a small enough $\e>0$ we may parameterize $(x,y) \in \Sigma(W_1)$ for $x \in [-a,-a+\e)$ by $x=a+g(y)$, where $g$ is a function satisfying $g(y_0)=g'(y_0)=\cdots = g^{(n_2-1)}=0$ and $g^{(n_2)}(y_0)\geq c_0 > 0$. Therefore, there exists a $C^1$ diffeomorphism $Y= \Phi(y)$ of a neighborhood of $y_0$, such that $\Phi(y_0)=0$ and $\Phi(g(y))=|Y|^{n_2}$.  Under this diffeomorphism the parametrization for $\Sigma(W_1)$ becomes $x=\ed{a} + |Y|^{n_2}$. }

Let $\mathcal{P}_x$ be the line in $\Rb^2$ passing through $(x,0)$ parallel to $e_2$. For each $x \in [-a, -a+\e)$ we have
\begin{equation}
 \Pc \cap \Sigma_W = \{(x,l^{-}(x)), (x,l^+(x))\}= \ed{\{(x,\pm \Phi((x+a)^{1/n_2})), (x,(x+a)^{1/n_2})\}.}
\end{equation}
1) When $W(z) \simeq d^{\beta} \ln(d^{-1})^{-\gamma}$, by Lemma \ref{distancelemma} and \ed{the diffeomorphism $\Phi$}, $d \simeq |x+a-|y|^{n_2}|$ for $x \in [-a, -a+\e)$ and so 
\begin{align}
A_{v_1}(W)(x) &= \frac{1}{2\pi} \int_{-\pi}^{\pi} W(x,y) dy = \frac{1}{2\pi} \int_{l^{-}(x)}^{l^+(x)} W(x,y) dy \\
&\simeq \int_{-(x+a)^{1/n_2}}^{(x+a)^{1/n_2}} |x+a-|\ed{Y}|^{n_2}|^{\beta} \ln(|x+a-|\ed{Y}|^{n_2}|^{-1})^{-\gamma} \ed{|(\Phi^{-1})'(Y)|} \ed{dY}. 
\end{align}
Replacing $x+a=X$ and using the symmetry of $\ed{|Y|}$ 
\begin{align}
A_{v_1}(W)(x) \simeq \int_0^{X^{1/n_2}} |X-\ed{Y}^{n_2}|^{\beta} \ln(|x-\ed{Y}^{n_2}|^{-1})^{-\gamma} \ed{dY}. 
\end{align}
Now for $\ed{Y} \in (0,X^{1/n_2})$ by Lemma \ref{diffpowerslemma}, there exists $C>0$ such that 
\begin{equation}\label{diffpowersellipse}
C^{-1}(X-X^{\frac{n_2-1}{n_2}}Y) \lesssim X-Y^{n_2} \lesssim C(X-X^{\frac{n_2-1}{n_2}}Y) .
\end{equation}
So then
\begin{equation}
A_{v_1}(W)(x) \simeq \int_0^{X^{1/n_2}} (X-X^{\frac{n_2-1}{n_2}} \ed{Y})^{\beta} \ln((X-X^{\frac{n_2-1}{n_2}} \ed{Y})^{-1})^{-\gamma} dy. 
\end{equation}
Let $u=X-X^{\frac{n_2-1}{n_2}}\ed{Y}$ 
\begin{equation}
A_{v_1}(W)(x) \simeq \int_0^X X^{\frac{1-n_2}{n_2}} u^{\beta} \ln(u^{-1})^{-\gamma} du \simeq X^{\frac{1-n_2}{n_2}} X^{\beta+1} \ln(\ed{X}^{-1})^{-\gamma}.
\end{equation}
where the final inequality follows from Lemma \ref{uppergammalem}.
So indeed 
\begin{equation}
A_{v_1}(W)(x) \simeq (x+\ed{a})^{\beta+\frac{1}{n_2}} \ln((x+\ed{a})^{-1})^{-\gamma}.
\end{equation}
2) When $W \simeq d^{\beta} \exp(-cd^{-\alpha})$ this follows an analogous proof with modifications similar to those in the proof of Lemma \ref{avgconveximproveprop}
\end{proof}
\subsection{Proofs of energy decay rates using averaging results}\label{proofconvexrate}
In this subsection we prove energy decay rates for damping satisfying non-polynomial growth conditions and supported on rectangles (Theorem  \ref{rectthm}), strictly convex sets (Theorem \ref{polylogconveximprovethm}) or superellipses (Theorem \ref{ellipsethm}).

We begin by proving energy decay for damping supported on a rectangle, Theorem \ref{rectthm} case 1. The proof uses Theorem \ref{averagedresolventthm} to reduce to resolvent estimates with the damping replaced by an average. These resolvent estimates are in turn proved by controlling the turn on behavior of the damping using Lemma \ref{averagingsquarelemma} and then invoking the $y$-invariant resolvent estimates of Theorem \ref{wideresolvethm}.
\begin{example}\label{rectex}(Proof of Theorem \ref{rectthm} case 1)
\ed{Recall} $W$ is supported on a rectangle, and if $x,y$ are local coordinates for the boundary of its support then $W \simeq x^{\beta_1} y^{\ed{\beta_2}} \ln(x^{-1})^{-\gamma_1} \ln(y^{-1})^{-\gamma_2}$. 

Consider $v \in \Sb^1$ periodic. When $v=(0,1)$, for $s$ a local coordinate for the boundary of $\{A_v(W)>0\}$ then $A_v(W) \simeq s_+^{\beta_1} \ln(s_+^{-1})^{-\gamma_1}$. Then by Theorem \ref{wideresolvethm} part 1 and Example \ref{polylog1dex} \ed{there exist $h_v>0$ such that for $h \in (0,h_v)$ we have}
\begin{equation}\label{polylogrectresolve}
\nm{\left(\frac{i}{h}-\Ac_v \right)^{-1}}_{\Lc(\Hc)} \lesssim h^{-\frac{\ed{\beta_1}+3}{\beta_1+2}} \ln(h^{-1})^{-\frac{\gamma_1}{\beta_1+2}}.
\end{equation}
\ed{Note that we did not assume that $W$ satisfies Assumption \ref{noninvarassumption} and so we are not controlling the derivatives of $A_v(W)$ directly. However, to apply Theorem \ref{wideresolvethm} it suffices that $A_v(W) \simeq s_+^{\beta_1} \ln(s_+^{-1})^{-\gamma_1}$ whose derivatives we do have control of.}
\ed{Arguing} similarly when $v=(1,0)$
\begin{equation}
\nm{\left(\frac{i}{h}-\Ac_v \right)^{-1}}_{\Lc(\Hc)} \lesssim h^{-\frac{\beta_2+3}{\beta_2+2}} \ln(h^{-1})^{-\frac{\gamma_2}{\beta_2+2}}.
\end{equation}
For all other $v$, by Lemma \ref{averagingsquarelemma},  we have
\begin{equation}
A_v(W)(s) \simeq s^{\beta_1+\beta_2+1} \ln(s^{-1})^{-\gamma_1-\gamma_2}.
\end{equation}
So by Theorem \ref{wideresolvethm} part 1 and Example \ref{polylog1dex}, \ed{there exist $h_v>0$ such that for $h \in (0,h_v)$ we have}
\begin{equation}
\nm{\left(\frac{i}{h}-\Ac_v \right)^{-1}}_{\Lc(\Hc)} \lesssim h^{-\frac{\beta_1+\beta_2+4}{\beta_1+\beta_2+3}} \ln(h^{-1})^{-\left(\frac{\gamma_1+\gamma_2}{\beta_1+\beta_2+3}\right)}.
\end{equation}
Note that by the hypotheses on $\beta_1, \beta_2, \gamma_1, \gamma_2$ we have
\begin{equation}
 h^{-\frac{\beta_1+3}{\beta_1+2}} \ln(h^{-1})^{-\frac{\gamma_1}{\beta_1+2}} \geq h^{-\frac{\beta_2+3}{\beta_2+2}} \ln(h^{-1})^{-\frac{\gamma_2}{\beta_2+2}} \geq h^{-\frac{\beta_1+\beta_2+4}{\beta_1+\beta_2+3}} \ln(h^{-1})^{-\left(\frac{\gamma_1+\gamma_2}{\beta_1+\beta_2+3}\right)}.
\end{equation}
Therefore, for all $v \in \Sb^1$ periodic \eqref{polylogrectresolve} holds.
Thus, by Theorem \ref{averagedresolventthm} and writing $h^{-1}=\lambda$, \ed{there exists $\lambda_0>0$ such that for $\lambda \in \Rb, |\lambda|\geq \lambda_0$ we have}
\begin{equation}
\nm{(i\lambda -\Ac)^{-1}}_{\Lc(\Hc)} \lesssim \lambda^{\frac{\beta_1+3}{\beta_1+2}} \ln(\lambda)^{-\frac{\gamma_1}{\beta_1+2}}.
\end{equation}
The right hand side is of positive increase in $\lambda$, and so energy decays at rate 
\begin{equation}
t^{-\frac{\beta_1+2}{\beta_1+3}} \ln(\ed{2}+t)^{-\frac{\gamma_1}{\beta_1+3}}.
\end{equation}
\qed
\end{example}
Next, the proof of case 2 of Theorem \ref{rectthm} follows the same rough outline, but has some meaningful differences, for example we must invoke Lemma \ref{averagingdbclemma} to ensure the derivative bound conditions are appropriately satisfied. This is due to different behavior of poly-exponential functions under averaging, and so we provide a complete proof here. 
\begin{example}\label{rectex1}(Proof of Theorem \ref{rectthm} case 2)
\ed{Recall} $W$ is supported on a rectangle and \ed{if $x$ and $y$ are local coordinates for the boundary of its support}, $W \simeq x^{\beta_1} y^{\beta_2} \exp(-c_1 x^{-\alpha_1}) \exp(-c_2 y^{-\alpha_2})$. Consider $v \in \Sb^1$ periodic. When $v=(0,1)$, for $s$ a local coordinate of the boundary of $\{A_v(W)>0\}$ then $A_v(s) \simeq s_+^{\beta_1} \exp(-c_1 s^{-\alpha_1}_+)$. Then by Theorem \ref{wideresolvethm} part 3 and Example \ref{polyexp1dex} \ed{there exists $h_v>0$ such that for $h \in (0,h_v)$ we have}
\begin{equation}
\nm{\left(\frac{i}{h}-\Ac_v \right)^{-1}}_{\Lc(\Hc)} \lesssim h^{-1} \ln(h^{-1})^{\frac{2\alpha_1+1}{\alpha_1}}.
\end{equation}
Similarly when $v=(1,0)$, then $A_v(W)(s) \simeq s^{\beta_2} \exp(-c_2 y^{-\alpha_2})$ and  \ed{there exists $h_v>0$ such that for $h \in (0,h_v)$ we have}
\begin{equation}
\nm{\left(\frac{i}{h}-\Ac_v \right)^{-1}}_{\Lc(\Hc)} \lesssim h^{-1} \ln(h^{-1})^{\frac{2\alpha_2+1}{\alpha_2}}.
\end{equation}
For all $v \ed{\not \in} \{(1,0), (0,1)\}$, by Lemma \ref{averagingsquarelemma} we have
\begin{equation}
\ed{C^{-1} \exp(-c_{v} s^{-\alpha_2}) \leq A_v(W)(s) \leq C \exp(-c_{v}^{-1} s^{-\alpha_2}).}
\end{equation}
Note that although the upper and lower bounds of $A_v(W)$ both satisfy the derivative bound conditions (Assumptions \ref{noninvarassumption} and \ref{widegenbadassumption}) with the same $q$ and $p$, they are not constant multiples of each other, and so we may not directly apply Theorem \ref{wideresolvethm} and Example \ref{polyexp1dex}.
Instead, \ed{recall} that $W$ satisfies Assumptions \ref{noninvarassumption} and \ref{widegenbadassumption} with $q(z) = \ln(z^{-1})^{-\left(\frac{\alpha_1+1}{\alpha_1}\right)}, p(z)=q(z)^2$. Therefore by Lemma \ref{averagingdbclemma} part 3, $A_v(W)$ satisfies Assumptions \ref{wideassumption} and \ref{widebadassumption} with the same $q$ and $p$. Finally 
 \begin{equation}
 \digamma_v(\d z) = \ed{\nu}( \{x \in \Sb^1: 0<A_v(W) < \d z\})  \simeq \ln(z^{-1})^{-\frac{1}{\alpha_2}} \geq q(z),
 \end{equation}
 so Assumption \ref{wideinversebassumption} is satisfied. Now following Example \ref{polyexp1dex}
\begin{align}
&R_2(z) \simeq z \ln(z^{-1})^{-\frac{2(\alpha_1+1)}{\alpha_1}}, \qquad \ti{R}_2^{-1}(h) \simeq h \ln(h^{-1})^{\frac{2(\alpha_1+1)}{\alpha_1}}.
\end{align}
Then by Theorem \ref{wideresolvethm} part 3, \ed{for all $v \in \Sb^1$ periodic there exists $h_v>0$ such that for $h \in (0,h_v)$ we have} 
\begin{equation}\label{exprectresolve}
\nm{\left(\frac{i}{h}-\Ac_v \right)^{-1}}_{\Lc(\Hc)} \lesssim h^{-1} \ln(h^{-1})^{\frac{2(\alpha_1+1)}{\alpha_1} - \frac{1}{\alpha_2}}.
\end{equation}
Then since $0 < \alpha_1 \leq \alpha_2$
\begin{equation}
\frac{2(\alpha_2+1)}{\alpha_2} -\frac{1}{\alpha_2} \leq \frac{2(\alpha_1+1)}{\alpha_1} -\frac{1}{\alpha_1}\leq \frac{2(\alpha_1+1)}{\alpha_1} -\frac{1}{\alpha_2}.
\end{equation}
So for all $v \in \Sb^1$ periodic, \eqref{exprectresolve} holds, and so by Theorem \ref{averagedresolventthm}, and writing $h^{-1}=\lambda$, \ed{there exists $\lambda_0>0$ such that for $\lambda \in \Rb, |\lambda| \geq \lambda_0$ we have }
\begin{equation}
\nm{(i\lambda-\Ac)^{-1}}_{\Lc(\Hc)} \leq \lambda \ln(\lambda)^{\frac{2(\alpha_1+1)}{\alpha_1} - \frac{1}{\alpha_2}}.
\end{equation}
This right hand side is of positive increase, so energy decays at rate 
\begin{equation}
r(t) = t^{-1} \ln(\ed{2}+t)^{\frac{2(\alpha_1+1)}{\alpha_1} - \frac{1}{\alpha_2}}.
\end{equation}
\qed
\end{example}

The proof of energy decay for damping supported on a strictly convex set, Theorem \ref{polylogconveximprovethm}, follows a similar outline. We again use Theorem \ref{averagedresolventthm} to reduce to resolvent estimates with the damping replaced by an average. These resolvent estimates are in turn proved by controlling the turn-on behavior of the damping using Lemma \ref{avgconveximproveprop} and then invoking the $y$-invariant resolvent estimates of Theorem \ref{wideresolvethm}. One notable change is that Lemma \ref{avgconveximproveprop} provides turn-on behavior that does not depend on the averaging direction $v$. 
\begin{example}\label{convexex}(Proof of Theorem \ref{polylogconveximprovethm})
Assume $\{W>0\}$ is a locally strictly convex set with positive curvature and for $d$ a local coordinate for the boundary of $\{W>0\}$, $W \ed{\simeq} d^{\beta} \ln(d^{-1})^{-\gamma}$. 

By Lemma \ref{avgconveximproveprop} for any periodic direction $v \in \Sb^1$, for $s$ a local coordinate of the boundary of $\{A_v(W) >0\}$ we have 
\begin{equation}
R^{-1}s^{\beta+\frac{1}{2}} \ln(s^{-1})^{-\gamma} \leq A_v(W)(s) \leq R s^{\beta+\frac{1}{2}} \ln(s^{-1})^{-\gamma}.
\end{equation}
So by Theorem \ref{wideresolvethm} part 1 and Example \ref{polylog1dex}, for all $v \in \Sb^1$ periodic \ed{there exists $h_v>0$ such that for $h \in (0,h_v)$ we have}
\begin{equation}
\nm{\left(\frac{i}{h}-\Ac_v \right)^{-1}}_{\Lc(\Hc)} \lesssim h^{-\frac{\beta+\frac{1}{2}+3}{\beta+\frac{1}{2}+2}} \ln(h^{-1})^{-\frac{\gamma}{\beta+\frac{1}{2}+2}}.
\end{equation}
Thus, by Theorem \ref{averagedresolventthm}, and writing $h^{-1}=\lambda$, \ed{there exists $\lambda_0>0$, such that for all $\lambda \in \Rb, |\lambda| \geq \lambda_0$ we have}
\begin{equation}
\nm{(i\lambda -\Ac)^{-1}}_{\Lc(\Hc)} \lesssim \lambda^{\frac{\beta+\frac{1}{2}+3}{\beta+\frac{1}{2}+2}} \ln(\lambda)^{-\frac{\gamma}{\beta+\frac{1}{2}+2}}.
\end{equation}
The right hand side is of positive increase in $\lambda$, and so energy decays at rate 
\begin{equation}
t^{-\frac{\beta+\frac{1}{2}+2}{\beta+\frac{1}{2}+3}} \ln(\ed{2}+t)^{-\frac{\gamma}{\beta+\frac{1}{2}+3}}.
\end{equation}

To see the second statement, we begin by noting that $W$ satisfies Assumptions \ref{noninvarassumption} and \ref{widegenbadassumption} with $q(z) =\ln(z^{-1})^{-(\frac{\alpha+1}{\alpha})}, p(z)=q(z)^2$. Therefore by Lemma \ref{averagingdbclemma} part 3, $A_v(W)$ satisfies Assumptions \ref{wideassumption} and \ref{widebadassumption} with the same $q$ and $p$, for all $v \in \Sb^1$ periodic. Now by Lemma \ref{avgconveximproveprop} part 2), $\ed{A_v(W) \leq R_v \exp(-c_v^{-1} s^{-\alpha})}$ and so 
 \begin{equation}
 \digamma_v(\d z) = \ed{\nu}( \{x \in \Sb^1: 0<A_v(W) < \d z\})  \simeq \ln(z^{-1})^{-\frac{1}{\alpha}} \geq q(z),
 \end{equation}
and Assumption \ref{wideinversebassumption} is satisfied. Then by Example \ref{polyexp1dex} and Theorem \ref{wideresolvethm} part 3, for all $v \in \Sb^1$ periodic, \ed{there exists $h_v>0$ such that for $h \in (0,h_v)$ we have}
\begin{equation}
\nm{\ed{\left( \frac{i}{h} -\Ac_v\right)}^{-1}}_{\Lc(\Hc)} \lesssim h^{-1} \ln(h^{-1})^{\frac{2\alpha+1}{\alpha}}.
\end{equation}
\ed{Thus} by Theorem \ref{averagedresolventthm}, and writing $h^{-1}=\lambda$, \ed{there exists $\lambda_0>0$ such that for $\lambda\in \Rb, |\lambda| \geq \lambda_0$ we have}
\begin{equation}
\nm{(i\lambda -\Ac)^{-1}}_{\Lc(\Hc)} \lesssim \lambda^{-1} \ln(\lambda)^{-\frac{2\alpha+1}{\alpha}}.
\end{equation}
The right hand side is of positive increase in $\lambda$ and so energy decays at rate 
\begin{equation}
r(t) = t^{-1} \ln(\ed{2}+t)^{\frac{2\alpha+1}{\alpha}}.
\end{equation}
\qed
\end{example}
The proof of energy decay for damping supported on a superellipse, Theorem \ref{ellipsethm}, follows a similar outline. We again use Theorem \ref{averagedresolventthm} to reduce to resolvent estimates with the damping replaced by an average. These resolvent estimates are in turn proved by controlling the turn-on behavior of the damping using Lemma \ref{avgconveximproveprop} and then invoking the $y$-invariant resolvent estimates of Theorem \ref{wideresolvethm}. As in the proof of Theorem \ref{rectthm} there are different resolvent estimates for $A_v(W)$ depending on the direction $v$. 

\begin{example}\label{superellipseexample}(Proof of Theorem \ref{ellipsethm})
Assume $\{W>0\}$ is a superellipse and for $d$ a local coordinate for the boundary of $\{W>0\}$, $W \ed{\simeq} d^{\beta} \ln(d^{-1})^{-\gamma}$. 

Consider $v \in \Sb^1$ periodic. When $v \neq (1,0)$ or $(0,1)$, then by Lemma \ref{superellipseaveraging} for $s$ a local coordinate for the boundary of $\{A_v(W)>0\}$, we have $A_v(W)(s) \simeq s^{\beta+\frac{1}{2}} \ln(s^{-1})^{-\gamma}$. Thus by Theorem \ref{wideresolvethm} and Example \ref{polylog1dex} \ed{there exists $h_v>0$ such that for $h \in (0,h_v)$ we have}
\begin{equation}
\nm{\left(\frac{i}{h}-\Ac_v \right)^{-1}}_{\Lc(\Hc)} \lesssim h^{-\frac{\beta+\frac{1}{2}+3}{\beta+\frac{1}{2}+2}} \ln(h^{-1})^{-\frac{\gamma}{\beta+\frac{1}{2}+2}}.
\end{equation}
Similarly, for $v_1=(1,0)$ and $v_2=(0,1)$, and letting $n_1=n, n_2=m$, by Lemma \ref{superellipseaveraging} $A_{v_i}(s) \simeq s^{\beta+\frac{1}{n_i}} \ln(s^{-1})^{-\gamma}$. Thus by Theorem \ref{wideresolvethm} part 1 and Example \ref{polylog1dex} \ed{there exists $h_i>0$ such that for $h \in (0,h_i)$ we have}
\begin{equation}
\nm{\left(\frac{i}{h}-\Ac_v \right)^{-1}}_{\Lc(\Hc)} \lesssim h^{-\frac{\beta+\frac{1}{n_i}+3}{\beta+\frac{1}{n_i}+2}} \ln(h^{-1})^{-\frac{\gamma}{\beta+\frac{1}{n_i}+2}}.
\end{equation}
Since $n=n_1 \geq n_2=m \geq 2$
\begin{equation}
h^{-\frac{\beta+\frac{1}{2}+3}{\beta+\frac{1}{2}+2} } \ln(h^{-1})^{-\frac{\gamma}{\beta+\frac{1}{2}+2}} \leq h^{-\frac{\beta+\frac{1}{m}+3}{\beta+\frac{1}{m}+2} }\ln(h^{-1})^{-\frac{\gamma}{\beta+\frac{1}{m}+2}} \leq h^{-\frac{\beta+\frac{1}{n}+3}{\beta+\frac{1}{n}+2}} \ln(h^{-1})^{-\frac{\gamma}{\beta+\frac{1}{n}+2}}.
\end{equation}
Therefore, for all $v \in \Sb^1$ periodic \ed{there exists $h_v>0$ such that for $h \in (0,h_v)$ we have}
\begin{equation}
\nm{\left(\ed{\frac{i}{h}-\Ac_v }\right)^{-1}}_{\Lc(\Hc)} \lesssim h^{-\frac{\beta+\frac{1}{n}+2}{\beta+\frac{1}{n}+2}} \ln(h^{-1})^{-\frac{\gamma}{\beta+\frac{1}{n}+2}}.
\end{equation}
Thus by Theorem \ref{averagedresolventthm} and setting $\lambda=h^{-1}$, \ed{there exists $\lambda_0>0$ such that for $\lambda \in \Rb, |\lambda| \geq \lambda_0$ we have}
\begin{equation}
\nm{(i\lambda -\Ac)^{-1}}_{\Lc(\Hc)} \lesssim \lambda^{\frac{\beta+\frac{1}{n}+3}{\beta+\frac{1}{n}+2}} \ln(\lambda)^{-\frac{\gamma}{\beta+\frac{1}{n}+2}}.
\end{equation}
The right hand side is of positive increase in $\lambda$ and so energy decays at rate 
\begin{equation}
r(t)=t^{-\frac{\beta+\frac{1}{n}+2}{\beta+\frac{1}{n}+3}} \ln(\ed{2}+t)^{-\frac{\gamma}{\beta+\frac{1}{n}+3}}.
\end{equation}
The second statement follows an analogous argument, modifying as we did in Example \ref{convexex}. \qed
\end{example}

\section{Eigenfunction Concentration on Product Manifolds}\label{eigsec}
A necessary ingredient for the proof of Theorem \ref{thinresolvethm} is a concentration estimate for Laplace eigenfunctions on product manifolds. This is similar to \cite{BurqZuily2015}[Theorem 1.1], but with a modification to allow for non-polynomial scales. 

Before stating the nonconcentration result, we define the product setup we will use.

Let $(M_j, g_j), j=1,2,$ be compact Riemannian manifolds. Let $\mathcal{M}=M_1 \times M_2, \g=g_1 \otimes g_2$ be the product, and use $d_j$ (respectively $d$) to denote the geodesic distance on $M_j$ (resp. $\mathcal{M}$). Let $\Delta_{\g}$ be the Laplace-Beltrami operator on $(\mathcal{M}, \g$) and assume that $M_j$ is $W^{j,\infty}$ and $g_j \in W^{j-1, \infty}$ for $j=1,2$. 

Let $q_0 \in M_2$ and $\Sigma \in M_1 \times \{q_0\}$. For $\eta>0$ define
$$
N_{\eta} = \{m = (p,q) \in \mathcal{M}: d(m,\Sigma) < \eta) = M_1 \times \{q \in M_2: d_2(q,q_0) < \eta\}.
$$
Now we can state the eigenfunction nonconcentration result at non-polynomial scales. This is \cite[Theorem 1.1]{BurqZuily2015} with $h^{\d}$ replaced by $\rho(h)$, and the outer radius on the right hand side terms changed from $2\rho(h)$ to $\e \rho(h)$. For completeness, we include a proof of this proposition here, although we use the same approach.
\begin{proposition} \label{bzprop}
For any $\e>0$, and any function $\rho(h): (0,\infty) \ra (0,\infty)$ with $\limh \rho(h)=0$, there exists $C, h_0>0$, such that, for all $h \in (0, h_0)$, and every $\psi \in L^2\ed{(\mathcal{M})}$
$$
\nm{\psi}_{L^2(N_{\rho(h)})} \leq C \left( \nm{\psi}_{L^2(N_{\e\rho(h)}\backslash N_{\rho(h)})} + \frac{\rho(h)^2}{h^2} \nm{ (h^2\Delta + 1) \psi}_{L^2(N_{\e\rho(h)})} \right).
$$
\end{proposition}

Following the strategy of \cite{BurqZuily2015, BurqZworski2004} we can use eigenfunctions of $M_1$ to reduce \ed{to} a simpler problem on $M_2$.  

Let $B(q_0, r) \subset M_2$ be the ball of radius $r$ (with respect to $d_2$) centered at $q_0$. The following is analogous to \cite[Proposition 3.1]{BurqZuily2015}, with similar replacements as described above. 

\begin{proposition}\label{mtwoconcprop}
For any $\e>0$, and any function $\rho(h): (0,\infty) \ra (0,\infty)$ with $\limh \rho(h)=0$, there exists $C, h_0>0$, such that, for all $h \in (0, h_0)$, every $\tau \in \Rb$ and every $U \in H^2(M_2)$
$$
\nm{U}_{L^2(B(q_0, \rho(h)))} \leq C \left( \nm{U}_{L^2(B(q_0,\e \rho(h)) \backslash B(q_0, \rho(h)))} + \rho(h)^2 \nm{(-\Delta_{g_2} - \tau) U}_{L^2(B(q_0, \e \rho(h)))} \right).
$$
\end{proposition}
\begin{proof}[Proof of Proposition \ref{bzprop} via Proposition \ref{mtwoconcprop}]
Note that on a compact Riemannian manifold with an $L^{\infty}$ metric, the Laplace-Beltrami operator can be defined via a quadratic form on $H^1$. It is self adjoint with its natural domain and has compact resolvent. Because of this, there exists an orthonormal basis of $L^2(M_1)$ of eigenfunctions $e_n$ of $\Delta_{M_1}$ with eigenvalues $-\lambda_n^2$. 

So now, for $\psi \in \ed{L^2(\mathcal{M})}$, we can decompose it with respect to this orthonormal basis, by writing $\hat{\psi}_n(q) = \<\psi(\cdot, q), e_n(\cdot)\>_{L^2(M_1)}$. Then $\psi(p,q) = \sum_{n \in \Nb} \hat{\psi}_n(q) e_n(p)$ and for $r>0$, in analog to the Plancherel inequality
\begin{equation}\label{plancherel}
\nm{\psi}_{L^2(N_r)}^2 = \nm{\psi}^2_{L^2(M_1 \times B(q_0, r))} = \sum_{n \in \Nb} \nm{\hat{\psi}_n}_{L^2(B(q_0,r))}^2.
\end{equation}
Now set $F=(h^2 \Delta_g+1) \psi$ and pair both sides of the equation with $e_n$ to see 
$$
(-h^2 \Delta_{g_2} + h^2 \lambda_n^2-1) \hat{\psi}_n =\hat{F}_n.
$$
Set $\tau= h^{-2} -\lambda_n^2$ to rewrite this as $(-\Delta_{g_2} - \tau) \hat{\psi}_n = h^{-2} \hat{F}_n$, and then 
apply Proposition \ref{mtwoconcprop} to obtain
$$
\nm{\hat{\psi}_n}^2_{L^2(B(q_0, \rho(h)))} \leq C \left( \nm{\hat{\psi}_n}^2_{L^2(B(q_0, \e \rho(h)) \backslash B(q_0, \rho(h)))} + \rho(h)^4 \nm{h^{-2} \hat{F}_n}_{L^2(B(q_0,\e \rho(h)))}^2 \right).
$$
Taking the sum over $n \in \Nb$, applying \eqref{plancherel}, and taking square roots of both sides gives the desired estimate.
\end{proof}

\begin{proof}[Proof of Proposition \ref{mtwoconcprop}]
We follow the proof of \cite[Proposition 3.1]{BurqZuily2015}. The problem is local near $q_0$, so we will use a diffeomorphism to work in a neighborhood of the origin in $\Rb_z^k$ (where $k=\dim(M_2)$) such that the new metric $g$ satisfies $g|_{z=0} =\operatorname{Id}$. 

Now making the change of variables $x = \frac{z}{\rho(h)}$ we set 
\begin{align*}
u(x) = U(\rho(h)x) = U(z), \\
G(z)=(-\Delta_{g} - \tau) U(z), \\
F(x) = G(\rho(h)x) = G(z), \\
g^{h}(x) = g(\rho(h) x) = g(z).
\end{align*}
This gives the new equation
\begin{equation}
(-\Delta_{g^h} - \rho(h)^2 \tau ) u(x) = \rho(h)^2 F(x). 
\end{equation}
Note that because $g|_{z=0}$, the metric $g^h$ converges in the Lipschitz topology to the flat metric $g_0 = \operatorname{Id}$ as $h \ra 0^+$. Proposition \ref{mtwoconcprop} then follows immediately from
\begin{proposition} For $\ed{\e>1}$, the $g_n$ be a family of metrics on $B(0,\e) \subset \Rb^k$, which converges to the flat metric as $n \ra \infty$. Then, there exists $C(\e)>0, N_0>0$ such that for every $n \geq N_0, \tau \in \Rb$, $u \in H^2(B(0,\e))$
\begin{equation}
\nm{u}_{L^2(B(0,1))} \leq C \left( \nm{u}_{L^2(B(0,\e) \backslash B(0,1))} + \frac{1}{1+|\tau|^{\frac{1}{2}}} \nm{\ed{(-\Delta_{g_n} - \tau) u}}_{L^2(B(0,\e))}\right).\end{equation}
\end{proposition}
This is exactly \cite[Proposition 3.2]{BurqZuily2015} with $\e>1$ replacing $2$. Making this change does not alter the proof of this proposition. It is only necessary that $\e>1$, so that cutoff functions can turn off between $1$ and $\e$. Because of this we do not reproduce their argument here.
\end{proof}
\section{Proofs when damping is $0$ only along geodesics}\label{thinsec}
In this section we prove Theorems \ref{thinresolvethm} and \ref{thinqmthm}. 
By Lemma \ref{resolvequivlem} to prove Theorem \ref{thinresolvethm} it is enough to prove. 
\begin{proposition}\label{thinresolventprop}
With the same geometric assumptions and assumptions on $V$ as in Theorem \ref{thinresolvethm}, for all $\e>1$, there exist $C, h_0>0$ such that for all $0 < h \leq h_0$
\begin{equation}\label{georeseq}
\ltwo{u} \leq \ed{ C\max\left( \frac{h}{V(\rho(h))}, \rho^2(h)\right) \ltwo{\left(-\Delta + \frac{i}{h} W - \frac{1}{h^2} \right) u},}
\end{equation}
where $\rho(h) = \ti{R}_{\e}^{-1}(h)$, for $R_{\e}(z) = z^2 \sqrt{V(z) V(\e z)}$. 
\end{proposition}
The core idea of the proof is to estimate $u$ separately where $V$ is small and where $V$ is large. Where $V$ is large we use an a priori estimate to control $u$. Because $V$ is only $0$ at $x=0$, it is small only on in a neighborhood of $x=0$, and so $u$ can be estimated there using the eigenfunction concentration result Proposition \ref{bzprop}. 
The definition of $R_{\e}(z)$ is needed to balance terms in \eqref{whyredef}. Note that $\limh \ti{R}_{\e}^{-1}(h)=0$. 

The strategy follows that of \cite[Proposition 2.1]{BurqZuily2015} with $\rho(h)$ replacing their $h^{\d}$, although the introduction of non-polynomial $\rho(h)$ requires a departure from their argument. 
\begin{proof}
Let 
\begin{equation}\label{fuM}
f = \left(-\Delta + \frac{i}{h} W - \frac{1}{h^2} \right) u.
\end{equation}
Multiplying both sides by $\bar{u}$, integrating by parts on $M$ and taking the real and imaginary parts
\begin{align}
\ltwom{W^{\frac{1}{2}}u} \leq \ed{h^{\frac{1}{2}} \ltwom{u}^{\frac{1}{2}} \ltwom{f}^{\frac{1}{2}}, }\label{wuM} \\
h^2 \ltwom{\nabla_{g} u}^2 \leq \ltwom{u}^2 + h^2 \ltwom{u} \ltwom{f} \label{graduM}.
\end{align}
Now, in the neighborhood $\mathcal{U}$ of $T$ using the isometry $\Psi$, we set 
\begin{equation}\label{isoM}
\phi(p,x) = u(\Psi^{-1}(p,z)), \quad \ti{W}(p,x) = W(\Psi^{-1}(p,x)), \quad \ti{f}(p,x) = f(\Psi^{-1}(p,x)).
\end{equation}
Then from \eqref{fuM} we have on $M_1 \times B(0,1)$, after rearranging 
$$
(h^2\Delta_{\ti{g}} +1) \phi = i h \ti{W} \phi - h^2 \ti{f}.
$$
Now let $\rho(h):(0,\infty) \ra (0,\infty)$ such that $\limh \rho(h)=0$. Note, that we eventually choose $\rho(h) =\ti{R}^{-1}_{\e}(h)$. 

On one hand for $\{|x| \geq \rho(h)\}$, since $V$ is increasing $V(x) \geq V(\rho(h))$ so 
\begin{equation}\label{ncgeq}
\nm{\phi}_{L^2(M_1 \times \{\rho(h) \leq |x|\})} \leq \nm{ \frac{V^{\frac{1}{2}}(x)}{V^{\frac{1}{2}}(\rho(h))} \phi }_{L^2(M_1 \times \{\rho(h) \leq |x|\})} \leq \frac{C}{V(\rho(h))^{\frac{1}{2}}} \nm{W^{\frac{1}{2}} \phi}_{L^2(M_1 \times B(0,1))}.
\end{equation}
On the other hand, applying Proposition \ref{bzprop} with $\Sigma = M_1 \times \{0\}$
\begin{align}
\nm{\phi}&_{L^2(M_1 \times \{|x| \leq \rho(h)\})}\leq C \left( \nm{\phi}_{L^2(M_1 \times \{\rho(h) \leq |x| \leq \e \rho(h)\})} + \frac{\rho(h)^2}{h^2} \nm{ih \ti{W} \phi - h^2 \ti{f}}_{L^2(M_1\times \{|x| \leq \e \rho(h)\})} \right) \nonumber\\
&\leq C \nm{\phi}_{L^2(M_1 \times \{\rho(h) \leq |x|\})} + C\rho(h)^2 \nm{\ti{f}}_{L^2(M_1 \times B(0,1))}+ \frac{C \rho(h)^2}{h} \nm{\ti{W} \phi}_{L^2(M_1 \times \{|x| \leq \e\rho(h)\})} \nonumber\\
&\leq \left( \frac{C}{V(\rho(h))^{\frac{1}{2}}} + \frac{C \rho(h)^2}{h} V(\e \rho(h))^{\frac{1}{2}} \right) \nm{\ti{W}^{\frac{1}{2}} \phi}_{L^2(M_1 \times B(0,1))} + C \rho(h)^2 \nm{\ti{f}}_{L^2(M_1 \times B(0,1))}, \label{ngeq}
\end{align}
where we used that $\ti{W}^{\frac{1}{2}} \leq C V^{\frac{1}{2}}(\e \rho(h))$ on $\{|x| \leq \e\rho(h)\}$ and $\frac{\ed{\ti W(x)}}{V(\rho(h))} \geq C$ on $\{\rho(h) \leq |x|\}$. 

Combining \eqref{ncgeq} and \eqref{ngeq}
\begin{align*}
\nm{\phi}_{L^2(M_1 \times B(0,1))} \leq& C \left( \frac{1}{V(\rho(h))^{\frac{1}{2}}} + \frac{\rho(h)^2 V(\e \rho(h))^{\frac{1}{2}}}{h} \right) \nm{\ti{W}^{\frac{1}{2}} \phi}_{L^2(M_1 \times B(0,1))} \\
&+ C \rho(h)^2 \nm{\ti{f}}_{L^2(M_1 \times B(0,1))}.
\end{align*}
Now applying \eqref{isoM} and \eqref{wuM} 
\begin{equation}
\nm{u}_{L^2(\mathcal{U})} \leq C \left( \frac{h^{\frac{1}{2}}}{V(\rho(h))^{\frac{1}{2}}} + \frac{\rho(h)^2 V(\e \rho(h))^{\frac{1}{2}}}{h^{\frac{1}{2}}} \right) \ltwom{u}^{\frac{1}{2}} \ltwom{f}^{\frac{1}{2}} + C \rho(h)^2 \ltwom{f}.
\end{equation}
Now since $R_{\e}(z) = z^2 \sqrt{V(z) V(\e z)}$, setting $\rho(h) = \ti{R}_{\e}^{-1}(h)$ gives $h \simeq \rho(h)^2 \sqrt{V(\rho(h)) V(\e \rho(h))}$ and so 
\begin{equation}\label{whyredef}
\frac{h^{\frac{1}{2}}}{V(\rho(h))^{\frac{1}{2}}} \simeq \frac{\rho(h)^2 V(\e \rho(h))^{\frac{1}{2}}}{h^{\frac{1}{2}}}.
\end{equation}
Therefore
\begin{equation}\label{uUgeoeq}
\nm{u}_{L^2(\mathcal{U})} \leq C \frac{h^{\frac{1}{2}}}{V(\rho(h))^{\frac{1}{2}}} \ltwom{u}^{\frac{1}{2}} \ltwom{f}^{\frac{1}{2}} + C \rho(h)^2 \ltwom{f}.
\end{equation}
Now to conclude the proof we will proceed by contradiction. 

If \eqref{georeseq} were not true, then there would exist sequences $u_n \in H^2(M), f_n \in L^2(M), h_n \ra 0^+$ such that 
$$
\left(-\Delta_g + \frac{i}{h} W - \frac{1}{h_n^2} \right) u_n = f_n, \quad 1 = \ltwom{u_n} > n \max \left( \frac{h_n}{V(\rho(h_n))}, \rho(h_n)^2\right) \ltwom{f_n}.
$$
So $(-h_n^2 \Delta_g + i h_n W - 1) u_n = h_n^2 f_n$ and $\ltwo{h_n^2 f_n} = o\left(\min\left( h_n V(\rho(h_n)), \frac{h_n^2}{\rho(h_n)}\right)\right)$. In particular 
\begin{equation}\label{hnfneq}
\ltwo{h_n^2 f_n} = o(h_n),
\end{equation}
because $\limh \rho(h_n)=0$ and $V(0)=0$. We could also write $\ltwo{f_n}=o(\min(h_n^{-1} V(\rho(h_n)), \rho(h_n)^{-1})$, and then from \eqref{uUgeoeq} 
\begin{equation}\label{unconceq}
\limn \nm{u_n}_{L^2(\mathcal{U})} =0.
\end{equation}

Now note that since $(u_n)$ is bounded in $L^2(M)$, up to replacement by a subsequence, there exists a semiclassical measure $\mu$ on $S^* M$, such that for any $a \in C_0^{\infty}(S^*M)$
\begin{equation}
\limn \<a(x,h_nD_x) u_n, u_n\>_{L^2(M)} = \int_{S^*M} a(x,\xi) d \mu,
\end{equation}
again see \cite[Chapter 5]{Zworski2012}. From \eqref{graduM}, the sequence $u_n$ is $(h_n)$ oscillating, so any $\mu$ has total mass $ =\limn \ltwom{u_n}=1$ \cite[Proposition 4]{Burq1997}. 

From \eqref{wuM} and \eqref{hnfneq}, noting that $W \leq C W^{\frac{1}{2}}$ 
\begin{equation}
(-h_n^2 \Delta - 1) u_n = -i h_n W u_n + h_n^2 f_n = o(h_n)_{L^2}.
\end{equation}
Therefore by \cite[Proposition 4.4]{Burq2002} the measure $\mu$ is invariant under the bicharacteristic flow. By \eqref{wuM} $\mu$ is $0$ on $S^* \Omega$ and so by this invariance under the flow, it is $0$ on $\mathcal{GC}$. However, by \eqref{unconceq} it is also $0$ on $S^* \mathcal{U} \supset \mathcal{T}$, and since $S^*M = \mathcal{GC} \cup \mathcal{T}$, then $\mu$ is identically $0$. This contradicts that $\mu$ has total mass 1, so the desired conclusion must be true. 
\end{proof}
By Lemma \ref{resolvequivlem} and Proposition \ref{bdlemma}, to prove Theorem \ref{thinqmthm} it is enough to prove the following proposition.

\begin{proposition}\label{thinqmprop}
Under the same geometric assumptions as in Theorem \ref{thinresolvethm}, if there exists an increasing function $V:[0,\infty) \ra [0,\infty)$ with $V(0)=0$ and  $C>0$ such that $W\left(\Psi^{-1}(p,x)\right) \leq C V(|x|)$ for all $(p,x) \in M_1 \times B(0,1)$,  and recalling $R(z)=z^2 V(z)$, then there exist sequences $h_n \ra 0$ and $\{v_n\} \in L^2(\T^2)$ such that 
\begin{equation}
\ltwom{\left(-\Delta+\frac{i}{h_n}W - \frac{1}{h_n^2}\right) v_n} \lesssim \frac{V(\ti{R}^{-1}(h_n))}{h_n}, \quad \ltwom{v_n}=1.
\end{equation}
\end{proposition}
The definition of $R(z)$ is needed to balance terms in \eqref{whyrdef}. Note also that $\limh \ti{R}^{-1}(h)=0$. 
To construct the quasimodes we take a cutoff function centered at $x=0$ in $B(0,1)$ and then further concentrate its support by rescaling via $\ti{R}^{-1}(h)$. Its size as a quasimode is then computed directly.
\begin{proof}
Throughout we will write $\rho(h)=\ti{R}^{-1}(h)$.

In the neighborhood $\mathcal{U}$ of $T$ we use the isometry $\Psi$ to work with $(p,x) \in M_1 \times B(0,1)$, where $B(0,1)$ is the unit ball in $\Rb^{d-k}$. We will define the quasimodes to be a product of an eigenfunction of $M_1$ and a smooth cutoff compactly supported in $B(0,1)$. 

As in the proof of Proposition \ref{bzprop}, there exists an orthonormal basis $e_n$ of $L^2(M_1)$ of eigenfunctions of $\Delta_{g_1}$ with eigenvalues $\lambda_n \ra \infty$. Define the sequence $h_n=\lambda_n^{-1}$, so $e_n$ solves $(-\Delta_{g_1} - \frac{1}{h_n^2}) e_n=0$. For ease of notation we will write $h:=h_n$. 
Define $\chi  \in C_c^{\infty}(B(0,1), [0,1])$. Then define 
\begin{align}
u_n(p,x) &= \frac{ c_1}{\rho(h)^{\frac{d-k}{2}}} \chi (x \rho(h)^{-1} ) e_n(p), \quad (p,x) \in M_1 \times B(0,1) \\
 v_n(X)& = u_n(\Psi(X)), \quad X \in \mathcal{U},
\end{align}
where $c_1 = \ed{\|\chi\|_{L^2(B(0,1))}^{-1}}$, so $\ltwomo{u_h}=1$. We extend $v_n$ by $0$ for $ z\in M$ outside of $\mathcal{U}$.  This extension is smooth because $\p \mathcal{U}$ is exactly $\ed{\Psi^{-1}(M_1 \times \p B(0,1))}$.

Now compute, using that $W(\Psi^{-1}(p,x)) \leq C V(|x|)$ and the isometry $\Psi$
\begin{align}
\ltwom{\left(-\Delta_g+\frac{i}{h} W - \frac{1}{h^2}\right) v_n} = \nm{\left(-\Delta_g+\frac{i}{h} W - \frac{1}{h^2}\right) v_n}_{L^2(\mathcal{U})} \\
=\ltwomo{\left(-\Delta_{g_1\otimes g_2}+\frac{i}{h}W(\Psi^{-1}(p,x)) - \frac{1}{h^2}\right) u_n}\\
=  \frac{c_1}{\rho(h)^{\frac{d-k}{2}}}\ltwomo{\left(-\Delta_{g_1} - \Delta_{g_2} + \frac{i}{h}W(\Psi^{-1}(p,x)) - \frac{1}{h^2}  \right)  \chi\left(\frac{x}{\rho(h)} \right) e_n(p) }\\
 \leq \frac{C}{\rho(h)^{\frac{d-k}{2}}} \left( \int_{B(0,1)} \frac{1}{\rho(h)^4} \left|(\Delta_{g_2} \chi)\left( \frac{x}{\rho(h)} \right)\right|^2 + \frac{1}{h^2} V(\ed{|x|})^2 \left|\chi \left( \frac{x}{\rho(h)} \right)\right|^2 dx \right)^{\frac{1}{2}}.
\end{align}
Make a change of variables $\rho(h)y=x$ and obtain
\begin{align}
 \ltwom{\left(-\Delta_g+\frac{i}{h} W - \frac{1}{h^2}\right) v_n} & \leq C \left( \int_{B(0,1/\rho(h))} \frac{1}{\rho(h)^4} |\Delta_{g_2} \chi(y)|^2 + \frac{V(\ed{\rho(h)|y|})^2}{h^2} |\chi(y)|^2 dy \right)^{\frac{1}{2}} \\
 &\leq C \left( \frac{1}{\rho(h)^2} + \frac{V(\rho(h))}{h}\right),
\end{align}
where we used $\supp \chi, \Delta \chi \subset B(0,1)$ and that $V$ is increasing. Now since $R(z)=z^2 V(z)$ and $\rho(h) = \ti{R}^{-1}(h)$  we have $h \simeq \rho(h)^2 V(\rho(h))$ and so
\begin{equation}\label{whyrdef}
\ed{\rho(h)^{-2} \simeq \frac{V(\rho(h))}{h}.}
\end{equation}
Therefore
\begin{equation}
\ltwom{\left(-\Delta+\frac{i}{h} W - \frac{1}{h^2} \right) v_n }\leq C\frac{V(\ti{R}^{-1}(h))}{h},
\end{equation}
as desired. 
\end{proof}
\appendix
\section{Basic Inequalities and Estimates}\label{appendix}

First a lemma giving envelope inverses for products of powers with logarithms or exponentials. 
\begin{lemma}\label{envinvlemma}
For any $\zeta, \eta, \theta, c \in \Rb$
\begin{enumerate}
	\item If $m(\lambda) = \lambda^{\eta} \ln(\lambda)^{-\theta}$,  then $\ti{m}^{-1}(t) = t^{\frac{1}{\eta}}\ln(\ed{2}+t)^{\frac{\theta}{\eta}}$, for $t$ large enough.

	\item If $R(z) = z^{\eta} \ln (z^{-1})^{\theta}$,  then  $\ti{R}^{-1}(h) = h^{\frac{1}{\eta}}\ln(h^{-1})^{-\frac{\theta}{\eta}}$, for $h$ small enough.

	\item If $R(z) = z^{\eta} \exp(-cz^{-\zeta})$,  then $\ti{R}^{-1}(h) = \ln\left(h^{-\frac{1}{c}} \ln(h^{-\frac{1}{c}})^{-\frac{\eta}{\zeta c}} \right)^{-\frac{1}{\zeta}} $, for $h$ small enough. 
\end{enumerate}
\end{lemma}
\begin{proof}
1) For large $t$
$$
m(\ti{m}^{-1}(t)) = t \ln(\ed{2}+t)^{\theta} \ln\left(t^{1/\eta} \ln(\ed{2}+t)^{\theta/\eta}\right)^{-\theta} \simeq \ed{t}.
$$

2) For small $h$
$$
R(\ti{R}^{-1}(h)) =h\ln(h^{-1})^{-\theta} \ln\left(h^{-1/\eta} \ln(h)^{\theta/\eta}\right)^{\theta} \simeq h.
$$

3) For small $h$
$$
R(\ti{R}^{-1}(h)) =\ln\left(h^{-\frac{1}{c}} \ln(h^{-\frac{1}{c}})^{-\frac{\eta}{\zeta c}} \right)^{-\frac{\eta}{\zeta}} h \ln(h^{-\frac{1}{c}})^{\frac{\eta}{\zeta}} \simeq h.
$$
\end{proof}

We now prove some lemmas that are used in the proofs of Lemma \ref{averagingsquarelemma}, \ref{avgconveximproveprop} and \ref{superellipseaveraging}.

First, we prove a lemma stating that the distance between the curve $x=y^n$ and a point $(x_0,y_0)$ to the right of it, is approximated by the distance between  $(x_0,y_0)$ and $(y_0^n, y_0)$.
\begin{lemma}\label{distancelemma}
Fix $n>1$, there exists $C,\e>0$ such that if $(x_0,y_0) \in \Rb^2$ with $0<x_0<\e$ and satisfies $x_0 \geq y_0^n$, then
\begin{equation}
C^{-1} |y_0^{n}-x_0| \leq \dist( (x_0,y_0), \{(x,y) \in \Rb^2: x=|y|^n\} ) \leq C |y_0^{n}-x_0|.
\end{equation}
\end{lemma}
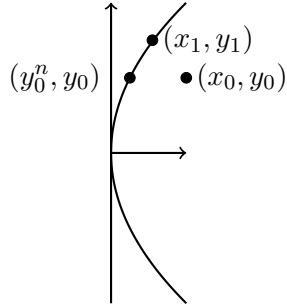
\begin{figure}[h]
\centering
\begin{tikzpicture}
[
declare function={
  a=1;
}]
\draw[domain=-a:a,samples=100,variable=\y,thick] plot ({\y*\y},{2*\y});
\draw[->, thick] (0,-2*a)--(0,2*a);
\draw[->, thick] (0,0)--(a*a,0);
\filldraw[black] (1,1) circle (2pt) node[anchor=west]{$(x_0,y_0)$};
\filldraw[black] (.25,1) circle (2pt);
\node[anchor=east] at (0,1) {$(y_0^n,y_0)$}; 
\filldraw[black] (.55,1.5) circle (2pt) node[anchor=west]{$(x_1,y_1)$};
\end{tikzpicture}
\caption{The curve $x^n=y$.}
\end{figure}
\begin{proof}
By symmetry we can consider $y_0 \geq 0$. Let $(x_1, y_1)$ be the point on $x=|y|^n$ closest to $(x_0, y_0)$. By inspection we can see $0 \leq x_1 \leq x_0$ and $y_1\geq 0$. The  line \ed{perpendicular} to the curve at $(x_1,y_1)$ must also pass through $(x_0, y_0)$ and has slope $-n y_1^{n-1}$. So 
\begin{equation}
\frac{y_1-y_0}{x_1-x_0} = - n y_1^{n-1}.
\end{equation}
Since $x_1\leq x_0<\e$, $y_1=x_1^{1/n} <\e^{1/n}$. Thus taking $\e<n^{-\frac{n}{n-1}}$
\begin{equation}
-1< - n \e^{\frac{n-1}{n}} < - n y_1^{n-1} =\frac{y_1-y_0}{x_1-x_0}.
\end{equation}
Since $x_1-x_0<0$, rearranging gives $y_1-y_0 < x_0-x_1$. Therefore 
\begin{equation}
|x_1-x_0| \leq d( (x_0,y_0), (x_1,y_1)) = \sqrt{(x_1-x_0)^2+(y_1-y_0)^2} \leq |x_1-x_0| \sqrt{2}.
\end{equation}

Now note that $|x_0-x_1| \leq |y_0^{n}-x_0|$, otherwise $(y_0^{n}, y_0)$ would be a point on the curve closer to $(x_0,y_0)$ than $(x_1, y_1)$. Again using that the slope of the line between $(x_0,y_0)$ and $(x_1, y_1)$ has magnitude at most $|n \e^{\frac{n-1}{n}}|$
\begin{equation}
|y_1-y_0| \leq n \e^{\frac{n-1}{n}} |y_0^n -x_0|.
\end{equation}
Since $(x_1, y_1)$ is on the curve, \ed{and by a Taylor series for $\e<0$, we have} $x_1=y_1^{n} \simeq y_0^{n} + c \e^{\frac{n-1}{n}} |y_0^n -x_0|$ and so
\begin{equation}
|x_1-x_0| \simeq |y_0^n-x_0| + c \e^{\frac{n-1}{n}} |y_0^n -x_0| \simeq |y_0^n -x_0|, 
\end{equation}
for $\e>0$ small enough.
\end{proof}

Next we prove a basic inequality related to the difference of powers formula. 
\begin{lemma}\label{diffpowerslemma}
Suppose $x,y >0$ and $n \in [1,\infty)$. If $y \in [0,x^{1/n}]$, then
\begin{equation}
(x-x^{\frac{n-1}{n}}y) \leq x-y^n \leq n(x-x^{\frac{n-1}{n}}y).
\end{equation}
\end{lemma}
\begin{proof}
When $n \in \Nb$ this follows immediately from the difference of powers formula and that $y \in [0,x^{1/n}]$
\begin{equation}
x-y^n = (x^{\frac{1}{n}}-y)\left(x^{\frac{n-1}{n}} + x^{\frac{n-2}{n}} y + \cdots + x^{\frac{1}{n}} y^{n-2} + y^{n-1} \right) \simeq (x^{\frac{1}{n}} - y) x^{\frac{n-1}{n}}.
\end{equation}
To see this for general $n \in [1,\infty)$, note that $x-y^n=n(x-x^{\frac{n-1}{n}}y)=0$ at $y=x^{1/n}$. Furthermore, on $y \in [0,x^{1/n}]$
\begin{equation}
\frac{d}{dy} (x-y^n) = - ny^{n-1} \geq -n x^{\frac{n-1}{n}} = \frac{d}{dy} n(x-x^{\frac{n-1}{n}} y).
\end{equation}
So indeed $x-y^n \leq n(x-x^{\frac{n-1}{n} } y)$ for $y \in [0,x^{\frac{1}{n}}]$. On the other hand, \ed{by the mean value theorem applied to $f(z)=z^n$, there exists $\zeta \in (y,x^{1/n})$ such that 
	\begin{equation}
		\frac{x-y^n}{x^{1/n}-y} = n \zeta^{n-1}.
	\end{equation}
Rearranging this gives 
\begin{equation}
	x-y^n = n \zeta^{n-1} (x^{1/n} -y ) \leq n x^{\frac{n-1}{n}}(x^{1/n}-y).
\end{equation}}
\end{proof}

The following lemma is used to obtain an improvement to the polynomial growth of damping after averaging.
\begin{lemma}\label{uppergammalem}
\begin{enumerate}
\item Fix $\beta>0, \gamma \in \Rb$, for $x$ small enough 
\begin{equation}
\int_0^x u^{\beta} \ln(u^{-1})^{-\gamma} du \simeq x^{\beta+1} \ln(x^{-1})^{-\gamma}.
\end{equation}
\item Fix $\alpha, c>0, \beta \in \Rb$, for $x$ small enough 
\begin{equation}
\int_0^{cx} u^{\beta} \exp(-u^{-\alpha}) du \simeq x^{\beta+\alpha+1} \exp(-(cx)^{-\alpha}).
\end{equation}
\end{enumerate}
\end{lemma}
\begin{proof}
1) Make the substitution $v=\ln(u^{-1})$
\begin{equation}
\int_0^x u^{\beta} \ln(u^{-1})^{-\gamma} du =   \int_{\ln(x^{-1})}^{\infty} e^{-v(\beta+1)} v^{-\gamma} dv. 
\end{equation}
Writing $(\beta+1)v=t$, we have the above
\begin{equation}
= C \int_{(\beta+1)\ln(x^{-1})}^{\infty} e^{-t} t^{-\gamma} dt = C \Gamma(-\gamma+1, (\beta+1)\ln(x^{-1})),
\end{equation}	
where $\Gamma(a,z)$ is the upper incomplete Gamma function as defined in \cite[\href{https://dlmf.nist.gov/8.2}{(8.2)}]{DLMF}. By \cite[\href{https://dlmf.nist.gov/8.11}{(8.11.i)}]{DLMF} 
\begin{equation}\label{uppergamma}
\Gamma(a,z) \simeq z^{a-1} e^{-z}(1+O(z^{-1})) \text{ as } z \ra \infty, \quad a \text{ fixed}.
\end{equation}
So for $x$ near $0$
\begin{equation}
\int_0^x u^{\beta} \ln(u^{-1})^{-\gamma} du  \simeq \ln(x^{-1})^{-\gamma} e^{-(\beta+1) \ln(x^{-1})} (1+\ln(x^{-1})^{-1} ) \simeq x^{\beta+1} \ln(x^{-1})^{-\gamma}.
\end{equation}

2) Make the substitution $u^{-\alpha}=t$
\begin{equation}
\int_0^{cx} u^{\beta} \exp(-u^{-\alpha}) \ed{du} =  C \int_{(c x)^{-\alpha}}^{\infty} t^{-\frac{1}{\alpha}-1} t^{-\beta/\alpha} \exp(-t) dt=C \Gamma\left( \frac{-1-\beta}{\alpha}, (cx)^{-\alpha}\right).
\end{equation}
Once again by \eqref{uppergamma} for $x$ near 0
\begin{equation}
\int_0^{Cx} u^{\beta} \exp(-u^{-\alpha}) \simeq (x^{-\alpha})^{\frac{-1-\beta}{\alpha}-1} \exp(-(cx)^{-\alpha}) (1+x^{\alpha}) \simeq x^{1+\beta+\alpha} \exp(-(c x)^{-\alpha}).
\end{equation}
\end{proof}

Finally, we compute elementary derivatives to check that  the concavity portion of Assumptions \ref{noninvarassumption} and \ref{widegenbadassumption} are satisfied for the examples we work with. This concavity is needed in order to apply Lemma \ref{averagingdbclemma}
\begin{lemma}\label{logconcavelemma}
If $q(z) = \ln(z^{-1})^{-\gamma}$ for $\gamma>0$, then there exists $\e>0$ such that $R(z)=\frac{z}{q(z)}$ is concave on $[0,\e]$.
\end{lemma}
Since $\gamma>0$ is arbitrary, this shows that $z \ln(z^{-1})^{\left(\frac{\alpha+1}{\alpha}\right)}$ and $z \ln(z^{-1})^{2\left(\frac{\alpha+1}{\alpha}\right)}$ are both concave on some interval $(0,\e)$. 
\begin{proof}
Since 
$$
R''(z) = \frac{-q'' q^2 z - 2 q^2 q' + 2(q')^2 z q}{q^4},
$$
it is enough to see that the numerator is negative. Well $q'(z) = \frac{\gamma}{z} \ln(z^{-1})^{-\gamma-1}$ and $q''(z) = z^{-2} \ln(z^{-1})^{-\gamma-1} \left(-\gamma+(\gamma+1)\gamma \ln(z^{-1})^{-1} \right).$ Therefore
\begin{align*}
-q'' qz -2 q q' + 2(q')^2 z = \frac{\gamma}{z} \ln(z^{-1})^{-2\gamma-1}(-1+(\gamma-1) \ln(z^{-1})^{-1}).
\end{align*}
For $z$ small enough $(\gamma-1) \ln(z^{-1}) < \frac{1}{2}$, so indeed $R''(z)<0$. 
\end{proof}

\bibliographystyle{alpha}
\bibliography{mybib}

\end{document}